\documentclass[a4paper, 11pt]{amsart}
\usepackage{graphicx}
\usepackage{amssymb}
\usepackage{epstopdf}
\usepackage{fancyhdr}
\usepackage{amsmath}
\usepackage{amsthm}
\usepackage{color}

\usepackage[top=2.1 cm, bottom=1.9 cm,left=1.9 cm, right=1.9 cm]{geometry}
\def\AMS{{\it AMS subject classifications: }}

\definecolor{purple}{rgb}{.9,0,.9}

\newcommand{\Real}{\mathbb{R}}

\newcommand{\rfa}{\qquad {\rm for \ all}\ }

\newcommand{\cA}{{\mathcal A}}
\newcommand{\cE}{{\mathcal E}}
\newcommand{\cI}{{\mathcal I}}

\newcommand{\cN}{{\mathcal N}}
\newcommand{\cR}{{\mathcal R}}
\newcommand{\cT}{{\mathcal T}}\newcommand{\cU}{{\mathcal U}}
\newcommand{\cV}{{\mathcal V}}\newcommand{\cW}{{\mathcal W}}

\newcommand{\bb}{{\bf b}}
\newcommand{\bd}{{\bf d}}\newcommand{\be}{{\bf e}}\newcommand{\bff}{{\bf f}}

\newcommand{\bj}{{\bf j}}

\newcommand{\br}{{\bf r}}
\newcommand{\bu}{{\bf u}}
\newcommand{\bw}{{\bf w}}
\newcommand{\bA}{{\bf A}}
\newcommand{\bC}{{\bf C}}\newcommand{\bD}{{\bf D}}

\newcommand{\bM}{{\bf M}}

\newcommand{\bQ}{{\bf Q}}
\newcommand{\bT}{{\bf T}}\newcommand{\bU}{{\bf U}}
\newcommand{\bW}{{\bf W}}
\newcommand{\bZ}{{\bf Z}}

\newcommand{\p}{\textbf{p}}
\newcommand{\n}{\textbf{n}}
\newcommand{\F}{\textbf{F}}
\newcommand{\R}{\textbf{R}}
\newcommand{\G}{\textbf{G}}
\newcommand{\Q}{\textbf{Q}}
\newcommand{\J}{\textbf{J}}

\newtheorem{theorem}{Theorem}[section]
\newtheorem{lemma}[theorem]{Lemma}

\newtheorem{definition}[theorem]{Definition}

\newtheorem{assumption}{Assumption}

\newcommand{\beqn}{\begin{equation}}
\newcommand{\eeqn}{\end{equation}}

\newcommand{\half}{\frac{1}{2}}
\newcommand{\bone}{\textbf{1}}

\newcommand{\norm}[1]{\left\|#1\right \|}

\newcommand{\bbE}{\mathbb{E}}

\newcommand{\bnu}{\boldsymbol{\nu}}
\newcommand{\ve}{\varepsilon}

\def\Xint#1{\mathchoice
{\XXint\displaystyle\textstyle{#1}}%
{\XXint\textstyle\scriptstyle{#1}}%
{\XXint\scriptstyle\scriptscriptstyle{#1}}%
{\XXint\scriptscriptstyle\scriptscriptstyle{#1}}%
\!\int}
\def\XXint#1#2#3{{\setbox0=\hbox{$#1{#2#3}{\int}$ }
\vcenter{\hbox{$#2#3$ }}\kern-.58\wd0}}

\def\dashint{\Xint-}

\title[Homogenization of a system of elastic and reaction-diffusion equations]{Homogenization of a system of elastic and reaction-diffusion equations modelling plant cell wall biomechanics}
\author{Mariya Ptashnyk and Brian Seguin}

 \thanks{\noindent Department of Mathematics, University of Dundee
Dundee, DD1 4HN, UK, 
 m.ptashnyk@dundee.ac.uk, mptashnyk@maths.dundee.ac.uk, bseguin@maths.dundee.ac.uk  \\
M.\ Ptashnyk and B.\ Seguin gratefully acknowledge the support of the EPSRC First Grant EP/K036521/1 ``Multiscale modelling and analysis of mechanical properties of plant cells and tissues''.}

\begin{document}

\maketitle

%
%
%
\begin{abstract}
In this paper we present  a derivation and  multiscale analysis of  a mathematical model   for plant cell wall biomechanics that takes into account both the microscopic structure of a cell wall  coming from the cellulose microfibrils and the chemical reactions between the cell wall's constituents.  Particular attention is paid to the role of pectin and the impact of  calcium-pectin cross-linking chemistry on the mechanical properties of the cell wall.  We prove the existence and uniqueness of the strongly coupled microscopic problem consisting of the equations of linear elasticity  and a system of reaction-diffusion and ordinary differential equations.  Using homogenization  techniques  (two-scale convergence and periodic unfolding methods) we derive  a macroscopic model for plant cell wall biomechanics.  
\end{abstract}
%
%
{\small \noindent {\it Key words: } {Homogenization; two-scale convergence; periodic unfolding method;  elasticity; reaction-diffusion equations;  plant modelling.}

\noindent \AMS{35B27, 35Q92, 35Kxx, 74Qxx, 74A40,  74D05}
}
%
%
\section*{Introduction}

For  a better understanding of plant growth and development  it is important to  analyse the influence of chemical processes  on the mechanical properties (elasticity and extensibility)  of plant cells. The main feature of plant cells are their walls, which must be strong to resist a high internal hydrostatic pressure (turgor pressure) and flexible to permit growth. Plant cell walls consist  of a wall  matrix  (composed mainly of  pectin, hemicellulose, structural proteins, and water)  and cellulose microfibrils. It is supposed that calcium-pectin cross-linking chemistry is one of the main regulators of plant cell wall elasticity and extension \cite{WHH}.
Pectin is deposited into cell walls in a methylesterified form. In cell walls  pectin can be modified by the enzyme pectin methylesterase (PME), which removes methyl groups by breaking ester bonds. The de-esterified pectin is able to form calcium-pectin cross-links, and so stiffen the cell wall and reduce its expansion. On the other hand, mechanical stresses can break calcium-pectin cross-links and hence increase the extensibility of plant cell walls.
  
  To analyse the interactions between calcium-pectin dynamics and the deformations of a plant cell wall, as well as the influence of the microscopic structure on the mechanical properties of a cell wall,  we derive a mathematical model for plant cell wall biomechanics at the length scale of cell wall microfibrils. We model the cell wall as a three-dimensional continuum consisting of a wall matrix and microfibrils.  Within the wall matrix, we consider the dynamics of five chemical substances:  the enzyme PME, methylesterfied pectin, demethylesterfied pectin, calcium ions, and calcium-pectin cross-links.  The cell wall matrix is assumed to be isotropic and   linearly elastic, whereas microfibrils are modelled as an anisotropic,  linearly elastic material.  The interplay between the mechanics and the cross-link dynamics comes in by assuming that the elastic  properties of the matrix depend on the density of the cross-links and that strain or stress within the cell wall can break calcium-pectin cross-links.   The strain- or stress-dependent opening of calcium channels in the cell plasma membrane is addressed in the flux boundary conditions for calcium ions.   We consider  two different cases, one in which the calcium-pectin cross-links diffuse and another in which they do not diffuse.
Thus the microscopic problem is a strongly coupled system of  reaction-diffusion equations or reaction-diffusion  and  ordinary differential equations, with reaction terms depending on the displacement gradient, and   the equations of linear elasticity, with elastic moduli depending on the density of calcium-pectin cross-links.

To analyse the macroscopic behaviour of  plant cell walls, comprising  a complex microscopic structure, we  rigorously derive a macroscopic  model for plant cell wall biomechanics.    As there are thousands of microfibrils in a plant cell wall,   the derivation of the macroscopic equations is also important for   effective  numerical simulations.  The two-scale convergence, e.g.~\cite{allaire,Nguetseng}, and the periodic unfolding  method, e.g.~\cite{CDG,CDDGZ}, are applied to obtain the macroscopic equations.          
Some previous results on the homogenization of problems in linear elasticity can be found  in \cite{Allaire,BLP,JKO,OShY,SP} (and the references therein).  A multiscale analysis of  microscopic problems comprising the equations of linear elasticity  for a solid matrix or cells combined with the Stokes equations for the fluid part  was considered in \cite{GM,JMN,MW}.

The main novelty of this paper is twofold: (i) we derive a new  model for plant cell wall biomechanics  where 
the mechanical properties and biochemical processes in a cell wall are considered on the scale of 
its structural elements (on the scale of  the microfibrils) and (ii)  using homogenization techniques we   obtain  a macroscopic model for plant cell wall biomechanics from a microscopic description of the mechanical and chemical processes. 
This approach allows us to take into account the complex microscopic structure of a plant cell wall and  to analyze the impact of the heterogeneous distribution of cell wall's structural elements on the mechanical properties and development of plants. 

The main mathematical difficulty arises from the strong coupling between the equations of linear  elasticity   for cell wall mechanics  and the system of reaction-diffusion and ordinary differential  equations for the chemical processes in the wall matrix.  The Galerkin method together with classical fixed-point  approaches  are used to prove  the existence of a unique solution of the microscopic problem. However, since the reaction terms depend on the displacement gradient  and  the elasticity tensor is a function of the density of chemical substances,  the derivation of a contraction inequality is non-standard and relies on estimates for the $L^\infty$-norm of the solutions of the reaction-diffusion and ordinary differential equations in term of the $L^2$-norm of the displacement gradient.
The theory of positively invariant regions \cite{Redlinger,Smoller} and the Alikakos \cite{Alikakos} iteration techniques are applied to show the non-negativity and  uniform boundedness of solutions of the microscopic model.   The  iteration technique \cite{Alikakos} is also used  to derive a contraction inequality. 

 The analysis of the coupled system also depends strongly  on the microscopic model for the chemical processes. For the chemical processes  in the cell wall matrix we consider two situations: (i)  chemical processes are described  by a  system of reaction-diffusion and ordinary differential  equations and   (ii)   all chemical processes are modelled by reaction-diffusion equations.
 
  In the first situation, the  solutions of  the ordinary differential equation have the same regularity with respect to the  spatial variables as  the reaction terms. Thus, for the proof of the well-posedness results for the microscopic problem and for  the  rigorous derivation of the macroscopic equations,  the dependence of the reaction terms  on a local average of   the displacement   gradient is essential. The well-posedness  of the microscopic problem  can be proven by considering  an $\ve$-average, where  the small parameter $\ve$ characterizes the microscopic structure of the cell wall. However,  the proof of the strong convergence for a sequence of solutions of  the ordinary differential equation, which is necessary for the homogenization of the microscopic problem,  relies on the fact that the  local average is  independent of $\ve$.  
 Also,   in  the proof of the strong convergence we apply  the unfolding operator  to map  solutions of  the ordinary differential equation,   defined in a perforated $\ve$-dependent  domain,  to a fixed domain. 

In the second case when all  chemical substances  diffuse,    solutions of the reaction-diffusion equations  have higher regularity with respect to the spatial variables and a point-wise dependence of the reaction terms on the displacement  gradient can be considered.   In this situation in order to pass to the limit in the nonlinear reaction terms we  prove the strong two-scale convergence for the displacement  gradient. 

  Similar to the microscopic problems,  the uniqueness of a solution of the macroscopic equations  is proven by  deriving a contraction inequality involving the $L^\infty$-norm of the difference of two solutions of the reaction-diffusion and ordinary differential equations.

The paper is organised as follows.  In Section~\ref{section:model} we give the general setting of the two microscopic models for plant cell wall biomechanics. 
The main results of the paper are summarized in Section~\ref{main_results}.    The existence and uniqueness  results for     weak solutions of  the two  microscopic problems are proven in Sections~\ref{exist_uniq_ModelI}~and~\ref{exist_uniq_ModelII}.  
  The homogenization and derivation of the macroscopic equations for  both microscopic models   for plant cell wall biomechanics  are conducted in Sections~\ref{secthom} and \ref{macro_model_II}.  Some  results on the numerical simulations of the unit cell problems, which determine  the effective macroscopic  elastic properties of  a plant cell wall, are given in Section~\ref{numerics}. 
The  detailed derivation of the microscopic model  for plant cell wall biomechanics on the length-scale of the cell wall microfibrils is presented in Section~\ref{sectdermodel}.    Concluding remarks are included in  Section~\ref{conclusions}.

\section{Formulation of the mathematical  models for plant cell wall biomechanics.}\label{section:model}

 In the mathematical model for plant cell wall biomechanics we consider   interactions between the mechanical properties of the plant cell wall   and  the chemical processes in  the cell wall.   The  derivation of the models is presented in Section~\ref{sectdermodel}.    
  
  In the mathematical model we consider  the  microscopic structure of a plant cell wall, which  is given by   microfibrils embedded in the cell wall matrix. 
By $\Omega\subset \mathbb R^3$ we denote a domain occupied by   a flat section of a plant cell wall and can consider   $\Omega=(0,a_1)\times(0,a_2)\times(0,a_3)$, where $a_i$, $i=1,2,3$, are positive numbers. We  assume that  the microfibrils are oriented in the $x_3$-direction, see Fig.~\ref{FigDomain1}(a).
The part of $\partial\Omega$ on the exterior of the cell wall   is given by
$
\Gamma_{\cE}= \{a_1\}\times(0, a_2)\times(0,a_3),
$
and the interior boundary $\Gamma_\cI$ of the cell wall is given by
$
\Gamma_{\cI}=\{0\}\times(0, a_2)\times (0,a_3).
$
The top and bottom boundaries are defined by $\Gamma_{\mathcal U} = (0, a_1)\times \{ 0\} \times (0, a_3) \cup (0, a_1)\times \{ a_2\} \times (0, a_3)$. 

To determine the  microscopic structure of the cell wall, we consider    $\hat Y = (0,1)^2$ and an open subset $\hat Y_F$, with    $\overline {\hat Y_F} \subset \hat Y$,  and define  $\hat Y_M= \hat Y \setminus \overline {\hat Y_F}$, $Y= \hat Y\times (0,a_3)$,  $Y_F = \hat  Y_F \times (0, a_3)$, and $Y_M = Y \setminus \overline{Y_F}$,  where $Y_M$ and $Y_F$ represent the  cell wall matrix and a microfibril, respectivelly,  see Fig.~\ref{FigDomain1}(b).  We also denote  $\Gamma= \partial Y_F$ and $\hat \Gamma = \partial \hat Y_F$.

We assume that the  microfibrils in the cell wall  are distributed periodically  and have a diameter on the order of $\ve$, where  the small parameter $\ve$  characterizes the size of the microstructure (the ratio of the diameter  of  microfibrils   to the thickness of the cell wall, i.e.\  the  microfibrils of a plant cell wall are about $3$ nm in diameter and are separated by a distance of about $6$~nm, see e.g.~\cite{Colvin,KSJW,Thomas}, whereas  the thickness of a plant cell wall is of the order of a few micrometers). 
The domains 
\begin{equation*}
\Omega_F^\ve =\bigcup_{\xi \in  \mathbb Z^2} \big\{ \ve (\hat Y_F+ \xi) \times(0, a_3)\;\; | \;\;   \ve (\hat Y+ \xi)\subset (0,a_1)\times(0,a_2) \big\} \qquad  \text{ and } \qquad   \Omega_M^\ve =\Omega \, \setminus \overline {\Omega_F^\ve}
\end{equation*}
denote the part of $\Omega$ occupied by the microfibrils and  by the cell wall matrix, respectively.  The boundary between the matrix and the microfibrils is denoted by
\begin{equation*}
\Gamma^\ve =\partial\Omega_M^\ve\cap\partial\Omega_F^\ve.
\end{equation*} 

\begin{figure}[t]
\includegraphics[width=4.5in]{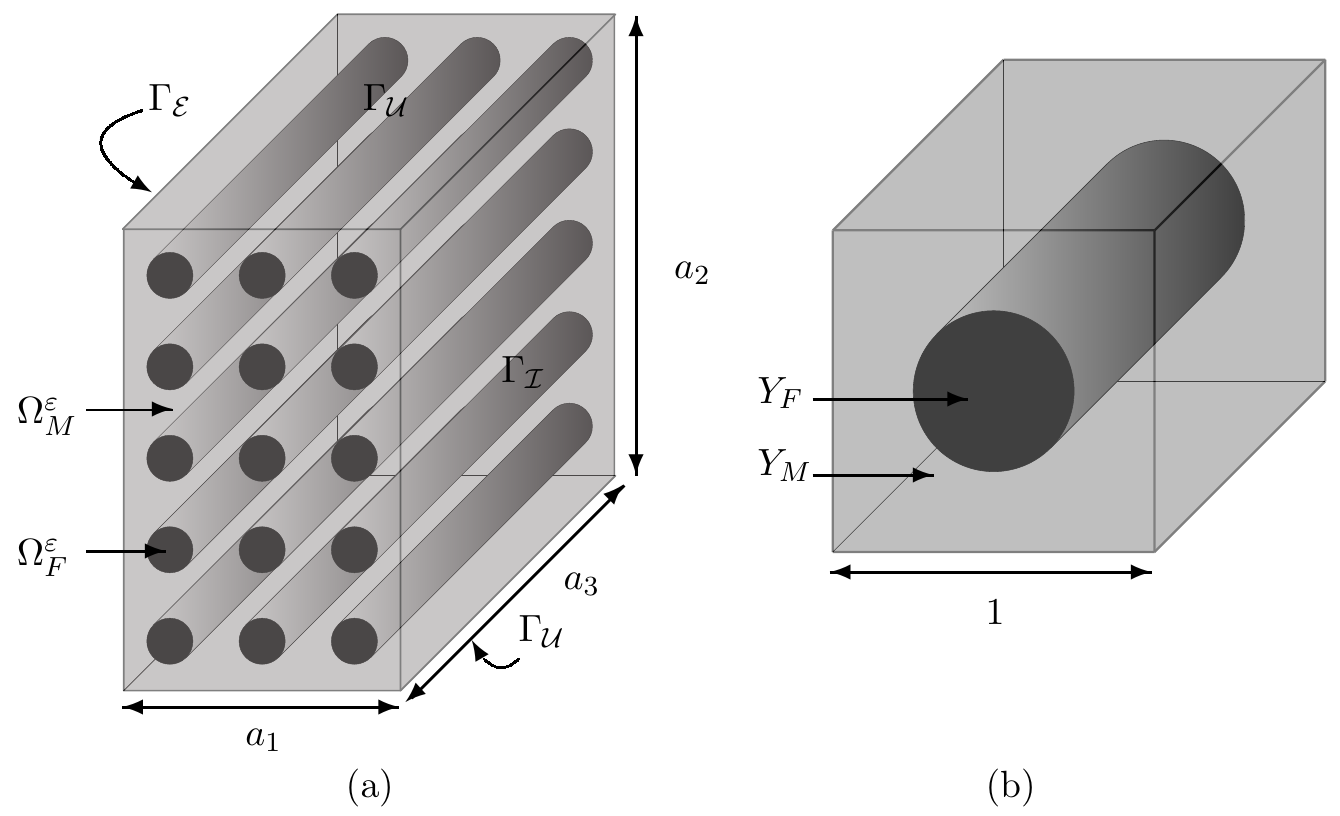}
\caption{(a) A depiction of the domain $\Omega$ with the subsets representing the cell wall matrix $\Omega_M^\ve$ and the microfibrils $\Omega_F^\ve$.  The surface $\Gamma_\cI$ is in contact with the interior of the cell, and the (hidden) surface $\Gamma_\cE$ is facing the outside of the cell, and $\Gamma_\cU$ is the union of the surfaces on the top and bottom of $\Omega$.  (b) A depiction of the unit cell $Y$.}\label{FigDomain1}
\end{figure}

In the mathematical model of plant cell wall biomechanics  we consider   deformations of the cell wall and the  interactions between   five species within the  plant  cell  wall  matrix:  
the number densities   of  methylestrified  pectin   $\p_1^\ve$,   of  enzyme PME  $\p_2^\ve$,  of   demethylestrified pectin $\n_1^\ve$,  of calcium ions $\n_2^\ve$, and  of  calcium-pectin cross-links $b^\ve$.  
   We shall consider two situations: (i) it is assumed that the  calcium-pectin cross-links   $b^\ve$ do not diffuse and   (ii) the diffusion  of the calcium-pectin cross-links  $b^\ve$ in the cell wall matrix is considered. \\
 
 {\bf Model I: } In the first case,  where the calcium-pectin cross-links  $b^\ve$ do not diffuse, the microscopic problem  is composed of  a system of reaction-diffusion and ordinary differential equations for  $\p^\ve=(\p^\ve_1, \p^\ve_2)^T$, $\n^\ve=(\n^\ve_1, \n^\ve_2)^T$, and $b^\ve$, coupled with the equations of linear elasticity for the displacement  $\bu^\ve$ 
\begin{eqnarray}\label{sumbal}
&&\partial_t {\bf p}^\ve =\text{div}(D_p \nabla \p^\ve) -  \F_p(\p^\ve)  
\hspace{ 5.9 cm } \text{in}\ (0,T)\times\Omega^\ve_M, 
\\
&&\begin{aligned}\label{sumbal_11}
&\partial_t \n^\ve =\text{div}(D_n\nabla \n^\ve) +\F_n(\p^\ve, \n^\ve)
+  \R_n (\n^\ve, b^\ve,  \mathcal N_\delta(\be(\bu^\ve)))  \quad   \\
& \partial_t b^\ve = \phantom{ \text{div}(D_c\nabla n^\ve)  ++\F_n(\p^\ve, \n^\ve) \quad }  R_{b}(\n^\ve, b^\ve, \mathcal N_\delta(\be(\bu^\ve)) ) 
\end{aligned}\qquad   \text{in}\ (0,T)\times\Omega^\ve_M,
\end{eqnarray}
 with     the boundary and initial  conditions 
\beqn\label{BC1}
\begin{aligned}
& \begin{cases}
D_{p}\nabla \p^\ve\, \bnu =  \J_p(\p^\ve)\\ 
D_{n}\nabla \n^\ve\, \bnu = \G(\n^\ve) \, \mathcal N_\delta(\be(\bu^\ve)) 
\end{cases}
\qquad & \text{on}\ (0,T)\times\Gamma_\cI,
\\
& \begin{cases}
D_p\nabla \p^\ve \, \bnu =- \gamma_p\,  \p^\ve  \qquad \\
 D_n\nabla \n^\ve \, \bnu = \J_n(\n^\ve)
 \end{cases}
& \text{on}\ (0,T)\times\Gamma_\cE, \\
& 
\; D_p \nabla \p^\ve \,  \bnu = 0,  \; \; \quad  D_n \nabla \n^\ve \,  \bnu = 0 \qquad &\text{on}\ (0,T)\times(\Gamma_{\mathcal U}\cup \Gamma^\ve), \\ 
 &\; \p^\ve, \;  \n^\ve  & \quad \text{$a_3$-periodic in } x_3, 
 \\
&\p^\ve(0,x) = \p_{0}(x), \quad \n^\ve (0,x) = \n_{0}(x), \quad b^\ve (0,x) = b_{0}(x) \qquad &\text{ for } \quad x\in \Omega_M^\ve, 
\end{aligned}
\eeqn
where $\text{div}(D_p \nabla \p^\ve)=(\text{div}(D_p^1 \nabla \p^\ve_1),  \text{div}(D_p^2 \nabla \p^\ve_2))^T$ and  $\text{div}(D_n \nabla \n^\ve)=(\text{div}(D_n^1 \nabla \n^\ve_1),  \text{div}(D_n^2 \nabla \n^\ve_2))^T$. 

The displacement $\bu^\ve$ satisfies the equations of linear elasticity
\begin{equation}\label{sumbal2}
\left\{
\begin{aligned}
\text{div}(\mathbb E^\ve(b^\ve,x)\be(\bu^\ve))&=\textbf{0} \qquad \qquad &&\qquad \text{in}\ (0,T)\times\Omega, \\
(\mathbb E^\ve(b^\ve,x)\be(\bu^\ve))\, \bnu &=-p_\cI\, \bnu  &&\qquad \text{on}\  (0,T)\times\Gamma_\cI, \\
(\mathbb E^\ve(b^\ve,x)\be(\bu^\ve))\, \bnu &=\bff &&\qquad \text{on} \ (0,T)\times \big(\Gamma_\cE\cup \Gamma_{\mathcal U}\big),\\
\bu^\ve &&&\qquad \text{$a_3$-periodic in } x_3.
\end{aligned}
\right.
\end{equation}
The elasticity tensor is defined as  $\mathbb E^\ve(\xi, x)=  \bbE(\xi, \hat x/\ve)$,  where the  $\hat Y$-periodic in $y$  function $\mathbb E$ is given by
$
\bbE(\xi ,y)= \bbE_M(\xi) \chi_{\hat Y_M}(y) + 
\bbE_F \chi_{\hat Y_F}(y),
$
 with constant elastic properties of the microfibrils and the elastic properties of cell wall matrix depending on the density of calcium-pectin cross-links.

In \eqref{sumbal_11} and \eqref{BC1},   
${\cN}_\delta(\be(\bu^\ve))$ denotes the  positive part of a  local average of the trace of the elastic   stress 
\begin{equation*}
\mathcal N_\delta (\be(\bu^\ve))(t, x) =\Big(\dashint_{B_\delta(x)\cap \Omega}\text{tr } \bbE^\ve(b^\ve,\tilde x) \be(\bu^\ve)(t, \tilde x) \, d\tilde x \Big)^+\qquad \text{for all } \; x\in \overline\Omega \text{ and } t \in (0,T),
\end{equation*}
where $\delta>0$ is arbitrary fixed and  $w^{+} = \max\{ w, 0\}$. 
 
From a biological point of view the non-local dependence of the chemical reactions on  the displacement  gradient is motivated by the fact that pectins are very long molecules and hence cell wall mechanics has a nonlocal impact on the chemical processes.
The positive part in the definition of ${\cN}_\delta(\be(\bu^\ve))$   reflects the fact that  extension rather than compression causes the breakage of cross-links. 

In the boundary conditions \eqref{BC1} we assumed that the flow of calcium ions between the interior of the cell  and the cell wall depends on the displacement gradient, which corresponds to the stress-dependent opening of calcium channels in the plasma  membrane \cite{White}. 

The assumed dependence of  the reaction terms  on the local average  of the displacement  gradient  is also important  for the analysis of  Model I. The dependence on the local average of the displacement  gradient in the  ordinary differential equation for $b^\ve$ allow us to derive an estimate for  the $L^\infty$-norm of the  difference of two solutions $b^{\ve,1}$ and $b^{\ve,2}$ in terms of  $\|\be(\bu^{\ve,1}- \bu^{\ve,2})\|_{L^q(0,T; L^2(\Omega))}$ for some $q\geq 2$, which is important  for the proof of the well-posedness of the coupled system \eqref{sumbal}--\eqref{sumbal2}.  The fact that the local average $\mathcal N_\delta$ is independent of $\ve$  is  used in the proof   of the strong convergence of  $b^\ve$,  and hence  is important for the homogenization of the  microscopic problem \eqref{sumbal}--\eqref{sumbal2}. 
However, if we  assume  diffusion of  $b^\ve$, then  a point-wise dependence  on $\be(\bu^\ve)$ can be considered. From a biological point of view  the  situation where cross-links diffuse corresponds to a  less connected  network of calcium-pectin cross-links  and, hence, the mechanical stress in the cell wall will have a point-wise impact on the chemical processes. This motivates the consideration of the following model.  \\ 

{\bf Model II }   In the second case we consider  \eqref{sumbal} and \eqref{sumbal2}  together with  the  modified equations for $\n^\ve$ and $b^\ve$, which  include the  diffusion of $b^\ve$ and 
 reaction terms  depending on $ \bbE^\ve(b^\ve,x) \be(\bu^\ve)$ instead of its  local average
\begin{eqnarray}\label{sumbal_1}
\begin{cases}
\partial_t \n^\ve =\text{div}(D_n\nabla \n^\ve) +  \F_n(\p^\ve, \n^\ve)
 +  \Q_n(\n^\ve, b^\ve, \be(\bu^\ve)) \quad   \\
 \partial_t b^\ve = \text{div}(D_b\nabla b^\ve) \; \;  +    Q_b(\n^\ve, b^\ve, \be(\bu^\ve))   
\end{cases}\qquad   \text{in}\ (0,T)\times\Omega^\ve_M.
\end{eqnarray}
In addition to the boundary conditions in \eqref{BC1},  we define the boundary conditions for $b^\ve$:
 \beqn\label{BC4}
 \begin{aligned}
D_b\nabla b^\ve\cdot\bnu =0 & \qquad \text{on}\ (0,T)\times(\Gamma^\ve\cup \Gamma_{\cI}\cup \Gamma_{\mathcal U}),  && \quad
D_b\nabla b^\ve \cdot\bnu = - \gamma_b  \,  b^\ve && \quad \text{on}\ (0,T)\times  \Gamma_{\cE}, \\
b^\ve & \qquad \text{$a_3$-periodic in }  x_3.
\end{aligned}
\eeqn
As an example  we can consider $\Q_n(\n^\ve, b^\ve, \be(\bu^\ve))= \Q(\n^\ve, b^\ve) P(b^\ve,  \be(\bu^\ve))$, where $\Q:  \mathbb R^2\times \mathbb R\to \mathbb R^2$ is continuously differentiable and  $P:  \mathbb R\times \mathbb R^{3\times 3} \to \mathbb R$ is a positive continuous function given e.g.\ by
\begin{equation}\label{def_G_2}
P (b^\ve, \be(\bu^\ve)) =\left (\text{tr } \bbE^\ve(b^\ve, x) \be(\bu^\ve)\right)^{+}. 
\end{equation}
Notice that in both models, Model I and  Model II, the boundary conditions for $\n^\ve$ depend on $\mathcal N_\delta(\be(\bu^\ve))$, see \eqref{BC1}, because the elasticity equations do not provide enough regularity to consider  the trace of $\be(\bu^\ve)$ on the boundary $(0,T)\times \Gamma_{\cI}$. 

Next we shall analyse  the  two  microscopic  problems:  Model I  comprised of   equations \eqref{sumbal}--\eqref{sumbal2}
and Model~II   given by  equations \eqref{sumbal}, \eqref{BC1}--\eqref{BC4}.  The main difference in the analysis of the two  models  is related to   the regularity of $b^\ve$. 
If  we have an ordinary differential equation for $b^\ve$, then    $b^\ve$ has the same regularity   with respect to the spatial variables as the functions in the reaction terms.  
Whereas  in Model~II the  diffusion term in the equation  for $b^\ve$  in \eqref{sumbal_1} ensures higher  integrability and  spatial regularity of $b^\ve$. \\

We  adopt the following  notations for time-space domains:
$
\Omega_T= (0,T) \times \Omega$,
 $\Omega_{M,T}^\ve= (0,T) \times \Omega_M^\ve$, $\Gamma^\ve_T = (0,T)\times \Gamma^\ve$, 
$\Gamma_{\cI,T}= (0,T) \times \Gamma_\cI$, 
$\Gamma_{\cE,T}= (0,T) \times \Gamma_\cE$, $\Gamma_{\cU,T}= (0,T) \times \Gamma_\cU$,  $\Gamma_{\cE\cU,T}= (0,T) \times (\Gamma_{\cE} \cup \Gamma_\cU)$,  and define 
\begin{equation*}
\begin{aligned}
\cW(\Omega)&=\{\bu\in H^{1}(\Omega;\Real^3)\ \big |\; \ \int_\Omega\bu \, dx=\textbf{0},\; \;  \int_\Omega\big[(\nabla\bu)_{12}- (\nabla \bu)_{21}\big] \, dx=0, \; \text{and} \;  \bu\ \text{is $a_3$-periodic in  $x_3$}\},\\
\mathcal V(\Omega_{M}^\ve)&=\{v\in  H^1(\Omega_M^\ve) \,  \big | \; 
\; v\; \text{is $a_3$-periodic in   $x_3$}\},  \qquad 
\mathcal V(\Omega)=\{v\in  H^1(\Omega) \, \;  \big | \; 
\; v\; \text{is $a_3$-periodic in   $x_3$}\}.
\end{aligned}
\end{equation*}
By Korn's second inequality,  the $L^2$-norm of the strain 
\begin{equation*}
\norm{\bu}_{\cW(\Omega)}=\norm{\be(\bu)}_{L^2(\Omega)}\rfa \bu\in\cW(\Omega)
\end{equation*}
defines a norm on $\cW(\Omega)$, see e.g.\  \cite{CC, Korn, OShY}.   The Korn inequality holds  since $\cW(\Omega)\cap\cR(\Omega)=\{\textbf{0}\}$, see  e.g.\ \cite[Lemma~2.5]{OShY}, where the space of all rigid displacements of $\Omega$
\begin{align*}
\cR(\Omega)&=\{ \br\in H^1(\Omega;\Real^3)\ |\ \br(x)=\bd+\bW x\ \text{for}\ x\in\Omega,\ \bd\in\Real^3 \;  \text{and}\ \bW\ \text{is a skew matrix}\}
\end{align*}
is the kernel  of the  symmetric  gradient. To show that $\cW(\Omega)\cap\cR(\Omega)=\{\textbf{0}\}$, consider $\br\in\cW(\Omega)\cap\cR(\Omega)$ of the form $\br(x)=\bd+\bW x$  for all $x\in\Omega,$
where $x$ is viewed as a column vector.  It follows from the second condition in the definition of  $\cW(\Omega)$ that $\bW_{12}=0$.  Using the third condition,  we have $\br(0)=\br((0,0,a_3))$, which yields $\bW_{13}=\bW_{23}=0$.  Finally, since we now know that $\bW$ is zero, the first condition in $\cW(\Omega)$ implies that $\bd=\textbf{0}$, and hence  $\br=\textbf{0}$.

For a given  measurable set $\cA$ we use the notation $\langle \phi_1,\phi_2\rangle_{\cA}= \int_\cA \phi_1\phi_2\, dx,$
where the product of $\phi_1$ and $\phi_2$  is the scalar-product if they are vector valued.  

By 
$\langle \psi_1,\psi_2\rangle_{\mathcal V^\prime , \mathcal V}$ we denote the dual product between  $\psi_1 \in L^2(0,T; \mathcal V(\Omega_{M}^\ve)^\prime) $ and  $\psi_2 \in L^2(0,T; \mathcal V(\Omega_{M}^\ve))$ and  by 
$\langle \phi_1,\phi_2\rangle_{\mathcal V^\prime , \mathcal V(\Omega)}$ we denote the dual product between  $\phi_1 \in L^2(0,T; \mathcal V(\Omega)^\prime) $ and  $\phi_2 \in L^2(0,T; \mathcal V(\Omega))$.  

For some $\mu>0$ we define $\mathcal I^k_\mu= ( -\mu, + \infty)^k$, with $k \in \mathbb N$.

\begin{assumption}\label{assumptions}
\begin{itemize}
\item[1.]  $D_{\alpha}^j$ and $D_b$ are symmetric,  with   $(D_{\alpha}^j \boldsymbol{\xi}, \boldsymbol{\xi})\geq d_\alpha |\boldsymbol{\xi}|^2$,  $(D_b \boldsymbol{\xi}, \boldsymbol{\xi}) \geq d_b|\boldsymbol{\xi}|^2$ for all $\boldsymbol{\xi} \in \mathbb R^3$ and some  $d_b, d_\alpha > 0$,  where $\alpha=p,n$, $j =1,2$, and  $\gamma_p , \gamma_b \geq 0$.
\item[2.]  $\F_{p}: \mathbb R^2 \to \mathbb R^2$  is continuously  differentiable in $\mathcal I_\mu^2$, with $F_{p,1}(0, \eta)= 0$, $F_{p,2}(\xi, 0)= 0$, $F_{p,1}(\xi, \eta)\geq 0$, and  $|F_{p,2}(\xi, \eta)|\leq g_1(\xi)(1+ \eta)$  for all $\xi, \eta \in \mathbb R_+$ and some $g_1 \in C^1(\mathbb R_+; \mathbb R_+)$.
\item[3.] $\J_p: \mathbb R^2\to \mathbb R^2$ is continuously  differentiable in $\mathcal I_\mu^2$,     with   $J_{p,1}(0, \eta) \geq 0$, $J_{p,2}(\xi, 0) \geq 0$, $|J_{p,1}(\xi, \eta)|\leq \gamma_J(1+ \xi)$,  and $|J_{p,2}(\xi, \eta)|\leq g(\xi)(1+ \eta)$   
for  all $\xi,\eta \in \mathbb R_+$ and some  $\gamma_J >0$ and   $g \in C^1(\mathbb R_+; \mathbb R_+)$.
\item[4.]  $\F_n: \mathbb R^4 \to \mathbb R^2$ is continuously  differentiable in $\mathcal I_\mu^4$, with  $ F_{n,1}(\boldsymbol{\xi}, 0, \boldsymbol{\eta}_2)\geq 0$, $F_{n,2}(\boldsymbol{\xi}, \boldsymbol{\eta}_1, 0)\geq 0$, and 
  \begin{equation*}
  \begin{aligned}
 |F_{n,1}(\boldsymbol{\xi}, \boldsymbol{\eta})| \leq \gamma^1_F (1+ g_2(\boldsymbol{\xi})+|\boldsymbol{\eta}|), 
  && |F_{n,2}(\boldsymbol{\xi}, \boldsymbol{\eta})| \leq \gamma^2_F (1+ g_2(\boldsymbol{\xi})+|\boldsymbol{\eta}|), 
  \end{aligned}
  \end{equation*}
  for  all $\boldsymbol{\xi}, \boldsymbol{\eta} \in \mathbb R^2_+$ and some $\gamma_F^1, \gamma_F^2  >0$ and  $g_2 \in C^1(\mathbb R_+^2; \mathbb R_+)$. 
\item[5.] $\R_n: \mathbb R^3\times \mathbb R_+ \to  \mathbb R^2$ and  $R_b:  \mathbb R^3\times \mathbb R_+ \to  \mathbb R$  are continuously differentiable in $\mathcal I_\mu^3 \times \mathbb R_+$ and satisfy 
  \begin{equation*}
  \begin{aligned}
&R_{n,1}(0, \boldsymbol{\xi}_2, \eta,  \zeta) \geq 0, \qquad &&   |\R_{n,1}(\boldsymbol{\xi},  \eta ,  \zeta)| \leq \beta_1(1+|\boldsymbol{\xi}| + \eta) (1+ \zeta),\\
&R_{n,2}(\boldsymbol{\xi}_1, 0, \eta,  \zeta) \geq 0 ,    && |\R_{n,2}(\boldsymbol{\xi},  \eta ,  \zeta)| \leq \beta_2(1+|\boldsymbol{\xi}| + \eta) (1+ \zeta) ,   \\
&  R_b(\boldsymbol{\xi}, 0,  \zeta) \geq 0, && |R_b(\boldsymbol{\xi},  \eta ,  \zeta)| \leq \beta_3(1+|\boldsymbol{\xi}| + \eta) (1+\zeta),
  \end{aligned}
  \end{equation*}
  for  some $\beta_1, \beta_2, \beta_3>0$ and  all $\boldsymbol{\xi} \in \mathbb R^2_+$, $\eta, \zeta \in \mathbb R_+$.
\item[6.] $\J_n: \mathbb R^2\to \mathbb R^2$ is continuously  differentiable in $\mathcal I_\mu^2$,   with $J_{n,1}(0, \eta) \geq 0$, $J_{n,2}(\xi, 0) \geq 0$,    $|J_{n,1} (\xi, \eta) |\leq \gamma_{n}^1(1+ \xi)$,  and $|J_{n,2} (\xi, \eta)| \leq \gamma_{n}^2(1+ \xi+\eta)$   
for all $\xi, \eta \in \mathbb R_+$ and some $\gamma_{n}^1, \gamma_n^2 >0$. 
\item[7.]  $\G(\xi, \eta): \mathbb R^2 \to \mathbb R^2$, with  $\G(\xi, \eta)=(0,  \gamma_1  - \gamma_2 \eta)^T$ for $\eta \in \mathbb R$ and  some $\gamma_1, \gamma_2 \geq 0$. 
\item[8.] The initial conditions  $\p_0, \n_{0} \in  L^\infty(\Omega)^2$ and  $b_0 \in H^1(\Omega)\cap L^\infty(\Omega)$ are non-negative.
\item[9.] $\mathbf{f} \in H^1(0,T; L^2(\Gamma_{\cE}\cup \Gamma_{\cU}))^3$ and $p_{\cI} \in H^1(0,T; L^2(\Gamma_{\cI}))$.
\item[10.] $\mathbb E_M \in C^1(\mathbb R)$,  $\mathbb E_F$, $\mathbb E_M$ possess major and minor symmetries, i.e. $\mathbb E_{L, ijkl}= \mathbb E_{L, klij}=\mathbb E_{L, jikl}=\mathbb E_{L, ijlk}$, for $L=F,M$, 
  and there exists $\omega_E >0$ s.t.\  $\mathbb E_F \bA \cdot \bA \geq \omega_E |\bA|^2$ and  $\mathbb E_M(\xi) \bA \cdot \bA \geq \omega_E |\bA|^2$  for  all symmetric $\bA \in \mathbb R^{3\times 3}$ and $\xi \in \mathbb R_+$.
   There exists $\gamma_{M}>0$ s.t.\  $|\mathbb E_M(\xi)| \leq \gamma_{M} $  for all  $\xi \in \mathbb R_+$.
  \item[11.]  $\Q_n \in C( \mathbb R^3\times \mathbb R^{3\times 3};  \mathbb R^2)$ and  $Q_b \in C( \mathbb R^3\times \mathbb R^{3\times 3};  \mathbb R)$ satisfy
  $$
  \begin{aligned}
&  Q_{n,1}(0,\boldsymbol{\xi}_2, \eta,  \bA) \geq 0, \qquad Q_{n,2}(\boldsymbol{\xi}_1, 0, \eta,  \bA) \geq 0, \qquad Q_b(\boldsymbol{\xi}, 0,  \bA) \geq 0, \\
&  |\Q_n(\boldsymbol{\xi},  \eta ,  \bA)| +
   |Q_b(\boldsymbol{\xi},  \eta ,  \bA)| \leq \gamma_1(1+|\bA|) (1+|\boldsymbol{\xi}| + \eta),  \\
&  |\Q_n(\boldsymbol{\xi},\eta, \bA^1)-  \Q_n(\boldsymbol{\xi}, \eta, \bA^2)| +  |Q_b(\boldsymbol{\xi},\eta, \bA^1)-  Q_b(\boldsymbol{\xi}, \eta, \bA^2)| \leq \gamma_2(1+  |\boldsymbol{\xi} |+ \eta)|\bA^1 - \bA^2|,  \\
&   |\Q_n(\boldsymbol{\xi}^1, \eta^1,  \bA)-  \Q_n(\boldsymbol{\xi}^2, \eta^2,  \bA)| \leq \gamma_3(1+ |\bA|)( 1+ |\boldsymbol{\xi}^1|+ |\boldsymbol{\xi}^2| + |\eta^1|+ |\eta^2|)\big(|\boldsymbol{\xi}^1 - \boldsymbol{\xi}^2|   + |\eta^1 - \eta^2|\big),  \hspace{-0.5 cm } \\
 &  |Q_b(\boldsymbol{\xi}^1, \eta^1,  \bA)-  Q_b(\boldsymbol{\xi}^2, \eta^2,  \bA)| \leq \gamma_4(1+ |\bA|)( 1+ |\boldsymbol{\xi}^1|+ |\boldsymbol{\xi}^2| + |\eta^1|+ |\eta^2|) |\eta^1 - \eta^2|\\
  &\hspace{ 6 cm }  + \gamma_5( 1+ |\boldsymbol{\xi}^1|+ |\boldsymbol{\xi}^2| + |\eta^1|+ |\eta^2|) |\boldsymbol{\xi}^1 - \boldsymbol{\xi}^2|
  \end{aligned}
  $$
 for  some  $\gamma_l>0$, $l=1,\ldots, 5$, and all $\bA, \bA^1, \bA^2 \in \mathbb R^{3\times 3}$,  $\boldsymbol{\xi} \in \mathbb R^2_+$, $\eta \in \mathbb R_+$, $\boldsymbol{\xi}^j \in  \mathcal I_\mu^2$, $\eta^j \in \mathcal I_\mu^1$,  $j=1,2$.   
\end{itemize} 
\end{assumption}  
{\bf Remark 1.} Notice that for  $P$ of the form \eqref{def_G_2}  Assumption~\ref{assumptions}$.11$ is satisfied if the elasticity tensor for the cell wall matrix is bounded from above, as in  Assumption~\ref{assumptions}$.10$.  This assumption is not restrictive, since every biological material will have a maximal possible stiffness. 

{\bf Remark 2.} To prove the non-negativity of solutions $\p^\ve, \n^\ve, b^\ve$ of the systems \eqref{sumbal},  \eqref{sumbal_11} or \eqref{sumbal}, \eqref{sumbal_1}, with the boundary conditions in \eqref{BC1} and \eqref{BC4},  Lipschitz continuity of the reaction terms and nonlinear functions in the boundary conditions in an open neighbourhood of zero is needed. However it is sufficient to specify  the growth assumptions only for non-negative values  of $\p^\ve_1$,  $\p^\ve_2$, $\n^\ve_1$, $\n_2^\ve$, and  $b^\ve$.   The non-negativity  assumptions on the nonlinear functions in the reaction terms and boundary conditions   ensure the non-negativity  of the solutions of the system. The non-negativity of $F_{p,1}(\xi, \eta)$ and the sub-linearity  of $J_{p,1}(\xi, \eta)$, uniform in $\eta$,  for all $\xi, \eta \in \mathbb R_+$, are used to show  the uniform boundedness of $\p^\ve_1$.

{\bf Remark 3.} Notice that the reaction terms and boundary conditions in  the model developed in Section~9 satisfy  Assumption~\ref{assumptions}, with 
$\F_p(\p) = (R_{eE}(\p), 0)^T$, \, $\F_n(\p, \n) = (R_{eE}(\p) - 2R_{dc}(\n) - R_d \n_1, -R_{dc}(\n))^T$, \, 
$\R_n(\n, b, {\cN}_\delta(\be(\bu)))=(2 R_b(b)\, {\cN}_\delta(\be(\bu)),  R_b(b)\, {\cN}_\delta(\be(\bu)))^T $, \,  $R_b(\n, b, {\cN}_\delta(\be(\bu)))=R_{dc}(\n)- R_b(b)\, {\cN}_\delta(\be(\bu))$, \, 
$\J_p(\p)=(J_e(\langle \p_1, 1\ \rangle_{\Omega_M}), J_E(\langle \p_1, 1\ \rangle_{\Omega_M})  - \zeta_E \p_2)^T$, and $\J_n(\n)= ( - \gamma_d \n_1,  \gamma_{c,1} - \gamma_{c,2} \n_2)^T$,  where $\p=(n_e, n_E)^T$,  $\n=(n_d, n_c)^T$ and $b=n_b$. \\

Next we give the  definitions of  weak solutions of both microscopic problems: Model~I and Model~II. 
\begin{definition}
A weak solution of the microscopic problem  \eqref{sumbal}--\eqref{sumbal2} are functions $(\p^\ve, \n^\ve, b^\ve, \bu^\ve)$,   such that 
$\bu^\ve \in L^2(0,T;\cW(\Omega))$, 
  $b^\ve \in H^1(0,T; L^2(\Omega^\ve_M))$, $\p^\ve, \n^\ve \in L^2(0,T;\cV(\Omega^\ve_M))^2$, $\partial_t \p^\ve, \partial_t \n^\ve \in L^2(0,T; \cV(\Omega^\ve_M)^\prime)^2$    and 
satisfy the equations 
\begin{eqnarray}\label{weak_sol_n1}
\qquad \langle \partial_t \p^\ve, \boldsymbol{\phi}_p \rangle_{\cV, \cV^\prime} + \langle D_p \nabla \p^\ve, \nabla \boldsymbol{\phi}_p \rangle_{\Omega^\ve_{M,T}} 
= -\langle  \F_{p}(\p^\ve), \boldsymbol{\phi}_p \rangle_{\Omega^\ve_{M,T}}
+ \big\langle \J_p(\p^\ve), \boldsymbol{\phi}_p\big \rangle_{\Gamma_{\cI,T}}-  \big\langle \gamma_p\, \p^\ve, \boldsymbol{\phi}_p \big\rangle_{\Gamma_{\cE, T}}
\end{eqnarray}
and
\begin{eqnarray}\label{weak_sol_n11}
\begin{aligned}
\langle \partial_t \n^\ve, \boldsymbol{\phi}_n \rangle_{\cV, \cV^\prime}  + \langle D_n \nabla \n^\ve, \nabla \boldsymbol{\phi}_n \rangle_{\Omega^\ve_{M,T}}  = & \; 
 \big\langle \F_n(\p^\ve, \n^\ve)  +  \R_n(\n^\ve, b^\ve, \mathcal N_\delta(\be(\bu^\ve))), \boldsymbol{\phi}_n \big \rangle_{\Omega^\ve_{M,T}}
\\ & + \big\langle  \J_n(\n^\ve), \boldsymbol{\phi}_n \big\rangle_{\Gamma_{\cE, T}} + 
\big\langle \G(\n^\ve)\,  \mathcal N_\delta(\be(\bu^\ve)), \boldsymbol{\phi}_n \big\rangle_{\Gamma_{\cI,T}}
\end{aligned}
\end{eqnarray}
for all  $\boldsymbol{\phi}_\alpha\in L^2(0,T; \cV(\Omega_{M}^\ve))^2$,     where $\alpha=p,n$,  
\begin{equation}\label{weak_sol_n2}
 \partial_t b^\ve =  R_b(\n^\ve, b^\ve, \mathcal N_\delta(\be(\bu^\ve)))  \quad \text{a.e.\ in } \; \;   \Omega^\ve_{M,T}, 
\end{equation}
and 
\begin{equation}\label{weak_sol_u}
\big\langle \bbE^\ve(b^\ve,x)\be(\bu^\ve), \be(\boldsymbol{\psi}) \big\rangle_{\Omega_T} = -\langle p_\cI\, \bnu, \boldsymbol{\psi} \rangle_{\Gamma_{\cI,T}} +\langle \bff, \boldsymbol{\psi} \rangle_{\Gamma_{\cE\cU,T}}
\end{equation}
for all $\boldsymbol{\psi} \in L^2(0,T; \mathcal W(\Omega))$. Furthermore, $\p^\ve$, $\n^\ve$, and $b^\ve$ satisfy the  initial conditions in $L^2(\Omega_M^\ve)$, i.e.\ $\p^\ve(t, \cdot) \to \p_0$,  $\n^\ve(t, \cdot) \to \n_0$ in $L^2(\Omega_M^\ve)^2$,  and  $b^\ve(t, \cdot)  \to b_0$  in $L^2(\Omega_M^\ve)$  as $t \to 0$.
\end{definition}

A weak solution of  Model II  is defined in the following way.
\begin{definition}
A weak solution of the microscopic problem  \eqref{sumbal}, \eqref{BC1}--\eqref{BC4} are functions $(\p^\ve, \n^\ve, b^\ve, \bu^\ve)$,   such that 
\, $\bu^\ve \in L^2(0,T;\cW(\Omega))$, 
 $\p^\ve, \n^\ve \in L^2(0,T;\cV(\Omega^\ve_M))^2$, 
 $\partial_t \p^\ve, \partial_t \n^\ve \in L^2(0,T; \cV(\Omega^\ve_M)^\prime)^2$,  $b^\ve \in L^2(0,T;\cV(\Omega^\ve_M))$,  $\partial_t b^\ve \in L^2(0,T; \cV(\Omega^\ve_M)^\prime) $,  and  satisfy the equations \eqref{weak_sol_n1} and  \eqref{weak_sol_u}, and  
\begin{eqnarray}\label{weak_sol_n3}
\begin{aligned}
\langle \partial_t \n^\ve, \boldsymbol{\phi}_n \rangle_{\cV, \cV^\prime}  + \langle D_n \nabla \n^\ve, \nabla \boldsymbol{\phi}_n \rangle_{\Omega^\ve_{M,T}} & = \langle \F_n(\p^\ve, \n^\ve) +   \Q_n(\n^\ve, b^\ve,   \be(\bu^\ve)), \boldsymbol{\phi}_n \rangle_{\Omega^\ve_{M,T}} + 
 \langle  \J_n(\n^\ve), \boldsymbol{\phi}_n \rangle_{\Gamma_{\cE, T}}  \hspace{-1 cm} \\
&\qquad  +  \big\langle  \G(\n^\ve)\,  \mathcal N_\delta(\be(\bu^\ve)), \boldsymbol{\phi}_n \big\rangle_{\Gamma_{\cI,T}} , \\
\langle \partial_t b^\ve, \phi_b \rangle_{\cV, \cV^\prime}  + \langle D_b \nabla b^\ve, \nabla \phi_b \rangle_{\Omega^\ve_{M,T}}  & = 
\big\langle  Q_b(\n^\ve, b^\ve, \be(\bu^\ve)), \phi_b \big\rangle_{\Omega^\ve_{M,T}} 
- \langle  \gamma_b \, b^\ve, \phi_b \rangle_{\Gamma_{\cE, T}}, 
\end{aligned}
\end{eqnarray}
for all $\boldsymbol{\phi}_n \in L^2(0,T; \cV(\Omega_{M}^\ve))^2$, $\phi_b\in L^2(0,T; \cV(\Omega_{M}^\ve))$. Furthermore,  $\p^\ve$, $\n^\ve$, and $b^\ve$ satisfy the initial conditions   in $L^2(\Omega_M^\ve)^2$ and $L^2(\Omega_M^\ve)$, respectively.
\end{definition}

\section{Formulation of  main results}\label{main_results} 
The main results of the paper are the establishment of  the existence and uniqueness results for both of the  microscopic problems  and the rigorous  derivation of the macroscopic equations using homogenization techniques.
  
To show the well-posedness of the microscopic  problems we consider first the system of linear elasticity for a given $b^\ve$ and the  reaction-diffusion system for a given displacement $\bu^\ve$. 
The Lax--Milgram theorem is used to show the existence of a solution of the problem \eqref{sumbal2} for a given $b^\ve$, whereas the Galerkin method and the Schauder fixed-point theorem are applied to prove the well-posedness of both systems   \eqref{sumbal}--\eqref{BC1}  and \eqref{sumbal},~\eqref{BC1},~\eqref{sumbal_1}, \eqref{BC4} for a given~$\bu^\ve$.
 Then we apply  the   Banach fixed-point theorem  to show the existence and uniqueness of a weak  solution  of the coupled system.  Because of quadratic non-linearities the proof of the fixed-point argument is non-standard, and 
the main difficulty  is in deriving a contraction inequality involving  the $L^\infty$-norm of $b^\ve$. 

 In the case of  Model I  (no diffusion term in the equation for  $b^\ve$) the dependence  of  the reaction term in the equation for $b^\ve$   on  a local average of $\bbE^\ve(b^\ve, x)\be(\bu^\ve)$    is important  for the derivation of a contraction inequality.  For the proof of the strong convergence of $b^\ve$ it is crucial   that the average in $\mathcal N_\delta(\be(\bu^\ve))$ is independent of $\ve$. 

 The regularity of $b^\ve$ and  delicate estimates for the terms $\Q_n(\n^\ve, b^\ve, \be(\bu^\ve))$ and $Q_b(\n^\ve, b^\ve, \be(\bu^\ve))$ are used to prove the existence of  a  unique solution of Model II.   To derive the macroscopic equations for the problem \eqref{sumbal},~\eqref{BC1}--\eqref{BC4}  we   prove the strong two-scale convergence of $\be(\bu^\ve)$. More specifically, the strong two-scale convergence of $\be(\bu^\ve)$ is needed  to pass to the limit in the nonlinear functions   $\Q_n(\n^\ve, b^\ve, \be(\bu^\ve))$ and $Q_b(\n^\ve, b^\ve, \be(\bu^\ve))$.  
Recursive estimations of the $L^p$-norms, for all $p\geq 2$,  are used to derive a  contraction inequality in $L^\infty(0, T; L^\infty(\Omega_M^\ve))$. This method is also applied to show the boundedness of $\n^\ve$ and $b^\ve$,  although  other methods  can also be used  to derive the $L^\infty$-estimates  for $\n^\ve$ and $b^\ve$, see e.g.\ \cite{AndresMariya, Ladyzhenskaya, MP2008}.  The uniform in $\ve$ bounded\-ness of $\p^\ve$,  $\n^\ve$ and $b^\ve$ is used  in the proof of the  strong two-scale convergence of $\be(\bu^\ve)$.

\begin{theorem}\label{existence_2}
Under Assumption~\ref{assumptions} there exists a unique non-negative weak solution of the microscopic Model~I,  {\rm \eqref{sumbal}--\eqref{sumbal2}}, satisfying the   \textit{a priori} estimates 
\begin{equation}\label{estim_un_1}
\begin{aligned}
&\| \bu^\ve\|_{L^\infty(0,T;\cW(\Omega))} +   \|\partial_t \bu^\ve\|_{L^2(0,T; \cW(\Omega))} \leq C,   \\
 & \| b^\ve \|_{W^{1,\infty}(0,T; L^\infty(\Omega_{M}^\ve))}\leq C, \\
&\|\p^\ve\|_{L^\infty(0,T; L^\infty(\Omega_{M}^\ve))} + \|\nabla \p^\ve\|_{L^2( \Omega_{M,T}^\ve)}+ \|\n^\ve\|_{L^\infty(0,T; L^\infty(\Omega_{M}^\ve))} + \|\nabla \n^\ve\|_{L^2( \Omega_{M,T}^\ve)} \leq C,  \\
 &    \| \vartheta_h\p^\ve - \p^\ve\|_{L^2(\Omega_{M, T-h}^\ve)}+ \| \vartheta_h\n^\ve - \n^\ve\|_{L^2(\Omega_{M, T-h}^\ve)} \leq Ch^{1/4}, 
   \end{aligned}
\end{equation}
for all $h>0$, where $\vartheta_h v (t,x) = v(t +h, x)$ for $(t,x) \in (0, T-h] \times \Omega_M^\ve$ and  the  constant $C$ is independent of $\ve$.
\end{theorem}
A similar result also holds   for  Model II. The main difference in the proof of the well-posedness results for   both of the  microscopic problems (Model I and Model II)  is in the derivation of  {\it a priori} estimates. 
\begin{theorem}\label{existence_3}
Under Assumption~\ref{assumptions} there exists a unique non-negative weak solution of the microscopic Model~II,   {\rm  \eqref{sumbal},~\eqref{BC1}--\eqref{BC4}}, satisfying the   \textit{a priori} estimates 
\begin{eqnarray}
&&  \| \bu^\ve\|_{L^\infty(0,T;\cW(\Omega))}  \leq C, \label{estim_un_2}   \\
&& \begin{aligned} \label{estim_un_3}
 &\|\p^\ve\|_{L^\infty(0,T; L^\infty(\Omega_{M}^\ve))} + \|\nabla \p^\ve\|_{L^2( \Omega_{M,T}^\ve)} + \|\n^\ve\|_{L^\infty(0,T; L^\infty(\Omega_{M}^\ve))} + \|\nabla \n^\ve\|_{L^2( \Omega_{M,T}^\ve)}    \leq C, \\
 & \|b^\ve\|_{L^\infty(0,T; L^\infty(\Omega_{M}^\ve))} + \|\nabla b^\ve\|_{L^2( \Omega_{M,T}^\ve)} \leq C, \\
 &\| \vartheta_h\p^\ve - \p^\ve\|_{L^2(\Omega_{M, T-h}^\ve)}+ \| \vartheta_h\n^\ve - \n^\ve\|_{L^2(\Omega_{M, T-h}^\ve)} 
 + \| \vartheta_h b^\ve - b^\ve\|_{L^2(\Omega_{M, T-h}^\ve)}\leq C h^{1/4},  
 \end{aligned}   
\end{eqnarray}
for all $0<h<T$, where $\vartheta_h v (t,x) = v(t +h, x)$ for $(t,x) \in (0, T-h] \times \Omega_M^\ve$ and   the  constant $C$ is independent of $\ve$. 
\end{theorem}

Using the \textit{a priori} estimates  in Theorems~\ref{existence_2}~and~\ref{existence_3}  and applying  the   two-scale convergence and the unfolding method, see e.g.~\cite{allaire,CDG,CDDGZ,Nguetseng}, we derive the macroscopic equations  for both microscopic models of plant cell wall biomechanics.  
First we formulate the unit cell problems,  which are obtained by the derivation of the macroscopic equations and  determine the macroscopic elastic and diffusive properties of the plant cell wall.   

The macroscopic diffusion coefficients and  elasticity tensor  are defined by  
\begin{equation}\label{macro_coef}
\begin{aligned}   
&\mathcal D^l_{\alpha, ij} =\dashint_{\hat Y_M} \big[ D_{\alpha, ij}^l  +  (D_\alpha^l \hat \nabla_{y}{\bf v}^j_{\alpha, l})_i \big] d\hat y \;\;    \text{ for }   i,j=1,2,3, \quad \hat \nabla_y  {\bf v}^j_{\alpha, l} =(\partial_{y_1}   {\bf v}^j_{\alpha, l}, \partial_{y_2}   {\bf v}^j_{\alpha, l}, 0)^T,
 \\
&\mathcal D_{b, ij} =\dashint_{\hat Y_M} \big[ D_{b, ij}  +  (D_b\hat \nabla_{y} {v}^j_{b})_i \big] d\hat y  \qquad  \text{ for }   i,j=1,2,3, \quad \hat \nabla_y  {v}^j_{b} =(\partial_{y_1}   {v}^j_{b}, \partial_{y_2}   {v}^j_{b}, 0)^T,  
 \\
&\mathbb E_{{\rm hom},ijkl} (b)=  \dashint_{\hat Y}  \big[\mathbb E_{ijkl} (b,y) + \big(\mathbb E (b,y) \hat \be_y(\textbf{w}^{ij})\big)_{kl} \big] d\hat y, \qquad\;  i,j,k,l=1,2,3,
\end{aligned}
\end{equation}
where $\hat y= (y_1, y_2)$, $\alpha = p,n$ and $l=1,2$,   and the functions   ${\bf v}_\alpha^j=({\bf v}_{\alpha,1}^j, {\bf v}_{\alpha,2}^j)^T$, $v_b^j$ and $\textbf{w}^{ij}$ are solutions of the unit cell problems
\begin{equation}\label{unit_n}
\begin{aligned}
&{\rm div}_{\hat y} (\hat D_\alpha^l \nabla_{\hat y} {\bf v}^j_{\alpha, l}) =0 &&\text{ in } \hat Y_M, \qquad j=1,2,3 \\
&  (\hat D_\alpha^l \nabla_{\hat y} {\bf v}^j_{\alpha, l} + \widetilde D_\alpha^l {\textbf{b}}_j) \cdot \bnu =0  &&\text{ on } \hat \Gamma, \qquad 
{\bf v}^j_{\alpha, l}  \quad \hat Y -\text{ periodic}, \; \quad  \langle {\bf v}^j_{\alpha, l}, 1 \rangle_{\hat Y_M} =0
\end{aligned}
\end{equation}
for $\alpha=n,p$, $l=1,2$, and $\hat D_\alpha^l =(D_{\alpha, ij}^l)_{i,j=1,2}$,  
$\widetilde D_\alpha^l =(D_{\alpha, ij}^l)_{i=1,2, j=1,2,3}$,
\begin{equation}\label{unit_b}
\begin{aligned}
&{\rm div}_{\hat y} (\hat D_b \nabla_{\hat y} {v}^j_b) =0 &&\text{ in } \hat Y_M, \qquad j=1,2,3 \\
& ( \hat D_b \nabla_{\hat y} v^j_b + \widetilde D_b {\textbf{b}}_j) \cdot \bnu =0  &&\text{ on } \hat \Gamma, \qquad 
v^j_b  \quad \hat Y -\text{ periodic}, \; \quad  \langle  v^j_b, 1 \rangle_{\hat Y_M} =0,
\end{aligned}
\end{equation}
  where $\hat D_b =(D_{b, ij})_{i,j=1,2}$ and $\widetilde D_b =(D_{b, ij})_{i=1,2, j=1,2,3}$,   and     
\begin{equation}\label{unit_u}
\begin{aligned}
&\hat{\rm div}_{y}  ({\mathbb E}(b, y)(\hat \be_{y}(\textbf{w}^{ij}) + \textbf{b}^{ij})) ={\bf 0} \quad &&\text{in }  \hat Y, \\
&\langle  \textbf{w}^{ij}, 1 \rangle_{\hat Y}  ={\bf 0}, \qquad \qquad \hspace{1 cm} \textbf{w}^{ij} \quad   && \hat Y -\text{ periodic}
\end{aligned}
\end{equation}
for  $(t,x) \in \Omega_T$. Here,     
$\textbf{b}^{jk} = \frac12 ({\textbf{b}}_j \otimes\textbf{b}_k +  \textbf{b}_k \otimes \textbf{b}_j)$, where $(\textbf{b}_j)_{1\leq j\le3}$ is the canonical basis of $\mathbb R^3$.
 For a vector-valued function  ${\bf v}$ we denote  $\hat{\rm div}_y  {\bf v} = \partial_{y_1} {\bf v}_1 +  \partial_{y_2} {\bf v}_2$ and
 $\hat \be_y({\bf v})$ is defined in the following way: 
$\hat \be_{y} ({\bf v})_{33} = 0$,   $\hat \be_{y} ({\bf v})_{3j} =\hat \be_{y} ({\bf v})_{j3} =  \frac 12 \partial_{y_j} {\bf v}_3$ for $j=1,2$, 
and $\hat \be_{y} ({\bf v})_{ij} =  \frac 12 (\partial_{y_i} {\bf v}_j +  \partial_{y_j} {\bf v}_i)$ for $i,j=1,2$. \\

We have the following macroscopic  equations for the microscopic models of plant cell wall biomechanics. 
\begin{theorem}\label{th:macro_1}
A sequence of solutions of the  microscopic problem  {\rm \eqref{sumbal}--\eqref{sumbal2}} converges to a solution  $(\p, \n, b, \bu)$, with   $\p, \n \in \big(L^2(0,T; \mathcal V(\Omega)) \cap L^\infty(0,T; L^\infty(\Omega))\big)^2$,  $b \in H^1(0,T; L^2(\Omega))\cap L^\infty(0,T; L^\infty(\Omega))$, and  $\bu \in L^\infty(0,T; \mathcal W(\Omega))$,  of the macroscopic equations 
\begin{eqnarray}
 \phantom{ \partial_t n_E -}\;  {\rm div} ( \mathbb E_{\rm hom}( b) \be(\bu)) =\; {\bf 0}
   \hspace{ 5.4 cm} && \text{ in } \Omega_T, \label{macro_1}\\
  \begin{cases}\label{macro_11}
 \partial_t  \p - {\rm div} (\mathcal D_p \nabla \p)\;  =  - \F_{p} (\p)    \\  
\partial_t \n- {\rm div} (\mathcal D_n \nabla \n) =\; \F_n (\p, \n) +  \R_n(\n, b, \mathcal N^{\rm eff}_\delta(\be(\bu)))   \\
\partial_t  b    \phantom{- {\rm div} (\mathcal D_c \nabla n_c)}\; \,  =  \; \, R_b(\n, b,  \mathcal N^{\rm eff}_\delta (\be(\bu)) ) 
\end{cases}  \quad &&\text{ in } \Omega_T, \; \; 
\end{eqnarray}
   together with the boundary and initial  conditions 
\begin{eqnarray}
&&\begin{cases}\label{macro_bc_1}
\begin{aligned}
& \mathcal D_p \nabla \p\,  \bnu = \theta_M^{-1} \J_p(\p) &&   \text{on } \Gamma_{\cI,T},  \qquad  \mathcal D_p \nabla \p  \,\bnu =  - \theta_M^{-1} \gamma_p \, \p   &&  \text{on } \Gamma_{\cE, T}\\
& \mathcal D_n \nabla \n \, \bnu = \theta_M^{-1} \,  \G(\n)\, \mathcal N^{\rm eff}_\delta(\be(\bu))   &&   \text{on } \Gamma_{\cI, T}, \qquad
  \mathcal D_n \nabla \n \, \bnu =  \;  \theta_M^{-1} \J_n(\n)  &&  \text{on } \Gamma_{\cE, T}, \\
&\mathcal D_p \nabla \p \, \bnu = 0, \quad \qquad \; \mathcal D_n \nabla \n \, \bnu = 0 &&   \text{on } \Gamma_{\mathcal U,T},  \\
&\;  \; \p (0,x) = \p_{0} (x), \qquad \n (0,x) = \n_{0} (x) &&  \text{in }  \Omega,  \qquad \qquad  b (0,x) = b_{0} (x) &&  \text{in }  \Omega,  \\
& \mathbb E_{\rm hom}(b) \be(\bu) \,\bnu = - p_\cI \bnu \quad \qquad \;  && \text{on } \Gamma_{\cI, T},  \qquad 
\mathbb E_{\rm hom}(b) \be(\bu)\,  \bnu = \bff   &&  \text{on } \Gamma_{\cE, T}\cup \Gamma_{\cU, T},
\end{aligned}
\end{cases}
\end{eqnarray}
where  $\theta_M= |\hat Y_M|/ |\hat Y|$ and 
\begin{equation}\label{macro_over_n_eu}
\mathcal N^{\rm eff}_\delta(\be(\bu))(t,x)= \Big(\dashint_{B_\delta(x)\cap \Omega} {\rm tr}\,\big( \bbE_{{\rm{hom}}}(b)\be(\bu) \big) \, d\tilde x\Big)^+ \qquad \qquad \text{ for } (t,x) \in \Omega_T.
\end{equation}  
Here ${\rm div} (\mathcal D_p \nabla \p) = ( {\rm div} (\mathcal D^1_p \nabla \p_1), {\rm div} (\mathcal D^2_p \nabla \p_2))^T$ and  
${\rm div} (\mathcal D_n \nabla \n) = ( {\rm div} (\mathcal D^1_n \nabla \n_1), {\rm div} (\mathcal D^2_n \nabla \n_2))^T$, and 
the macroscopic  diffusion coefficients  $\mathcal D_\alpha^l$,  for $\alpha=p,n$ and  $l=1,2$,  are defined  in \eqref{macro_coef}.
\end{theorem}

The difference between the macroscopic problems for Model~I and Model~II is reflected in the equations for  $\n$~and~$b$. 
\begin{theorem}\label{th:macro_2}
A sequence of solutions of the microscopic problem  {\rm \eqref{sumbal}, \eqref{BC1}--\eqref{BC4}}  converges to a solution
$\p, \n \in \big(L^2(0,T; \mathcal V(\Omega)) \cap L^\infty(0,T; L^\infty(\Omega))\big)^2$,  $b \in L^2(0,T; \mathcal V(\Omega)) \cap L^\infty(0,T; L^\infty(\Omega))$,  and $\bu \in L^\infty(0,T; \mathcal W(\Omega))$
 of the macroscopic equations  \eqref{macro_1} and 
\begin{equation}\label{macro_2}
\begin{cases}
\partial_t  \p - {\rm div} (\mathcal D_p \nabla \p)\, =  - \F_{p} (\p)    \\ 
  \partial_t \n- {\rm div} (\mathcal D_n \nabla \n) =\; \;  \F_n (\p, \n) +   \Q_n^{\rm eff}(\n, b, \be(\bu) )  \\
\partial_t b - {\rm div} (\mathcal D_b \nabla b) \;\; = \phantom{\; \;  \F_n (\p, \n) +\; \; } Q_b^{\rm eff}(\n, b, \be(\bu))
\end{cases} \quad \text{ in } \;  \Omega_T, 
\end{equation}
   together with the initial and boundary conditions \eqref{macro_bc_1} and 
\begin{equation}\label{macro_bc_2}
\begin{aligned}
&\mathcal D_b \nabla b \cdot \bnu = 0 && \text{on } \Gamma_{\cI, T}\cup \Gamma_{\cU, T},  \quad && \mathcal D_b \nabla b \cdot \bnu =  - \theta_M^{-1}\, \gamma_b \, b   &&  \text{on } \Gamma_{\cE, T}. 
\end{aligned}
\end{equation}
Here    $\Q^{\rm eff}_n (\n, b, \be(\bu))=\dashint_{\hat Y_M} \hspace{-0.15 cm } \Q_n(\n, b, \mathbb W(t,x,y) \be(\bu)) \, dy$, 
$Q^{\rm eff}_b (\n, b, \be(\bu))=\dashint_{\hat Y_M}  \hspace{-0.15 cm } Q_b(\n, b, \mathbb W(t,x,y) \be(\bu)) \, dy$, 
where   
$\mathbb W_{ijkl}(t,x,y) = \delta_{ik}\delta_{jl} + \big( \hat \be_y({\bf w}^{ij} (t,x,y)) \big)_{kl}$ and  ${\bf w}^{ij}$ being  solutions of the unit cell problems \eqref{unit_u}.  

The macroscopic  diffusion coefficients  $\mathcal D_\alpha$  are defined as in \eqref{macro_coef}, with $\alpha = n, p, b$,  where 
${\rm div} (\mathcal D_p \nabla \p) = ( {\rm div} (\mathcal D^1_p \nabla \p_1), {\rm div} (\mathcal D^2_p \nabla \p_2))^T$ and 
${\rm div} (\mathcal D_n \nabla \n) = ( {\rm div} (\mathcal D^1_n \nabla \n_1), {\rm div} (\mathcal D^2_n \nabla \n_2))^T$.
\end{theorem}


\section{\textit{A priori} estimates and  the existence and uniqueness  results  for  the microscopic Model I.}\label{exist_uniq_ModelI}
In this section we analyse  Model I, i.e.\  equations \eqref{sumbal}--\eqref{sumbal2}.   We split the proof of the existence and uniqueness results  into three steps.  First we show that for a given non-negative $b^\ve$ the   equations of linear elasticity  have a uniques solution.   Next we prove the well-posedness  of the problem \eqref{sumbal}--\eqref{BC1} for a given $\bu^\ve$. 
Finally,  showing a contraction inequality for $b^\ve$ in $L^\infty(0, T; L^\infty(\Omega_{M}^\ve))$ and applying the Banach fixed-point theorem,  we conclude  that there exists a unique  weak solution of the coupled system. 

\subsection{Existence and uniqueness of a weak solution $\bu^\ve$ of the problem \eqref{sumbal2} for a given $b^\ve$.}
\begin{lemma}\label{lemub}
Let $b^{\ve,1}, b^{\ve_,2}\in L^\infty(0,T; L^\infty(\Omega_{M}^\ve))$  be given  with  $b^{\ve,1}(t,x), b^{\ve, 2}(t,x)\geq 0$ for a.a.~$(t,x) \in \Omega_{M,T}^\ve$. Then there exist $\bu^{\ve,j}\in L^\infty(0,T;\cW(\Omega))$, with $j=1,2$, satisfying
\beqn\label{simpsysEU2}
\left\{
\begin{aligned}
{\rm div}\, (\bbE^\ve(b^{\ve,j},x)\be(\bu^{\ve,j}))&={\bf 0}&& {\rm in}\ (0,T)\times\Omega,\\
(\bbE^\ve(b^{\ve, j},x)\be(\bu^{\ve,j}))\bnu&=-p_\cI\bnu \quad && {\rm on}\ (0,T)\times\Gamma_\cI,\\
(\bbE^\ve(b^{\ve, j},x)\be(\bu^{\ve,j}))\bnu&=\bff && {\rm on}\ (0,T)\times(\Gamma_\cE\cup \Gamma_{\cU}),
\end{aligned}
\right.
\eeqn
and the  estimates
\begin{eqnarray}
&& \| \bu^{\ve, j}\|_{L^\infty(0,T;\cW(\Omega))} \leq C_1,  \qquad \qquad \qquad j = 1,2, \label{estim_u_21}
\\ 
&& \|\be(\bu^{\ve, 1}-\bu^{\ve, 2})\|_{L^\infty(0,T; L^2(\Omega))}\leq C_2\|b^{\ve, 1}-b^{\ve, 2}\|_{L^\infty(0,T; L^\infty(\Omega_M^\ve))}, \label{ub}
\end{eqnarray}
where the constants $C_1$ and $C_2$ are  independent of  $\ve$ and $b^{\ve, j}$, with  $j=1,2$.  
\end{lemma}

\begin{proof}  Due to the assumptions on $\mathbb E^\ve$, see Assumption~\ref{assumptions}.10, the solutions $\bu^{\ve, j}$ of \eqref{simpsysEU2} exist by the Lax--Milgram Theorem.
Taking $\bu^{\ve, j}$ as a test function in the weak formulation of \eqref{simpsysEU2} and using the properties of  $\mathbb E^\ve$ and the non-negativity  of $b^{\ve, j}$  we obtain
\begin{equation*}
\omega_E \|\be (\bu^{\ve,j}(t))\|_{L^2(\Omega)} \leq \sigma \big[\|\bu^{\ve, j}(t)\|_{L^2(\Gamma_{\mathcal I})} + \|\bu^{\ve, j}(t)\|_{L^2(\Gamma_{\mathcal E}\cup \Gamma_\cU)}  \big]
 + C_\sigma\big[\|\mathbf{f} (t)\|_{L^2(\Gamma_{\mathcal E}\cup \Gamma_\cU)} + \|p_{\mathcal I}(t)\|_{L^2(\Gamma_{\mathcal I})}\big]
\end{equation*}
for a.a.\ $t \in (0,T)$, where $\sigma >0$ is arbitrary and $C_\sigma$ is independent of $\ve$.
Applying the second Korn inequality for $\bu^{\ve, j} \in L^\infty(0,T; \mathcal  W(\Omega))$ and the trace estimate in $H^1(\Omega)$, and choosing $\sigma>0$ sufficiently small  yield the estimate~\eqref{estim_u_21}.

Taking $\bu^{\ve, 1}-\bu^{\ve,2}$ as a test function in the weak formulation  of \eqref{simpsysEU2} for $j=1,2$ and subtracting the resulting equations  imply
\begin{equation*}
\big\langle \bbE^\ve(b^{\ve, 1},x)\, \be(\bu^{\ve,1}-\bu^{\ve, 2}), \be(\bu^{\ve, 1}-\bu^{\ve,2})\big\rangle_{\Omega}=\big\langle\big(\bbE^\ve(b^{\ve,1},x)-\bbE^\ve(b^{\ve, 2},x)\big)\, \be(\bu^{\ve, 2}), \be(\bu^{\ve, 1}-\bu^{\ve, 2})\big\rangle_{\Omega}.
\end{equation*}
Then, using the positive definiteness  and regularity of $\bbE^\ve(b^{\ve,1},x)$ together with   the boundedness of $\be(\bu^{\ve,2})$ in $L^\infty(0,T; L^2(\Omega))$ and of  $b^{\ve, j}$ in $L^\infty(0,T; L^\infty(\Omega_{M}^\ve))$, where $j=1,2$,    we obtain   the inequality~\eqref{ub}.
\end{proof}

\subsection{Existence and uniqueness of a weak solution of {\rm \eqref{sumbal}--\eqref{BC1}} for a given $\bu^\ve$. }
In this subsection we prove that  for a given  $\bu^\ve$ the  system  {\rm \eqref{sumbal}--\eqref{BC1}}   has a unique weak solution. 
In the derivation of the  {\it a priori} estimates, uniform in $\ve$,  we shall use the properties of an  extension of $\p^\ve$ and  $\n^\ve$  from a connected  perforated domain $\Omega_M^\ve$ to $\Omega$. 
Using classical extension results \cite{Acerbi,CiorPaulin99}, we obtain the following lemma.
\begin{lemma}\label{lem:extension}
There exists an extension $\overline v^{\varepsilon}$ of $v^{\varepsilon}$ from
$W^{1,p}(\Omega_M^{\varepsilon})$ into $W^{1,p}(\Omega)$, with  $1\leq  p< \infty$, such that 
\[
\|\overline v^{\varepsilon}\|_{L^{p}(\Omega)}\leq\mu_1\|v^{\varepsilon}\|_{L^{p}(\Omega_M^{\varepsilon})} \; \text{ and } \; 
\|\nabla\overline v^{\varepsilon}\|_{L^{p}(\Omega)}\leq\mu_1\|\nabla v^{\varepsilon}\|_{L^{p}(\Omega_M^{\varepsilon})},
\]
where   the constant $\mu_1$ depends only on $Y$ and $Y_M$, and $\Omega_M^\ve$ is connected, with perforations (microfibrils)
 having empty intersection with $\partial \Omega$ or near $\partial \Omega$  microfibrils are perpendicular to some parts of $\partial \Omega$. See Section~\ref{section:model}  for the description of the microscopic structure  of $\Omega_{M}^\ve$. 
\end{lemma}

\noindent\textbf{Remark. } Notice that the microfibrils do not intersect the boundaries $\Gamma_\cI$, $\Gamma_{\cU}$,  and $\Gamma_{\cE}$, and near the boundaries $(\partial \Omega \setminus(\Gamma_{\cI} \cup \Gamma_{\cE}\cup \Gamma_{\cU}))$ it is sufficient to extend $\p^\ve$ and $\n^\ve$ by reflection in the directions parallel to the corresponding  boundary. Thus, classical extension results \cite{CiorPaulin99,Ptashnyk} apply.

For $v^{\varepsilon} \in L^p(0,T; W^{1,p}(\Omega_M^{\varepsilon}))\cap W^{1,p}(0,T; L^p(\Omega_M^{\varepsilon}))$, define $\hat v^{\varepsilon}(\cdot,t)=\overline v^{\varepsilon}(\cdot,t)$ almost everywhere in $(0,T)$.  Since the extension operator is linear and bounded  and $\Omega_M^{\varepsilon}$ does not depend on $t$, we have  $\hat v^{\varepsilon} \in L^p(0,T; W^{1,p}(\Omega))\cap W^{1,p}(0,T; L^p(\Omega))$ and 
\[
\|{\hat v}^{\varepsilon}\|_{L^{p}(\Omega_T)}
\leq\mu_1\| v^{\varepsilon}\|_{L^{p}(\Omega_{M,T}^{\varepsilon})},   \; \; 
\|\partial_{t}{\hat v}^{\varepsilon}\|_{L^{p}(\Omega_T)}
\leq\mu_1\|\partial_{t} v^{\varepsilon}\|_{L^{p}(\Omega_{M,T}^{\varepsilon})},  \;   \;  
\|\nabla{\hat v}^{\varepsilon}\|_{L^{p}(\Omega_T)}
\leq\mu_1\|\nabla v^{\varepsilon}\|_{L^{p}(\Omega_{M,T}^{\varepsilon})}.
\]
 In the sequel, we shall identify $\p^{\varepsilon}$ and $\n^\ve$ with their extensions.
\begin{theorem}\label{th:exist_1}
Under Assumption~\ref{assumptions} and for $\bu^\ve \in L^\infty(0,T; \cW(\Omega))$ such that 
\begin{equation}\label{estim_u_1}
\| \bu^\ve\|_{L^\infty(0,T;\cW(\Omega))} \leq C,
\end{equation}
where the constant $C$ is  independent of  $\ve$, 
 there exists a unique non-negative  weak solution $(\p^\ve, \n^\ve, b^\ve)$ of the microscopic problem  {\rm \eqref{sumbal}--\eqref{BC1}} satisfying 
 the \textit{a priori} estimates
\begin{equation}\label{apriori_estim}
\begin{aligned}
\|\p^\ve\|_{L^\infty(0,T; L^\infty(\Omega_{M}^\ve))} + \|\nabla \p^\ve\|_{L^2( \Omega_{M,T}^\ve)} +\|\n^\ve\|_{L^\infty(0,T; L^\infty(\Omega_{M}^\ve))} + \|\nabla \n^\ve\|_{L^2( \Omega_{M,T}^\ve)} \leq C, \\
   \| b^\ve \|_{W^{1,\infty}(0,T; L^\infty(\Omega_{M}^\ve))}\leq C,
   \end{aligned}
\end{equation}
where the  constant $C$ is independent of $\ve$.
\end{theorem}

\begin{proof} To show the existence of a  solution of  {\rm \eqref{sumbal}--\eqref{BC1}} for a given $\mathbf u^\ve \in L^\infty(0,T; \cW(\Omega))$, we  shall apply the Schauder and Schaefer  fixed-point theorems and the Galerkin method. 
First we consider the subsystem for $\p^\ve$. 

 For  $ \widetilde \p^\ve_2 \in  L^2(0, T; H^\varsigma(\Omega_{M}^\ve))$  with
 $0 \leq \widetilde \p^\ve_2(t,x)\leq A$
  for  $(t,x)\in \Omega_{M,T}^\ve$ and some constant $A>0$, and $\varsigma \in ( 1/2, 1)$, we consider 
\begin{equation}\label{linearised_1}
\begin{cases}
\partial_t \p^\ve =\text{div}(D_p\nabla \p^\ve) -  \widetilde \F_{p}(\p^\ve) & \text{ in } \Omega_{M,T}^\ve,\\
D_p\nabla \p^\ve\,\bnu =  \widetilde \J_p(\p^\ve) \qquad  \qquad \qquad  \text{ on }   \Gamma_{\cI, T},    \qquad 
D_p\nabla \p^\ve\,\bnu = - \gamma_p\,  \p^\ve \quad   &\text{ on }  \Gamma_{\cE, T},    \\
   \p^\ve \hspace{3.6 cm }  a_3\text{-periodic in } x_3, \quad \quad    D_p \nabla\p^\ve\,\bnu = 0  \quad &   \text{ on }  \Gamma_{\cU,T}\cup  \Gamma^\ve_{T}, \\
\p^\ve(0,x) = \p_{0}(x)  \quad  &   \text{ in } \Omega_{M}^\ve,
\end{cases}
\end{equation}
where $\widetilde \F_{p}(\p^\ve) = (F_{p,1}(\p^\ve_1, \widetilde \p^\ve_2),  F_{p,2}(\p^\ve))^T$
and $\widetilde \J_p(\p^\ve)= ( J_{p,1}(\p^\ve_1,  \widetilde \p^\ve_2), J_{p,2}(\p^\ve))^T$.
Applying the Galerkin method and using  the   estimates similar to   \eqref{energy_neE_1} and \eqref{energy_neE_2}, together with  the boundedness of $\p^\ve_1$ when considering  the problem for $\p^\ve_2$,  we obtain the  existence of a unique weak solution of the problem  \eqref{linearised_1}.

First, we use the theory of positive invariant regions to  show the non-negativity of the solutions of  \eqref{linearised_1}. 
The assumptions on $\F_{p}$ and $\J_p$ ensure  
\begin{eqnarray*}
 F_{p,1}(0,  \widetilde \p^\ve_{2})= 0,   \quad  &&
 J_{p,1}(0,  \widetilde \p^\ve_2)\geq 0,  \; \; \; \qquad  \text{ for all } \;  \widetilde \p^\ve_{2} \geq 0,\\
  F_{p,2}(\p^\ve_1, 0)= 0,   \quad &&
 J_{p,2}(\p^\ve_1,0)\geq 0,  \; \; \; \qquad  \text{ for all } \;   \p^\ve_{1} \geq 0
\end{eqnarray*} 
and  $\F_p$, $\J_p$ are Lipschitz continuous  in $(-\mu, M)^2$ for some $\mu>0$ and any $0< M < +\infty$.  Thus,   the non-negativity of the initial conditions  $\p_{0,1}$ and $\p_{0,2}$  and the Theorem on positive invariant regions    \cite[Theorem 2]{Redlinger},  with $K_1(\p^\ve)= - \p^\ve_1$ and $K_2(\p^\ve)= - \p^\ve_2$,   imply   $\p^\ve_{j}(t,x) \geq 0$ for  $(t,x) \in \Omega_{M,T}^\ve$ and $j = 1,2$. 

Considering   $\p^\ve_1$  as a test function in the  weak formulation of the equation for $\p^\ve_1$ in   \eqref{linearised_1}  and  using the non-negativity of  
$\p^\ve_1$ and $\widetilde \p^\ve_2$, along with the assumptions on $\J_p$ and $\F_{p}$,  we obtain  the  estimate 
\begin{equation}\label{energy_neE_1}
\|\p^\ve_1\|_{L^\infty(0,T; L^2(\Omega_M^\ve))} + \|\p^\ve_1\|_{L^2(0,T; H^1(\Omega_M^\ve))}
\leq C, 
\end{equation}
 where  the constant $C$ is independent of $\ve$.  
The  estimates for the  boundary terms 
 are obtained by using  the  extension of $\p^\ve_1$ from $\Omega^\ve_M$ to $\Omega$,  see  Lemma~\ref{lem:extension}, and the trace inequality
\begin{equation}\label{trace_est}
\begin{aligned}
\|\p^\ve_1\|^2_{L^2(\Gamma_{\cI})} +  \|\p^\ve_1\|^2_{L^2(\Gamma_{\cE})} &\leq C_\sigma \|\p^\ve_1\|^2_{L^2(\Omega)} +\sigma  \|\nabla \p^\ve_1\|^2_{L^2(\Omega)}  \leq 
 \mu_1 \big[C_\sigma \|\p^\ve_1\|^2_{L^2(\Omega_M^\ve)} + \sigma \|\nabla \p^\ve_1\|^2_{L^2(\Omega_M^\ve)} \big], \hspace{-0.5 cm } 
\end{aligned}
\end{equation} 
where $\sigma>0$ is arbitrary, the constant $C_\sigma$ is independent of $\ve$, and $\mu_1$ is as in Lemma~\ref{lem:extension}. 

 Next, we show the boundedness of  $\p^\ve_1$. We define $\Phi^\ve_{\beta}$ as the solution  of the linear problem 
\begin{equation}\label{ax_prob_1}
\left\{
\begin{aligned}
&\partial_t \Phi^\ve_\beta= \text{div}(D \nabla \Phi^\ve_{\beta})   
&& \text{ in } \Omega_{M,T}^\ve,  
&&\; \;  \Phi^\ve_{\beta}(0,x) =0  && \text{ in } \Omega_M^\ve, \\
& D \nabla \Phi^\ve_{\beta}\cdot \bnu = \beta(1+  \Phi^\ve_\beta)  && \text{ on }  \Gamma_{\mathcal I,T}, \qquad  && D \nabla \Phi^\ve_{\beta}\cdot \bnu = 0   && \text{ on }  \Gamma^\ve_T, \\
& D \nabla \Phi^\ve_{\beta}\cdot \bnu = 0  && \text{ on }  \Gamma_{\mathcal E\mathcal U,T}&& \hspace{ 0.6 cm} \Phi^\ve_{\beta} && \; a_3\text{-periodic in }   x_3,
\end{aligned}
\right.
\end{equation}
 where $D$ is symmetric and $(D\boldsymbol{\xi}, \boldsymbol{\xi})\geq d|\boldsymbol \xi|^2$ for  all $\boldsymbol{\xi}\in \mathbb R^3$ and some $d>0$,  and $\beta \geq 0$. In the same way as in \cite{MP2008}, using the extension of $\Phi^\ve_{\beta}$ from $\Omega_M^\ve$ to $\Omega$ we obtain 
$$\|\Phi^\ve_\beta \|_{L^\infty(0, T; L^\infty(\Omega_{M}^\ve))} \leq C,  \qquad  \|\overline {\Phi^\ve_\beta} \|_{L^\infty(0,T; L^\infty(\Omega))} \leq C, $$ 
  where $\overline{\Phi^\ve_\beta}$ denotes the extension of $\Phi^\ve_\beta$ from $\Omega_{M,T}^\ve$ to $\Omega_T$ and  the constant $C$ is independent of $\ve$. We also notice that $\Phi^\ve_{\beta} \geq 0$ in $\Omega_{M, T}^\ve$. 
 Considering $\hat \p^\ve_{1}= \p^\ve_{1} - \Phi^\ve_{\beta_1}$, where $\Phi^\ve_{\beta_1}$ is the solution of the problem \eqref{ax_prob_1} with $D=D_p^1$ and $\beta = \beta_1$, where  $\beta_1 =\gamma_J (A_1+1)$ with $A_1\geq  \|\p_{0,1}\|_{L^\infty(\Omega)}$ and $\gamma_J$ is as  in the Assumption~\ref{assumptions}$.3$,  and taking  $(\hat \p^\ve_{1} - A_1 )^{+}$ as a  test function yield 
$$
\partial_t \| (\hat \p^\ve_{1} - A_1)^{+}\|^2_{L^2(\Omega_M^\ve)} + 2d_{p}\|\nabla( \hat \p^\ve_1 - A_1 )^{+}\|_{L^2(\Omega_M^\ve)}^2
\leq   2 \gamma_J \|(\hat \p^\ve_{1} - A_1)^{+}\|^2_{L^2(\Gamma_\cI)}.
$$
Using the properties of the extension of $(\p^\ve_1-\Phi^\ve_{\beta_1}-A_1)^{+}$ from $\Omega_M^\ve$ to $\Omega$ and  the trace estimate, similar to \eqref{trace_est}, and applying  the Gronwall inequality   we conclude   $\p^\ve_1(t,x) \leq A_1 +  \|\Phi^\ve_{\beta_1} \|_{L^\infty(\Omega_{M,T}^\ve)}$ for a.a.\ $(t,x) \in  \Omega_{M,T}^\ve$.

Using the boundedness of $\p_1^\ve$ and the non-negativity of  
$\p^\ve_1$ and $\p^\ve_2$, along with  the assumptions on $\J_p$ and $\F_{p}$,  and considering   $\p^\ve_2$  as a test function in the  weak formulation of the  equation for $\p^\ve_2$ in   \eqref{linearised_1}   yield 
\begin{equation}\label{energy_neE_2}
\|\p^\ve_2\|_{L^\infty(0,T; L^2(\Omega_M^\ve))} + \|\p^\ve_2\|_{L^2(0,T; H^1(\Omega_M^\ve))} \leq C, 
\end{equation}
 where  the constant $C$ is independent of $\ve$.   The boundary terms are estimated using the inequality similar to \eqref{trace_est}.
 In the same way  as \eqref{energy_neE_1} and  \eqref{energy_neE_2} we also obtain  the uniform estimates for $\|\p^\ve\|_{L^2(0,T; H^1(\Omega_M^\ve))}$ and  $\|\p^\ve\|_{L^\infty(0,T; L^2(\Omega_M^\ve))}$,  with $\p_2^\ve$ instead of $\widetilde \p_2^\ve$ in \eqref{linearised_1}.

To show that $\p^\ve_2$ is bounded,    we consider $\hat \p^\ve_{2}= \p^\ve_{2} - \Phi^\ve_{\beta_2}$, where   $\Phi^\ve_{\beta_2}$ is the solution of  \eqref{ax_prob_1} with $D=D_p^2$ and $\beta= \beta_2$ with 
 $ \beta_2   \geq (A_2 +1)\|g (\p_1^\ve)\|_{L^\infty(\Gamma_{\mathcal I,T})}$, where $A_2 \geq  \|\p_{0,2}\|_{L^\infty(\Omega)}$ and  the function $g$ is as  in the Assumption~\ref{assumptions}$.3$.  
  Notice that the boundedness of $\p_1^\ve$ in $\Omega_{M,T}^\ve$ together with $\p^\ve_1 \in L^2(0,T; H^1(\Omega_M^\ve))$ ensures the boundedness of $\p_1^\ve$ on $\Gamma_{\mathcal I,T}$,  see e.g.\ \cite{Ptashnyk12}. 
  
 Taking   $(\hat \p^\ve_{2} - A_2e^{Mt} )^{+}$, with  $M$ such  that $A_2 M \geq  \big[A_2 +1 + \|\Phi^\ve_{\beta_2} \|_{L^\infty(\Omega_{M,T}^\ve)}\big]\|g_1 (\p_1^\ve)\|_{L^\infty(\Omega^\ve_{M,T})}$,  where   $g_1$ is introduced in  Assumption~\ref{assumptions}$.2$, as a test function in the weak formulation of the equation for $\p^\ve_2$  in \eqref{linearised_1} and using the assumptions on $F_{p,2}$ and $J_{p,2}$ and the  properties of the extension of $(\p^\ve_2-\Phi^\ve_{\beta_2}- A_2e^{Mt} )^{+}$,  and applying the Gronwall inequality   yield 
$\p_2^\ve(t,x) \leq A_2 e^{MT} +   \|\Phi^\ve_{\beta_2} \|_{L^\infty(\Omega_{M,T}^\ve)}$ for a.a.\ $(t,x) \in \Omega_{M,T}^\ve$. Since  $A$ in the assumptions on $\widetilde \p_2^\ve$ is an arbitrary constant, it can be chosen so that 
$A_2 e^{MT} +   \|\Phi^\ve_{\beta_2} \|_{L^\infty(\Omega_{M,T}^\ve)}\leq A$. 

From the equations  \eqref{linearised_1} and the estimates for
$\p^\ve$  in $L^2(0,T; \cV(\Omega_M^\ve))^2$  shown above,   we obtain the boundedness of $\partial_t \p^\ve$   in $L^2(0,T; \cV(\Omega_M^\ve)^\prime)^2$ for every fixed $\ve$.

To show the existence of a solution $\p^\ve$ of   \eqref{sumbal} with the corresponding boundary and initial conditions in \eqref{BC1},  we consider 
\begin{equation*}
 X=\{n \in  L^2(0, T; H^\varsigma(\Omega_{M}^\ve) )\;   \;  | \; \;  0 \leq n(t,x) \leq A \; \text{ for \ } (t,x)\in \Omega_{M,T}^\ve\},  
\end{equation*}
with $\varsigma \in (1/2,1)$,   and define an operator
$\mathcal K_1: X \to X$, where  $ \p_2^\ve = \mathcal K_1(\widetilde \p_2^\ve)$ is given as a solution of the problem~\eqref{linearised_1}.
The continuity of the functions $\F_{p}$ and $\J_p$,  along with the \textit{a priori} estimates for $\p^\ve$  and the compact embedding of   $L^2(0,T; \mathcal V(\Omega_M^\ve))\cap H^1(0,T; \mathcal V(\Omega_M^\ve)^\prime)$ in $L^2(0, T; H^\varsigma(\Omega_{M}^\ve))$, with $\varsigma < 1$, see e.g.\ \cite{Lions},  ensures the continuity of  $\mathcal K_1$. 
Utilizing the \textit{a priori} estimates and the compact embedding of   $L^2(0,T; \mathcal V(\Omega_M^\ve))\cap H^1(0,T; \mathcal V(\Omega_M^\ve)^\prime)$ in $L^2(0, T; H^\varsigma (\Omega_{M}^\ve))$ again,   and applying the Schauder fixed-point theorem yield  the existence of a  non-negative, bounded  weak solution $\p^\ve$ of the  equations   \eqref{sumbal} with  boundary  and initial conditions in \eqref{BC1}, for every fixed~$\ve$.

To show the existence of a weak  solution of the equations  \eqref{sumbal_11} for $(\n^\ve, b^\ve)$, 
we  first consider for a given $\widetilde \n_{2}^\ve  \in L^2(0,T; H^\varsigma(\Omega_{M}^\ve))\cap L^\infty(0,T; L^\infty(\Omega_M^\ve))$ with $\widetilde \n_{2}^\ve(t,x)\geq 0$ for $(t,x) \in \Omega_{M,T}^\ve$, where  $\varsigma \in (1/2, 1)$,
\beqn\label{linear_2}
\begin{cases}
\partial_t \n^\ve =\text{div}(D_n\nabla \n^\ve) + \widetilde \F_n(\p^\ve,\n^\ve)+ \widetilde \R_n (\n^\ve, b^\ve,  \mathcal N_\delta(\be(\bu^\ve)))  & \text{ in } \Omega_{M,T}^\ve,\\
\partial_t b^\ve = \phantom{\text{div}(D_c\nabla n^\ve_{c})+ \widetilde \F_n(\p^\ve,\n^\ve)  }\;  R_{b}(\n^\ve_1, \widetilde \n^\ve_2, b^\ve, \mathcal N_\delta(\be(\bu^\ve))) &   \text{ in } \Omega_{M,T}^\ve,\\
D_n\nabla \n^\ve \, \bnu  = \G(\n^\ve)\,   \mathcal N_\delta(\be(\bu^\ve))   \hspace{ 2 cm }    \text{ on } \Gamma_{\cI, T},  \qquad 
D_n\nabla \n^\ve\, \bnu =\widetilde \J_n(\n^\ve)  &  \text{ on } \Gamma_{\cE, T}, \\
   \n^\ve \hspace{ 4.9 cm }  a_3\text{-periodic in } x_3, \hspace{ 0.8 cm }    D_n \nabla \n^\ve\, \bnu = 0   &   \text{ on }  \Gamma_{\cU,T}\cup \Gamma^\ve_{T}, \\
\n^\ve(0,x) =\n_{0}(x),  \qquad b^\ve(0,x) =b_{0}(x),  &\text{ in } \Omega_M^\ve, 
\end{cases}
\eeqn
where 
$$
\begin{aligned}
& \widetilde \R_n (\n^\ve, b^\ve,  \mathcal N_\delta(\be(\bu^\ve)))=  (R_{n,1} (\n^\ve_1,\widetilde \n_2^\ve,  b^\ve,  \mathcal N_\delta(\be(\bu^\ve))),  R_{n,2} (\n^\ve, b^\ve,  \mathcal N_\delta(\be(\bu^\ve))))^T, \\
& \widetilde \F_n(\p^\ve,\n^\ve) =(F_{n,1}(\p^\ve,\n^\ve_1, \widetilde \n^\ve_2),   F_{n,2}(\p^\ve,\n^\ve))^T, \qquad  \widetilde \J_n(\n^\ve) = (J_{n,1}(\n_1^\ve, \widetilde \n_2^\ve), J_{n,2}(\n^\ve))^T.
\end{aligned}
$$
  Similar to the problem  \eqref{linearised_1},  applying the Galerkin method and using the estimates similar to   \eqref{energy_n_bd_1},  we obtain the existence of a unique weak solution of  \eqref{linear_2}.

To show the non-negativity of $\n_1^\ve$,  $\n_2^\ve$,   and $b^\ve$, we define the reaction terms in the  equations  \eqref{linear_2} by 
\begin{eqnarray*}
&& f_{n,1}(\n^\ve_1, b^\ve)=F_{n,1}(\p^\ve, \n^\ve_1,\widetilde \n^\ve_2)+  R_{n,1} (\n^\ve_1, \widetilde \n^\ve_2, b^\ve, \mathcal N_\delta(\be(\bu^\ve))), \\
&& f_{n,2}(\n^\ve, b^\ve)= F_{n,2}(\p^\ve, \n^\ve)+R_{n,2}(\n^\ve, b^\ve, \mathcal N_\delta(\be(\bu^\ve))),
\\
&& f_b(\n^\ve_1, b^\ve)=R_b(\n^\ve_{1},\widetilde \n^\ve_{2}, b^\ve, \mathcal N_\delta(\be(\bu^\ve))).
\end{eqnarray*}
Using the properties of the functions $\F_n$, $\R_{n}$, $R_b$, and $\mathbb E_M$ and the non-negativity of $\widetilde \n^\ve_2 $  and  $\p_j^\ve$, $j=1,2$,  we obtain that $f_{n,1}$, $f_b$ and  $f_{n,2}$  are Lipschitz continuous in $(-\mu, M)^2$ and 
$(-\mu, M)^3$, respectively,    for some $\mu>0$ and any $0<M<+\infty$ and
\begin{equation*}
f_{n,1}(0, b^\ve) \geq 0,   \qquad   f_b(\n^\ve_1, 0)  \geq 0,  \;\;  \quad 
 f_{n,2}(\n^\ve_1, 0, b^\ve)\geq 0 \quad  \text{ for } \; \n^\ve_1 \geq 0, \; b^\ve \geq 0.
\end{equation*} 
The assumptions on $\J_n$, $\G$, and $\mathbb E_M$ ensure that  the boundary terms  are Lipschitz continuous in $(-\mu, M)^2$ and $(-\mu, M)^3$, respectively.  Moreover,  $J_{n,1}(0, \widetilde \n_2^\ve) \geq 0$, 
  $J_{n,2}(\n_1^\ve, 0) \geq 0$  and   $\mathcal N_\delta({\bf A})\, G_2(\n^\ve_1, 0) \geq 0$ for  all $\n^\ve_1 \geq 0$ and ${\bf A} \in \mathbb R^{3\times 3}$.
  
Applying the Theorem on positive invariant regions \cite[Theorem 2]{Redlinger}, with $K_1(\n^\ve, b^\ve) = - \n_1^\ve$,   $K_2(\n^\ve, b^\ve) =  -\n^\ve_2$,  and $K_3(\n^\ve, b^\ve) = - b^\ve$, 
 and using the non-negativity of the initial data yield $\n_1^\ve(t,x) \geq 0$, $\n_2^\ve(t,x) \geq 0$, and $b^\ve(t,x) \geq 0$   for  $(t,x) \in \Omega_{M,T}^\ve$.

Next, we derive  estimates for  the solutions of \eqref{linear_2}.  Taking $\n^\ve$   as a test function  in the weak formulation of the equation for $\n^\ve$ in  \eqref{linear_2} yields
\begin{equation*}
\begin{aligned}
 &\partial_t \|\n^\ve\|^2_{L^2(\Omega_M^\ve)} +\|\nabla \n^\ve\|^2_{L^2(\Omega_M^\ve)}
\leq  C_1\big[1+  \|g_{2}(\p^\ve)\|^2_{L^2(\Omega_M^\ve)} +  \|\n^\ve\|^2_{L^2(\Omega_M^\ve)}
 + \|\widetilde \n^\ve_2\|^2_{L^2(\Omega_M^\ve)} 
\big]
\\& \qquad \qquad + C_2\big[1+\|\mathcal N_\delta(\be(\bu^\ve))\|_{L^\infty(\Omega)}\big] \big[1+\|\n^\ve\|_{L^2(\Omega_M^\ve)}
+\|\widetilde \n^\ve_2 \|_{L^2(\Omega_M^\ve)} 
+\|b^\ve \|_{L^2(\Omega_M^\ve)}\big] \|\n^\ve\|_{L^2(\Omega_M^\ve)}.
\end{aligned}
\end{equation*}
Notice that the estimate \eqref{estim_u_1} for $\bu^\ve$, and Assumption~\ref{assumptions}.10  ensure 
\begin{equation}\label{bound_N}
\|\mathcal N_\delta(\be(\bu^\ve))(t)\|_{L^\infty(\Omega)} \leq C\delta^{-3/2} \|\bu^\ve \|_{L^\infty(0,T; \cW(\Omega))}\leq C_\delta 
\end{equation}
for a.a.\ $t\in [0,T]$, where the constants $C$ and  $C_\delta$ are independent of $\ve$.  
The boundary integrals  are estimated as
\begin{equation}\label{boundary_nc_1}
\begin{aligned}
& \big|\big\langle J_{n,1}(\n^\ve_1, \widetilde \n_2^\ve),  \n^\ve_{1} \big\rangle_{\Gamma_\cE}\big|  \leq C_\sigma (1+  \|\n^\ve_1\|^2_{L^2(\Omega_M^\ve)})+  \sigma \|\nabla \n^\ve_1\|^2_{L^2(\Omega_M^\ve)}, \\ 
&\big|\big\langle  \mathcal N_\delta(\be(\bu^\ve))\,  G_2(\n^\ve), \n^\ve_{2} \big\rangle_{\Gamma_\cI} \big|+\big|\big\langle J_{n,2}(\n^\ve),  \n^\ve_{2} \big\rangle_{\Gamma_\cE}\big|  \\ 
&\hspace{ 4 cm } \leq  C_\sigma\big( \delta^{-3} \|\bu^\ve \|^2_{L^\infty(0,T; \cW(\Omega))} +  \|\n^\ve\|^2_{L^2(\Omega_M^\ve)}+1 \big) + 
  \sigma \|\nabla \n^\ve\|^2_{L^2(\Omega_M^\ve)},
\end{aligned}
\end{equation}
where $\sigma >0$ is  arbitrary fixed. Here we used the properties of the extension of $\n^\ve$, see Lemma~\ref{lem:extension}, and estimates similar to \eqref{trace_est}.  
Testing   the equation for $b^\ve$ in \eqref{linear_2} with  ${b}^\ve$ and   using  the assumptions on  $R_b$ and    \eqref{bound_N},   yield 
\begin{equation*}\label{estim_nb_2}
\partial_t \|{b}^\ve\|^2_{L^2(\Omega_M^\ve)}\leq C_\delta \big[ 1+\|b^\ve\|^2_{L^2(\Omega_{M}^\ve)} +  \|\n^\ve_1\|^2_{L^2(\Omega_{M}^\ve)} 
+ \|\widetilde \n^\ve_2\|^2_{L^2(\Omega_{M}^\ve)} 
\big]. 
\end{equation*}
Using the boundedness of $\p^\ve$ and the regularity of the  initial data,  and applying the Gronwall inequality imply 
\begin{equation}\label{energy_n_bd_10}
\begin{aligned}
\|b^\ve \|_{L^\infty(0,T; L^2(\Omega_M^\ve))}+\|\n^\ve \|_{L^\infty(0,T; L^2(\Omega_M^\ve))}  +  \|\nabla \n^\ve\|_{L^2(\Omega_{M, T}^\ve)} 
\leq C \big[ 1+\|\widetilde \n^\ve_2\|_{L^2(\Omega_{M,T}^\ve)}\big],   
\end{aligned}
\end{equation}
where $C$ is independent of $\ve$. Considering   $\n^\ve_2$ instead of $\widetilde \n^\ve_2$ in  \eqref{linear_2},   in the same way as \eqref{energy_n_bd_10}, we obtain  
\begin{equation}\label{energy_n_bd_1}
\begin{aligned}
\|b^\ve \|_{L^\infty(0,T; L^2(\Omega_M^\ve))}+\|\n^\ve \|_{L^\infty(0,T; L^2(\Omega_M^\ve))}  +  \|\nabla \n^\ve\|_{L^2(\Omega_{M, T}^\ve)} 
\leq C,  
\end{aligned}
\end{equation}
where $C$ is independent of $\ve$. 

 The estimates for $\n^\ve$  in $L^\infty(0, T; L^2(\Omega_{M}^\ve))$ and $L^2(0,T; H^1(\Omega_M^\ve))$ and for $b^\ve$ in $L^\infty(0,T; L^2(\Omega_M^\ve))$ and the weak formulation of equations \eqref{linear_2} ensure the boundedness of 
$\partial_t \n^\ve$ in $L^2(0,T; \cV(\Omega_M^\ve)^\prime)$ and $\partial_t b^\ve$  in $L^2(\Omega_{M,T}^\ve)$, for every fixed $\ve >0$.

Next we show the boundedness of $\n^\ve$ and $b^\ve$. For each fixed $\ve >0$, we have that  $\n^\ve$ are bounded as  solutions of  reaction-diffusion equations  in a Lipschitz domain $\Omega_M^\ve$ with the reaction terms in $L^\infty(0, T; L^2(\Omega_{M}^\ve))$ (see e.g.\ \cite[Theorem~III.7.1]{Ladyzhenskaya} generalized to Robin boundary conditions).
To show  the boundedness of $\n_1^\ve$ and $\n^\ve_2$  uniform in $\ve$   we use  the iteration  Lemma 3.2 in Alikakos \cite{Alikakos}. 
We  derive the $L^\infty$-estimates considering  $\n^\ve_2$ instead  of $\widetilde \n^\ve_2$ in  \eqref{linear_2}.  The derivation of the $L^\infty$-estimates for $\n^\ve$ and $b^\ve$ with  $\widetilde \n^\ve_2$  in  \eqref{linear_2} follows along the same lines, with the only difference that on the right-hand side of \eqref{estim_nd_infty} and  \eqref{estim_b_infty} we will have additionally  the term  $\|\widetilde \n^\ve_2\|^2_{L^\infty(0,T; L^\infty(\Omega_M^\ve))}$. 
Since    $\n_1^\ve$ and $\n_2^\ve$ are bounded for each fixed $\ve>0$, we have that $|\n^\ve|^{p-2}\n^\ve$, with  $p \geq 2$,  is an admissible test function.  Taking $|\n^\ve|^{p-2}\n^\ve$, with $p=2^\kappa$, $\kappa =1,2,3,\ldots$, as a test function in the weak formulation of the equation for $\n^\ve$ in  \eqref{linear_2} and using \eqref{bound_N},  we obtain 
\begin{equation*}
\begin{aligned}
& \partial_t \|\n^\ve\|^p_{L^p(\Omega_M^\ve)} +2 \frac{p-1}p \|\nabla |\n^\ve|^{\frac p2}\|^2_{L^2(\Omega_M^\ve)}
 \leq C_1^p
 +   p(p-1)\|\n^\ve\|^p_{L^p(\Omega_M^\ve)} \\
& \quad  + C_2\big[1+ \delta^{-\frac{3}2} \|\bu^\ve \|_{L^\infty(0, \tau; \cW(\Omega))}\big]\big[(p-1)\|\n^\ve\|^p_{L^p(\Omega_M^\ve)}
 +  \|b^\ve \|^p_{L^p(\Omega_M^\ve)}+ C_3^p\big] + C_4^p\delta^{-\frac{3p}2} \|\bu^\ve \|^p_{L^\infty(0, \tau; \cW(\Omega))} 
\end{aligned}
\end{equation*}
for $\tau \in (0,T]$. Here, the boundary terms  are estimated by applying  Lemma~\ref{lem:extension} to $|\n^\ve|^{p/2}$  together with   the trace inequality for $H^1$-functions: 
\begin{equation}\label{boundary_nc_2}
\begin{aligned}
&\big| \big \langle  \mathcal N_\delta(\be(\bu^\ve))\,  \G(\n^\ve), |\n^\ve|^{p-2} \n^\ve\big\rangle_{\Gamma_{\cI, \tau}}\big| +\big| \big\langle \J_{n}(\n^\ve),  |\n^\ve|^{p-2}\n^\ve \big  \rangle_{\Gamma_{\cE, \tau}} \big| 
\\
& \leq  (p-1)\Big[C_{\sigma}\|\n^\ve\|^p_{L^p(\Omega_\tau)}  + (\sigma/p^2)\|\nabla |\n^\ve|^{\frac p2}\|^2_{L^2(\Omega_\tau)} \Big] + C_1^p / p
  + (C^p_2/ p) \delta^{-\frac{3p}2} \|\bu^\ve \|^p_{L^\infty(0, \tau; \cW(\Omega))} 
 \\
 & \leq  \mu_1 (p-1)\Big[C_{\sigma} \|\n^\ve\|^p_{L^p(\Omega_{M, \tau}^\ve)}  
 + (\sigma/ p^2)\|\nabla |\n^\ve|^{\frac p2}\|^2_{L^2(\Omega_{M, \tau}^\ve)} \Big] + (C^p/ p) \Big[ \delta^{-\frac{3p}2} \|\bu^\ve \|^p_{L^\infty(0,\tau;\cW(\Omega))} +1 \Big],
 \end{aligned}
\end{equation}
where $\sigma>0$ is arbitrary and $\tau \in (0,T]$, and the constants $C$, $C_1$, $C_2$, $C_\sigma$, and $\mu_1$ are independent of $\ve$. Applying the extension   Lemma~\ref{lem:extension} to $|\n^\ve|^{p/2}$  and using   the Gagliardo--Nirenberg  inequality \cite{Brezis}  imply 
 \begin{equation*}
\begin{aligned}
 \|\n^\ve\|^p_{L^p(\Omega_M^\ve)} \leq   \||\n^\ve|^{\frac p 2}\|^2_{L^2(\Omega)} 
&\leq  \frac{ \sigma}{p^2}\|\nabla |\n^\ve|^{\frac p2}\|^2_{L^2(\Omega)} +
  C_{\sigma} p^{3} \||\n^\ve|^{\frac p 2}\|^2_{L^{1}(\Omega)}
 \\& \leq  \mu_1 \Big[\frac{\sigma}{p^2}\|\nabla |\n^\ve|^{\frac p 2}\|^2_{L^2(\Omega_M^\ve)} +
   C_{\sigma} p^{3} \||\n^\ve|^{\frac p 2}\|^2_{L^{1}(\Omega_M^\ve)} \Big],
\end{aligned}
\end{equation*}
where $\sigma>0$ is arbitrary,  the  constant   $C_{\sigma}$ is independent of $\ve$, and $\mu_1$  is as in Lemma~\ref{lem:extension}. Thus we obtain 
\begin{equation*}
\begin{aligned}
\| \n^\ve(\tau) \|^p_{L^p(\Omega_M^\ve)}+  \frac{p-1}p\, \|\nabla |\n^\ve|^{\frac p2}\|^2_{L^2(\Omega_{M, \tau}^\ve)}
&\leq  p^5 C_1\Big[\sup\limits_{(0,\tau)}\| \n^\ve\|^{\frac p2}_{L^{\frac p2}(\Omega_M^\ve)}\Big]^{2}\\
& + C_2^p \big[\delta^{-\frac{3p}2} \|\bu^\ve\|^p_{L^\infty(0, \tau; \cW(\Omega))}+   \|b^\ve\|^p_{L^\infty(0,\tau; L^p(\Omega))} +1\big]
\end{aligned}
\end{equation*}
for $\tau \in (0,T]$.
Then, using similar recursive iterations as in \cite[Lemma 3.2]{Alikakos}, we obtain 
\begin{equation*}
\| \n^\ve(\tau)\|^p_{L^p(\Omega_M^\ve)} \leq C^{p} 2^{2(p-1)} 2^{10p} [ 1+ \|\bu^\ve\|_{L^\infty(0,\tau; \mathcal W(\Omega))}^p + \|b^\ve\|^p_{L^\infty(0,\tau; L^p(\Omega))}]
\end{equation*}
for  $\tau \in (0,T]$ and $C \geq 1$. Applying 
 the $p$th root, and taking  $p \to \infty$,  yield 
 \begin{equation}\label{estim_nd_infty}
 \|\n^\ve\|^2_{L^\infty(0,\tau; L^\infty(\Omega_{M}^\ve))} \leq  C_1\big[1 + \|b^\ve\|^2_{L^\infty(0,\tau; L^\infty(\Omega_{M}^\ve))}\big], 
 \end{equation} 
for all $\tau \in (0,T]$  and $C_1$ is independent of $\ve$.
Multiplying   the equation for $b^\ve$ in \eqref{linear_2} with $b^\ve$, integrating over $(0,\tau)$, using the assumptions  on $R_b$, and  considering the supremum over $\Omega_{M}^\ve$ give  
\begin{equation}\label{estim_b_infty}
\begin{aligned}
\|b^\ve(\tau)\|^2_{L^\infty(\Omega_M^\ve)} & \leq  \|b_0\|^2_{L^\infty(\Omega_M^\ve)} +
\tau\,  C_2\big[1+\|\n^\ve\|^2_{L^\infty(0,\tau; L^\infty(\Omega_{M}^\ve))} + \|b^\ve\|^2_{L^\infty(0,\tau; L^\infty(\Omega_{M}^\ve))}\big]
\end{aligned}
\end{equation}
for  $\tau \in (0,T]$ and $C_2$ is independent of $\ve$. Using   \eqref{estim_nd_infty} and  iterating over time intervals of length $1/ (2C_2(C_1+1))$ yield  the estimates for  $b^\ve$ and, hence,  for $\n^\ve$ in $L^\infty(0,T; L^\infty(\Omega_{M}^\ve))$,  independent of $\ve$. 

 The  boundedness of $\mathcal N_\delta(\be(\bu^\ve))$, $\n^\ve$  and $b^\ve$   ensures the estimate for   $\|\partial_t  b^\ve\|_{L^\infty(0,T; L^\infty(\Omega_{M}^\ve))}$, independent of $\ve$.

 To show the existence of a weak solution $(\n^\ve, b^\ve)$ of the  equations  \eqref{sumbal_11} with  the boundary and initial conditions in \eqref{BC1}, we consider
 the operator  $\mathcal K_2: X
 \to X$ defined by $ \n^\ve_{2} =
\mathcal K_2(\widetilde \n^\ve_{2})$, where  $\n^\ve_{2}$ solves  the problem \eqref{linear_2} and $X=\{ n \in L^2(0,T; H^\varsigma(\Omega_M^\ve))\cap L^\infty(0,T; L^\infty(\Omega_M^\ve))\; | \;   n(t,x) \geq 0 \; \text{ for } (t,x) \in \Omega_{M, T}^\ve \}$, with $\varsigma\in (1/2, 1)$.  The continuity of  $\mathcal K_2$ is ensured by the continuity of   $\F_n$,  $\R_{n}$, $R_b$, $\J_n$,  and $\G$,   the \textit{a priori} estimates for $\n^\ve$ and $b^\ve$, the compact embedding of $L^2(0,T; H^1(\Omega_M^\ve)) \cap H^1(0,T; \cV(\Omega_M^\ve)^\prime)$ in $L^2(0, T; H^\varsigma(\Omega_{M}^\ve))$,  for $\varsigma < 1$, and the estimate 
\begin{equation}\label{estim_nb_11}
\sup_{(0,T)} \|b^{\ve,1}- b^{\ve,2}\|_{L^2(\Omega_M^\ve)} \leq C_\delta\big[ \big \|\n^{\ve,1}- \n^{\ve,2}\|_{L^2(\Omega_{M,T}^\ve)}+ \big \|\widetilde \n^{\ve,1}_2- \widetilde \n^{\ve,2}_2\|_{L^2(\Omega_{M,T}^\ve)}\big].
\end{equation}
The estimate \eqref{estim_nb_11} is obtained by considering the difference of equation \eqref{linear_2}  for $b^{\ve,1}$ and $b^{\ve,2}$, testing by  $b^{\ve,1}- b^{\ve,1}$, and using the properties of  $R_b$.  
Then applying the Schaefer fixed-point theorem and the compact embedding of $L^2(0,T; H^1(\Omega_M^\ve)) \cap H^1(0,T; \cV(\Omega_M^\ve)^\prime)$ in $L^2(0,T; H^\varsigma(\Omega_{M}^\ve))$, with $\varsigma \in (1/2,1)$,  yields  the existence of a fixed point of $\mathcal K_2$.

Hence, combining this result  with the existence result for $\p^\ve$,  ensures  the existence of a  weak solution of  \eqref{sumbal}--\eqref{BC1}.
Considering the equations for the difference of two solutions $\p^{\ve, 1}-\p^{\ve, 2}$,  $\n^{\ve, 1}-\n^{\ve, 2}$,  and $b^{\ve, 1}-b^{\ve, 2}$,   and using the uniform boundedness of $\p_j^{\ve,l}$, $\n_j^{\ve,l}$ and $b^{\ve, l}$, with $j=1,2$ and $l=1,2$, we obtain the uniqueness of a weak solution   of the  problem \eqref{sumbal}--\eqref{BC1} for a given $\bu^\ve \in L^\infty(0,T; \cW(\Omega))$. 
\end{proof}

\noindent\textbf{Remark.} The proof of Theorem~\ref{th:exist_1} follows  along the same lines if   $\J_p$  is a function  of  $\int_{\Omega_M^\ve} \p^\ve dx$ instead of $\p^\ve$.

\subsection{Existence of a unique  solution of the coupled system  {\rm \eqref{sumbal}--\eqref{sumbal2}}. Proof of  Theorem~\ref{existence_2}.}
Considering the estimates in Lemma~\ref{lemub},  to prove the well-posedness of  the coupled system we shall derive   estimates for  $\|\widetilde b^{\ve,j}\|_{L^\infty(0,T; L^\infty(\Omega_{M}^\ve))} $ in terms of  $\|\be(\widetilde \bu^{\ve,j})\|_{L^\infty(0,T; L^2(\Omega))}$, where $\widetilde b^{\ve,j} = b^{\ve, j}- b^{\ve, j+1}$ and 
$ \widetilde \bu^{\ve,j}= \bu^{\ve, j}-  \bu^{\ve, j+1}$ are the differences of  two fixed-point  iterations.
\begin{proof}[\bf Proof of Theorem~\ref{existence_2}.] 
We prove the existence of a unique weak solution of the coupled system by applying a contraction argument.
  We define the operator $\mathcal K: L^\infty(0, T; L^\infty(\Omega_{M}^\ve))\to   L^\infty(0, T; L^\infty(\Omega_{M}^\ve))$ by $\mathcal K(b^{\ve, j-1})=b^{\ve, j}$,
where $b^{\ve, j}$ is a solution of the system \eqref{sumbal}--\eqref{sumbal2} with $b^\ve$   in \eqref{sumbal2} replaced by  $b^{\ve, j-1}$ and with $\bu^\ve$ in the equations \eqref{sumbal_11}  and the boundary conditions in  \eqref{BC1} replaced by $\bu^{\ve, j}$.

For a given non-negative $b^{\ve, 1}\in L^\infty(0, T; L^\infty(\Omega_{M}^\ve))$, satisfying the initial condition in \eqref{BC1},  by Lemma~\ref{lemub} there exists a unique $\bu^{\ve,2}\in L^\infty(0,T;\cW(\Omega))$ satisfying \eqref{sumbal2}, with $b^\ve$ replaced by $b^{\ve, 1}$.  Then for  $\bu^{\ve,2}\in L^\infty(0,T;\cW(\Omega))$, by Theorem~\ref{th:exist_1} there are unique $\p^{\ve, 2}, \n^{\ve, 2}\in (L^2(0,T; \cV(\Omega_{M}^\ve))\cap   L^\infty(0,T; L^\infty(\Omega_{M}^\ve)))^2$, $b^{\ve, 2}\in W^{1, \infty}(0,T; L^\infty(\Omega_{M}^\ve))$ satisfying \eqref{sumbal}--\eqref{BC1}, and $\p^{\ve, 2}_l, \n^{\ve, 2}_l$, $b^{\ve, 2}$ are non-negative, with $l=1,2$.  Iterating for $j=3,4,\ldots$, we obtain $(\p^{\ve, j}, \n^{\ve, j}, b^{\ve, j}, \bu^{\ve,j})$,   for $j\geq 3$.

For each $j\geq 2$, similar to \eqref{estim_u_21} and \eqref{apriori_estim}, we obtain  {\it a priori} estimates for  $\bu^{\ve,j}$ in $L^\infty(0,T;\cW(\Omega))$,  for $\p^{\ve,j}, \n^{\ve, j}$ in $\big(L^2(0, T; H^1(\Omega_M^\ve)) \cap L^\infty(0,T; L^\infty(\Omega_{M}^\ve))\big)^2$  and for $b^{\ve,j}$  in $W^{1,\infty}(0,T; L^\infty(\Omega_{M}^\ve))$, independently of $ b^{\ve,j-1}$.

To derive a contraction inequality in $ L^\infty(0,  T; L^\infty(\Omega_{M}^\ve))$ we first  take $\boldsymbol{\phi}_n= |\widetilde \n^{\ve, j}|^{p-2} \widetilde \n^{\ve, j}$, where $p =2^\kappa$, with  $\kappa=1,2,3,\ldots$, and $\widetilde \n^{\ve, j}= \n^{\ve, j}- \n^{\ve, j+1}$,  as a test function in the  difference of the equations for $\n^{\ve,j}$ and $\n^{\ve,j+1}$.
For the boundary integrals in the equations for $\widetilde \n^{\ve,j}$ we have, using  the trace inequality, 
\begin{equation*}
\begin{aligned}
& \big \langle \mathcal N_\delta(\be(\bu^{\ve,j})) \big[\G(\n^{\ve,j}) - \G(\n^{\ve,j+1})\big],  |\widetilde \n^{\ve, j}|^{p-2} \widetilde \n^{\ve, j}\big \rangle_{\Gamma_\cI} \leq 0 ,\\
&   \big\langle \J_n(\n^{\ve,j})-  \J_n(\n^{\ve,j+1}),  |\widetilde \n^{\ve, j}|^{p-2} \widetilde \n^{\ve, j} \big\rangle_{\Gamma_\cE} 
 \leq   C_\sigma  p\,  \|\widetilde \n^{\ve,j} \|^p_{L^p(\Omega_M^\ve)} 
+ \sigma (p-1)/ p^{2}\,  \|\nabla|\widetilde \n^{\ve,j}|^{\frac p2}\|^2_{L^2(\Omega_M^\ve)} 
 \end{aligned}
\end{equation*}
 and 
\begin{equation*}
\begin{aligned}
\Big|\big\langle \G(\n^{\ve,j+1})\big[\mathcal N_\delta(\be(\bu^{\ve, j})) - \mathcal N_\delta(\be(\bu^{\ve, j+1})) \big], |\widetilde \n^{\ve, j}|^{p-2} \widetilde \n^{\ve, j} \big\rangle_{\Gamma_\cI} \Big| \leq   C_\sigma  (p-1) \|\widetilde \n^{\ve,j}_2 \|^p_{L^p(\Omega_M^\ve)} \qquad  \\
+ \sigma (p-1)/ p^{2}\,  \|\nabla|\widetilde \n^{\ve,j}_2|^{\frac p2}\|^2_{L^2(\Omega_M^\ve)}  + (C/p) \, \|\mathcal N_\delta(\be(\bu^{\ve, j}))-\mathcal N_\delta(\be( \bu^{\ve, j+1}))\|^p_{L^p(\Omega)}, 
\end{aligned}
\end{equation*}
with an arbitrary $\sigma>0$. Then, the uniform boundedness of $\n^{\ve, j}$ and $b^{\ve, j}$ and the Gagliardo--Nirenberg inequality  applied to  $|\widetilde \n^{\ve, j}|^{p/2}$   ensure
\begin{equation*}
\begin{aligned}
\partial_t \|\widetilde \n^{\ve, j} \|^p_{L^p(\Omega_M^\ve)} + 
 2 \frac {p-1} p\|\nabla |\widetilde \n^{\ve, j}|^{\frac p 2}\|^2_{L^2(\Omega_M^\ve)}
\leq  C \big[1+ \|\mathcal N_\delta(\be(\bu^{\ve, j}))\|_{L^\infty(\Omega)} \big]\Big[    
p^5   \|\widetilde \n^{\ve, j} \|^p_{L^{\frac p 2}(\Omega_M^\ve)}
\\ +  \|\widetilde b^{\ve, j}\|^p_{L^p(\Omega_M^\ve)}+\|\mathcal N_\delta(\be(\bu^{\ve, j}))-\cN_\delta(\be(\bu^{\ve, j+1}))\|^p_{L^p(\Omega)}\Big].
 \end{aligned}
\end{equation*}
Considering   iterations in $p$ as in    \cite[Lemma~3.2]{Alikakos} with $p = 2^\kappa$ and  $\kappa=2,3,\ldots,$  we obtain 
\begin{equation*}
 \|\widetilde \n^{\ve,j}(\tau) \|^p_{L^p(\Omega_M^\ve)}
  \leq  C^p_\delta 2^{10 p}  2^{2(p-1)}  \big[ \| \be(\widetilde \bu^{\ve, j})\|^p_{L^\infty(0,\tau;L^2(\Omega))} +\|\widetilde  b^{\ve, j} \|^p_{L^\infty(0, \tau; L^p(\Omega_{M}^\ve))} \big]
\end{equation*}
for $\tau \in (0, T]$ and $C_\delta \geq 1$.  
Here we also used   the estimate 
\begin{equation*}
 \|\mathcal N_\delta(\be(\bu^{\ve, j})) -\cN_\delta(\be(\bu^{\ve, j+1}))\|^p_{L^p(\Omega)} \leq C^p   \delta^{-\frac {3p}2} \big[\| \be(\widetilde \bu^{\ve, j})\|^p_{L^2(\Omega)} +  \|\widetilde b^{\ve, j} \|^p_{L^p(\Omega_M^\ve)}  \big].
\end{equation*}
Taking  the $p$th  root,  and considering $p \to \infty$ yield
\begin{eqnarray}\label{estim_ndc12}
\begin{aligned}
\|\widetilde \n^{\ve,j}\|_{L^\infty(0, \tau; L^\infty(\Omega_{M}^\ve))} 
\leq C_\delta \big[ \|\be(\widetilde \bu^{\ve,j})\|_{L^\infty(0,\tau; L^2(\Omega))} + \|\widetilde b^{\ve,j} \|_{ L^\infty(0, \tau; L^\infty(\Omega_{M}^\ve))} \big].
\end{aligned}
\end{eqnarray}
Consider the difference of  equations \eqref{weak_sol_n2} for  two iterations $b^{\ve,j}$ and $b^{\ve,j+1}$, and multiply by  $\phi_b =\widetilde  b^{\ve,j}$   to obtain 
\begin{equation*}
\begin{aligned}
  \|\widetilde b^{\ve,j}(\tau)\|^2_{L^\infty(\Omega^\ve_M)}
  \leq C_\delta \int_0^\tau \left[  \| \mathcal N_\delta(\be(\bu^{\ve, j})) -\cN_\delta(\be(\bu^{\ve, j+1}))\|^2_{L^\infty(\Omega_M^\ve)}  +
  \|\widetilde \n^{\ve,j}\|^2_{L^\infty(\Omega^\ve_M)}+   \|\widetilde b^{\ve,j}\|^2_{L^\infty(\Omega^\ve_M)}  \right] d\tau.
\end{aligned}
\end{equation*}
Using the estimate   \eqref{estim_ndc12}, the definition of $\mathcal N_\delta$, and the boundedness of $b^{\ve,j}$ yields 
\begin{equation*}
  \|\widetilde b^{\ve,j}\|^2_{L^\infty(0, \tau; L^\infty(\Omega^\ve_{M}))}
  \leq C_\delta \tau  \big[ \|\be(\widetilde \bu^{\ve, j})\|_{L^\infty(0,\tau; L^2(\Omega))}^2 + \|\widetilde  b^{\ve,j} \|^2_{ L^\infty(0, \tau; L^\infty(\Omega_{M}^\ve))} \big]
\end{equation*}
for  $\tau \in (0,T]$. Then, iterating over time intervals  of length $1/ (2 C_\delta)$,   ensures
$$
\|\widetilde b^{\ve,j}\|^2_{L^\infty(0,\tilde T; L^\infty( \Omega_{M}^\ve))} \leq C \tilde T \|\be(\widetilde \bu^{\ve,j})\|_{L^\infty(0,\tilde T; L^2(\Omega))}^2
$$
for all $\tilde T \in (0, T]$ and   the constant $C$  independent of $\tilde T$, $b^{\ve,1}$ and  $(\p^{\ve, j}, \n^{\ve, j}, b^{\ve, j},  \bu^{\ve,j})$ for all  $j\geq 2$.  Estimate \eqref{ub} yields
\begin{equation}\label{estim_conts_m_1_1}
\|\be(\widetilde \bu^{\ve,j})\|_{L^\infty(0,\tilde T; L^2(\Omega))}^2 \leq C \|\widetilde b^{\ve,j-1}\|^2_{L^\infty(0,\tilde T; L^\infty(\Omega_{M}^\ve))}.
\end{equation}
The  last two inequalities ensure  that for fixed $\delta$ and $\tilde T$ sufficiently small we have the  contraction inequality for  the operator  $\mathcal K$.  Thus,  the same arguments as in the proof of the Banach fixed-point theorem yield that  $\mathcal K$ has a unique fixed point.  Hence, there exists a unique  weak solution of \eqref{sumbal}--\eqref{sumbal2} in $(0,\tilde T)\times \Omega$.  Since $\tilde T$ depends only on the  model parameters,  iterating over time intervals yields the existence of a unique weak solution in $(0,T)\times\Omega$. The {\it a priori} estimates \eqref{estim_u_21},  \eqref{apriori_estim}, together with  \eqref{estim_time},  shown in Lemma~\ref{estim_deriv} below,  imply the estimates~\eqref{estim_un_1}.
\end{proof}

\begin{lemma}\label{estim_deriv}
Under Assumption~\ref{assumptions},   weak solutions of  \eqref{sumbal}--\eqref{sumbal2} satisfy 
\begin{equation}\label{estim_time}
\begin{aligned}
&\|\partial_t \bu^\ve\|_{L^\infty(0,T; \cW(\Omega))}\leq C, \\
&\|\vartheta_h \p^\ve - \p^\ve\|_{L^2( \Omega_{M, T-h}^\ve)}+ \|\vartheta_h \n^\ve - \n^\ve\|_{L^2(\Omega_{M, T-h}^\ve)} \leq Ch^{1/4}, 
\end{aligned}
\end{equation}
for any $h>0$, where $\vartheta_h v(t,x) = v(t+h, x)$ for $(t,x) \in(0, T-h]\times  \Omega_{M}^\ve$ and the constant $C$ is  independent of $\ve$. 
\end{lemma}
\begin{proof} 
Differentiating the equations of linear elasticity \eqref{sumbal2} with respect to time $t$, testing it with $\partial_t \bu^\ve$, and using the uniform boundedness of $\partial_t b^\ve$ imply the estimate for $\|\partial_t \be(\bu^\ve)\|_{L^\infty(0,T; L^2(\Omega))}$. Applying the second Korn inequality we obtain the estimate for $\partial_t \bu^\ve$ in $L^\infty(0,T; \cW(\Omega))$. 

To show the estimates for  $\vartheta_h \p^\ve - \p^\ve$ and $\vartheta_h \n^\ve - \n^\ve$ 
we integrate the equations for  $\p^\ve$ and $\n^\ve$ in \eqref{sumbal} and  \eqref{sumbal_11} over $(t, t+h)$ and consider $\vartheta_h \p^\ve - \p^\ve$  and $\vartheta_h \n^\ve - \n^\ve$  as   test functions, respectively,  
 \begin{equation*}
\begin{aligned}
& \| \vartheta_h \p^\ve - \p^\ve \|^2_{L^2( \Omega_{M, \tau}^\ve)} \leq
\Big| \Big\langle D_p\int_t^{t+h} \nabla \p^\ve\,  ds,  \vartheta_h\nabla \p^\ve -\nabla \p^\ve \Big \rangle_{\Omega_{M, \tau}^\ve} 
+ \Big\langle \int_t^{t+h} \F_p(\p^\ve)\, ds,  \vartheta_h \p^\ve - \p^\ve \Big \rangle_{\Omega_{M, \tau}^\ve} \Big|
\\
&+\Big| \Big\langle \int_t^{t+h} \J_p(\p^\ve )\, ds,  \vartheta_h \p^\ve - \p^\ve \Big \rangle_{\Gamma_{\cI, \tau}^\ve} 
- \Big\langle \gamma_p \int_t^{t+h}  \p^\ve \, ds,  \vartheta_h \p^\ve - \p^\ve \Big \rangle_{\Gamma_{\cE, \tau}^\ve} \Big|
\end{aligned}
\end{equation*}
and 
 \begin{equation*}
\begin{aligned}
& \| \vartheta_h \n^\ve - \n^\ve \|^2_{L^2( \Omega_{M, \tau}^\ve)} \leq 
\Big| \Big\langle D_n\int_t^{t+h}\hspace{-0.35 cm }  \nabla \n^\ve   ds,  \vartheta_h\nabla \n^\ve -\nabla \n^\ve \Big \rangle_{\Omega_{M, \tau}^\ve}
\hspace{-0.3 cm } + \Big \langle \int_t^{t+h} \hspace{-0.3 cm }   \G(\n^\ve)  \mathcal N_\delta(\be(\bu^\ve)) ds, \vartheta_h \n^\ve - \n^\ve \Big\rangle_{\Gamma_{\cI, \tau}} \Big|\\ 
&  + 
 \Big| \Big \langle \int_t^{t+h} \hspace{-0.25 cm } \big[\F_n(\p^\ve,  \n^\ve) +  \R_n(\n^\ve, b^\ve, \mathcal N_\delta(\be(\bu^\ve)))\big]ds,  \vartheta_h \n^\ve- \n^\ve \Big\rangle_{\Omega_{M, \tau}^\ve}\Big| 
 + \Big|  \Big\langle \int_t^{t+h} \hspace{-0.25 cm }   \J_n (\n^\ve) ds, \vartheta_h \n^\ve - \n^\ve \Big\rangle_{\Gamma_{\cE, \tau}} \Big|
 \end{aligned}
\end{equation*}
for all $\tau \in (0, T-h]$ and any $h>0$. Notice that  $\p^\ve, \n^\ve \in \big(L^2(0,T; H^1(\Omega_M^\ve)) \cap H^1(0,T; \cV(\Omega_M^\ve)^\prime)\big)^2$ for every fixed $\ve>0$. Then  the boundedness of $\p^\ve_j$, $\n^\ve_j$, and $b^\ve$,  with $j=1,2$, the  estimates for $\mathcal N_\delta(\be(\bu^\ve))$ in \eqref{bound_N} and for $\p^\ve$  and $\n^\ve$ in \eqref{apriori_estim},  together with the H\"older inequality,  imply the estimates for $\p^\ve(t+h,x) - \p^\ve(t,x)$ and  $\n^\ve(t+h,x) - \n^\ve(t,x)$,   stated in  the Lemma.  
\end{proof}

\section{\textit{A priori} estimates and existence and uniqueness  results  for the microscopic Model~II.}\label{exist_uniq_ModelII}
For the  equations of linear elasticity \eqref{sumbal2}  we have the same results as in Lemma~\ref{lemub}. 
The main difference in the proof of the well-posedness result for  Model II is in the derivation of   \textit{a priori} estimates for $\n^\ve$ and $b^\ve$.  

\subsection{Existence of a unique weak solution of the problem {\rm \eqref{sumbal}, \eqref{BC1}, \eqref{sumbal_1}, \eqref{BC4}}  for a given $\bu^\ve$.}
\begin{theorem}\label{th:exist_4}
Under Assumption~\ref{assumptions} and for $\bu^\ve \in L^\infty(0,T; \cW(\Omega))$ such that 
\begin{equation}\label{estim_u_2}
\| \bu^\ve\|_{L^\infty(0,T;\cW(\Omega))} \leq C,
\end{equation}
where the constant $C$ is  independent of  $\ve$, 
 there exists a unique  non-negative weak solution $(\p^\ve, \n^\ve, b^\ve)$ of the microscopic problem  {\rm \eqref{sumbal}, \eqref{BC1}, \eqref{sumbal_1}, \eqref{BC4}} satisfying   the \textit{a priori} estimates
\eqref{estim_un_3}.
\end{theorem}
\begin{proof}
The equations for $\p^\ve$  in  both microscopic problems,  Model I and Model II, are the same. Thus the proof of the existence and     uniqueness   and  the derivation of the  \textit{a priori} estimates  for  solutions of  the subsystem for $\p^\ve$  follows the same lines as in the  proof  of Theorem~\ref{th:exist_1}. 

The current proof differs from that  of Theorem~\ref{th:exist_1}   in the derivation of the {\it a priori} estimates for $\n^\ve$ and $b^\ve$, since now the reaction terms in  equations \eqref{sumbal_1}   depend on $\bbE^\ve(b^\ve, x)\be(\bu^\ve)$ and not on its  local  average. 
Similar to the proof of Theorem~\ref{th:exist_1}, to show the existence of a weak solution of  \eqref{sumbal_1},  with the initial and boundary conditions in  \eqref{BC1} and \eqref{BC4}, we apply a fixed-point argument   and  consider  $F_{n,1} (\p^\ve, \n_1^\ve, \widetilde \n_2^\ve)$ instead of $F_{n,1} (\p^\ve, \n^\ve)$ and $\Q_{n,1}(\n^\ve_1, \widetilde \n^\ve_2, b^\ve, \be(\bu^\ve))$ instead of $\Q_{n,1}(\n^\ve, b^\ve, \be(\bu^\ve))$  in the equation  for $\n_1^\ve$, as well as  $Q_{b}(\n^\ve_1, \widetilde \n^\ve_2, b^\ve, \be(\bu^\ve))$ instead of $Q_{b}(\n^\ve, b^\ve, \be(\bu^\ve))$  in the equation for $b^\ve$ and
$J_{n,1}(\n^\ve_1, \widetilde \n_2^\ve)$ instead of $J_{n,1}(\n^\ve)$ in the boundary conditions for $\n_1^\ve$,  for a given $ \widetilde \n_2^\ve  \in L^2(0,T; H^\varsigma(\Omega_M^\ve))\cap L^\infty(0,T; L^\infty(\Omega_M^\ve))$ with $ \widetilde \n_2^\ve \geq 0$ and $\varsigma \in (1/2, 1)$.  Notice that due to the assumptions on $\F_n$, $\Q_n$,  $Q_b$,  and $\J_n$,  the derivation of the {\it a priori} estimates     follows along  the same lines for     $\n^\ve_1, \widetilde \n_2^\ve \geq 0$ and $\n^\ve_1, \n_2^\ve \geq 0$. 

By applying the theory of invariant regions \cite[Theorem 2]{Redlinger} and \cite[Theorem 14.7]{Smoller},  we obtain the non-negativity of  $\n_j^\ve$,  $j=1,2$, and   $b^\ve$  in the same way as in the proof of  Theorem~\ref{th:exist_1}.
Taking $\n^\ve$  and $b^\ve$ as  test functions in \eqref{weak_sol_n3}  and using the non-negativity of   $\n^\ve_j$ and $b^\ve$ and  the boundedness of $\p^\ve$,  along with  the assumptions on  $\F_n$, $\Q_{n}$,  $Q_b$,  and  $\J_n$,  see Assumption~\ref{assumptions}, yield
\begin{equation}\label{estim_L2_21}
\begin{aligned}
 \Big[\|\n^\ve(\tau)\|^2_{L^2(\Omega_{M}^\ve)} + \|b^\ve(\tau)\|^2_{L^2(\Omega_{M}^\ve)} + \|\nabla \n^\ve\|^2_{L^2(\Omega_{M, \tau}^\ve)}+ \|\nabla b^\ve\|^2_{L^2(\Omega_{M, \tau}^\ve)}
\Big] \qquad \qquad \qquad 
\\  \leq C_1\big[ \|\n^\ve\|^2_{L^2(\Omega_{M,\tau}^\ve)}+ \|b^\ve\|^2_{L^2(\Omega_{M,\tau}^\ve)}\big]
 +C_2\big[ 1+\delta^{-3}\|\be(\bu^\ve)\|^2_{L^\infty(0,\tau; L^2(\Omega))} \big] \\ +  
C_3\|\be(\bu^\ve)\|_{L^\infty(0,\tau; L^2(\Omega_M^\ve))}\big[  \|\n^\ve\|^2_{L^2(0,\tau; L^4(\Omega_{M}^\ve))}  +  \|b^\ve\|^2_{L^2(0,\tau; L^4(\Omega_{M}^\ve))}\big]
\end{aligned}
\end{equation}
for  $\tau \in (0, T]$ and the constants $C_1$ and $C_2$ independent of $\ve$.   The boundary integrals for $\n^\ve$ are estimated in the same way as in the proof of Theorem~\ref{th:exist_1}, see  estimate \eqref{boundary_nc_1}. 
Considering  the properties of the extension of $\n^\ve$ and $b^\ve$ in Lemma~\ref{lem:extension},   applying the Gagliardo--Nirenberg inequality to estimate  $\|\n^\ve\|_{L^4(\Omega)}$ and $\|b^\ve\|_{L^4(\Omega)}$,   taking into account the estimate \eqref{estim_u_2},  and using  the Gronwall inequality  imply
 \begin{equation}\label{estim_L2_nd}
\|\n^\ve\|_{L^\infty(0, T; L^2(\Omega_{M}^\ve))} + \|\nabla \n^\ve\|_{L^2(\Omega_{M, T}^\ve)} + \|b^\ve\|_{L^\infty(0, T; L^2(\Omega_{M}^\ve))} + \|\nabla b^\ve\|_{L^2(\Omega_{M, T}^\ve)}\leq C, 
\end{equation}
where the constant $C$ is independent of $\ve$. 

The properties of   $\F_{n}$, $\Q_n$,  $\J_n$, and $Q_b$ and  the estimates for $\n^\ve$, $b^\ve$,  and   $\|\be(\bu^\ve)\|_{L^\infty(0,T; L^2(\Omega))}$ yield that for each fixed $\ve>0$ the functions $\n^\ve$ and $b^\ve$ are bounded, see e.g.\ \cite[Theorem III.7.1]{Ladyzhenskaya} generalized to Robin boundary conditions, and, hence,    $(b^\ve)^{p-1}$ and $|\n^\ve|^{p-2} \n^\ve$,  with $p \geq 2$, are admissible test functions in \eqref{weak_sol_n3}.  
Considering $|b^\ve|^{p-1}$ as a test function in the equation for $b^\ve$ in \eqref{weak_sol_n3}, using the assumptions on $Q_b$ and the non-negativity of $b^\ve$ yield 
\begin{equation*}
\begin{aligned}
 \|b^\ve(\tau)\|^p_{L^p(\Omega_{M}^\ve)} + 4 \frac{p-1} p \|\nabla |b^\ve|^{\frac p 2}\|^2_{L^2(\Omega_{M, \tau}^\ve)} \leq 
C^p_1\big[1+ \|\n^\ve\|^p_{L^\infty(0,\tau; L^2(\Omega_{M}^\ve))} \big] 
 + 
(p-1)\||b^\ve|^{\frac p 2}\|^2_{L^2(0,\tau; L^4(\Omega_{M}^\ve))} \\
+  C_2\|\be(\bu^\ve)\|_{L^\infty(0, \tau; L^2(\Omega_M^\ve))} \big[ (p-1)\||b^\ve|^{\frac p 2}\|^2_{L^2(0,\tau; L^4(\Omega_{M}^\ve))} + \||\n^\ve|^{\frac p 2}\|^2_{L^2(0,\tau; L^4(\Omega_{M}^\ve))} + C^p\big].
\end{aligned}
\end{equation*}
In a similar way, using the boundedness of $\p^\ve$ and the assumptions on $\F_n$, $\Q_n$,  $\G$,  and $\J_n$,  see Assumption~\ref{assumptions},  and considering $|\n^\ve|^{p-2}\n^\ve$  as  a test function in the equations for $\n^\ve$ in   \eqref{weak_sol_n3}   yield 
 \begin{equation*}
 \begin{aligned}
\|\n^\ve(\tau) \|^p_{ L^p(\Omega_{M}^\ve)} +4 \frac {p-1} p\|\nabla |\n^\ve|^{\frac p 2}\|^2_{L^2(\Omega_{M, \tau}^\ve)}\leq 
 C_1^p\big[(1+ \delta^{-\frac {3p} 2} )   \| \be(\bu^\ve)\|^p_{L^\infty(0,\tau; L^2(\Omega))} 
 + \| b^\ve\|^p_{L^\infty(0,\tau; L^2(\Omega_M^\ve))}\big] \\ + C_2^p 
    + \frac {p-1} p    \|\nabla |\n^\ve|^{\frac p 2}\|^2_{L^2(\Omega_{M, \tau}^\ve)} +C_3  \big[p \, \||\n^\ve|^{\frac p 2}\|^2_{L^2(\Omega_{M, \tau}^\ve)}
+  (p-1)\||\n^\ve|^{\frac p 2}\|^2_{L^2(0,\tau; L^4(\Omega_{M}^\ve))}\big]
  \\
  + C_4\|\be(\bu^\ve)\|_{L^\infty(0, \tau; L^2(\Omega_M^\ve))} \big[ (p-1)\||\n^\ve|^{\frac p 2}\|^2_{L^2(0,\tau; L^4(\Omega_{M}^\ve))} + \||b^\ve|^{\frac p 2}\|^2_{L^2(0,\tau; L^4(\Omega_{M}^\ve))}+ C^p\big].
   \end{aligned} 
\end{equation*}
  The boundary terms $\langle \mathcal N_\delta( \be(\bu^\ve))\,  \G(\n^\ve), |\n^\ve|^{p-2} \n^\ve \rangle_{\Gamma_{\cI, \tau}}$ and $\langle \J_{n}(\n^\ve), |\n^\ve|^{p-2} \n^\ve \rangle_{\Gamma_{\cE, \tau}}$  are  estimated in the same way as in the proof of  Theorem~\ref{th:exist_1}, see the estimate \eqref{boundary_nc_2}.
Applying the Gagliardo--Nirenberg inequality  together with the   properties of the extension of $|\n^\ve|^{\frac p 2}$ and $|b^\ve|^{\frac p 2}$, see Lemma~\ref{lem:extension}, implies
  \begin{equation*}
 \begin{aligned}
 &  \|b^\ve(\tau)\|^p_{L^p(\Omega_{M}^\ve)}  +  \|\n^\ve(\tau) \|^p_{ L^p(\Omega_{M}^\ve)} + 
  \|\nabla |b^\ve|^{\frac p 2}\|^2_{L^2(\Omega_{M, \tau}^\ve)}  +  \|\nabla |\n^\ve|^{\frac p 2}\|^2_{L^2(\Omega_{M, \tau}^\ve)}
    \leq  C_1  p^{10}  \Big[ \sup_{(0,\tau)}\|\n^\ve\|^{\frac p 2}_{L^{\frac p 2}(\Omega_{M}^\ve)}\Big]^2 \\
 &    + C_2 p^{10}  \Big[  \sup_{(0,\tau)}\|b^\ve\|^{\frac p2}_{L^{\frac p 2}(\Omega_M^\ve)} \Big]^2 \hspace{-0.15 cm} + C_3^p
\big[1 + \|\n^\ve\|^p_{L^\infty(0,\tau; L^2(\Omega_{M}^\ve))} \hspace{-0.08 cm} + \|b^\ve\|^p_{L^\infty(0,\tau; L^2(\Omega_{M}^\ve))}\big] \hspace{-0.08 cm} + C_\delta^p   \| \be(\bu^\ve)\|^p_{L^\infty(0,\tau; L^2(\Omega))}. 
 \end{aligned} 
\end{equation*}
  Then similar to   the proof of Theorem~\ref{th:exist_1},  iterating in $p$, see     \cite[Lemma 3.2]{Alikakos},    and using the estimate   \eqref{estim_L2_nd}  yield
\begin{equation}\label{estim_Linfty_2}
\begin{aligned}
&\|\n^\ve\|_{L^\infty(0,T; L^\infty(\Omega_{M}^\ve))} + \|b^\ve\|_{L^\infty(0,T; L^\infty(\Omega_{M}^\ve))} \\& \leq C_1 \big[1+\| \n^\ve\|_{L^\infty(0,\tau; L^2(\Omega_M^\ve))}+ \| b^\ve\|_{L^\infty(0,\tau; L^2(\Omega_M^\ve))}+ \|\be(\bu^\ve)\|_{L^\infty(0,T; L^2(\Omega))}\big] \leq C_2, 
\end{aligned}
\end{equation}
where $C_1$, $C_2$ are independent of $\ve$.  Here we used  that $$\|b^\ve\|_{L^p(\Omega_{M}^\ve)}  +  \|\n^\ve \|_{ L^p(\Omega_{M}^\ve)} \leq 2 \big[\|b^\ve\|^p_{L^p(\Omega_{M}^\ve)}  +  \|\n^\ve \|^p_{ L^p(\Omega_{M}^\ve)}\big]^{\frac 1p}. $$

Hence, similar to the proof of Theorem~\ref{th:exist_1}, 
using the {\it a priori} estimates  \eqref{estim_L2_nd} and \eqref{estim_Linfty_2}  and applying the Galerkin method and the Schaefer fixed-point theorem  yield  the existence of a unique  solution of the microscopic problem {\rm \eqref{sumbal}, \eqref{BC1}, \eqref{sumbal_1}, \eqref{BC4}} for a given $\bu^\ve$ with $\|\be(\bu^\ve)\|_{L^\infty(0,T; L^2(\Omega))}\leq C$.
The estimates \eqref{estim_L2_nd} and \eqref{estim_Linfty_2} also ensure $\partial_t \n^\ve  \in L^2(0,T; \mathcal V(\Omega_M^\ve)^\prime)^2$ and 
$\partial_t b^\ve  \in L^2(0,T; \mathcal V(\Omega_M^\ve)^\prime)$ for every fixed $\ve >0$.

Similar to the proof of Lemma~\ref{estim_deriv},  to show the last estimate in \eqref{estim_un_3} we integrate the equations for $\n^\ve$ in \eqref{sumbal_1} over $(t, t+h)$ and consider $\vartheta_h \n^\ve - \n^\ve$  as   a test function,  where $\vartheta_h v(t,x) = v(t+h,x)$ for $(t,x) \in (0,T-h]\times \Omega_M^\ve$, 
 \begin{equation*}
\begin{aligned}
& \| \vartheta_h \n^\ve - \n^\ve \|^2_{L^2( \Omega_{M, \tau}^\ve)} \leq 
\Big| \Big\langle D_n\int_t^{t+h} \nabla \n^\ve \,  ds,  \vartheta_h\nabla \n^\ve -\nabla \n^\ve \Big \rangle_{\Omega_{M, \tau}^\ve} \Big|
 + \Big|  \Big\langle \int_t^{t+h}  \J_n (\n^\ve) ds, \vartheta_h \n^\ve - \n^\ve \Big\rangle_{\Gamma_{\cE, \tau}} \Big| \\ &  
+ \Big| \Big \langle \int_t^{t+h}\hspace{-0.17  cm} \big[\F_n(\p^\ve,  \n^\ve) +  \Q_n(\n^\ve, b^\ve, \be(\bu^\ve))   \big]ds,  \vartheta_h \n^\ve- \n^\ve \Big\rangle_{\Omega_{M, \tau}^\ve}
 \hspace{-0.4 cm}  + \Big \langle \int_t^{t+h} \hspace{-0.45  cm} \G(\n^\ve) \mathcal N_\delta(\be(\bu^\ve))  ds, \vartheta_h \n^\ve - \n^\ve \Big\rangle_{\Gamma_{\cI, \tau}}\Big|
\end{aligned}
\end{equation*}
for all $\tau \in (0, T-h]$ and any $h>0$. Then  the boundedness of $\p^\ve_j$, with $j=1,2$, the  estimates for $\bu^\ve$ in \eqref{estim_u_2} and for $\n^\ve$  and $\nabla \n^\ve$ in \eqref{estim_L2_nd} and \eqref{estim_Linfty_2},  together with the H\"older inequality,  imply the estimates for $\n^\ve(t+h,x) - \n^\ve(t,x)$,   stated in  the Theorem.  Similar calculations ensure the corresponding estimate for $b^\ve(t+h,x) - b^\ve(t,x)$.
 \end{proof}

\subsection{Existence of a unique solution of the coupled system {\rm \eqref{sumbal}, \eqref{BC1}--\eqref{BC4}}. Proof of Theorem~\ref{existence_3}.}

We prove the existence of a unique solution of  {\rm \eqref{sumbal}, \eqref{BC1}--\eqref{BC4}} in a similar  way as Theorem~\ref{existence_2}. The only difference is in the derivation of the estimate for $\| b^{\ve,j}-b^{\ve, j+1}\|_{L^\infty(0, T; L^\infty(\Omega_{M}^\ve))}$ for  two fixed-point iterations.

\begin{proof}[\bf Proof of Theorem~\ref{existence_3}.] 
Similar to the proof of Theorem~\ref{existence_2} we define  $\mathcal K: L^\infty(0, T; L^\infty(\Omega_{M}^\ve))\to   L^\infty(0, T; L^\infty(\Omega_{M}^\ve))$ and  derive a contraction inequality.
Considering the equations for $\widetilde \n^{\ve, j}$ and $\widetilde b^{\ve, j}$, where 
$\widetilde \n^{\ve,j}= \n^{\ve, j} - \n^{\ve, j+1}$, $\widetilde b^{\ve,j}= b^{\ve, j} - b^{\ve, j+1}$ and $\widetilde \bu^{\ve,j} = \bu^{\ve,j} -  \bu^{\ve, j+1}$, 
and taking $\widetilde \n^{\ve, j}$ and $\widetilde b^{\ve, j}$  as test functions yield
\begin{equation*}
\begin{aligned}
&\|\widetilde \n^{\ve,j}(\tau)  \|^2_{L^2(\Omega_M^\ve)} + \|\nabla \widetilde \n^{\ve,j}\|^2_{L^2(\Omega_{M,\tau}^\ve)}+
\|\widetilde b^{\ve,j}(\tau)  \|^2_{L^2(\Omega_M^\ve)} + \|\nabla \widetilde b^{\ve,j}\|^2_{L^2(\Omega_{M,\tau}^\ve)}
\\ 
&\leq  C_1 \Big[1+\|b^{\ve,j}\|_{L^\infty(\Omega_{M,\tau}^\ve)} 
 +\|\n^{\ve,j}\|_{L^\infty(\Omega_{M,\tau}^\ve)}\Big] \Big[\|\be(\widetilde \bu^{\ve,j})  \|^2_{L^2(\Omega_{M,\tau}^\ve)}   + \|\widetilde \n^{\ve,j}  \|^2_{L^2(\Omega_{M,\tau}^\ve)}  + \|\widetilde  b^{\ve,j}  \|^2_{L^2(\Omega_{M,\tau}^\ve)} \Big]
\\ 
&\qquad\qquad + C_2\|\be(\bu^{\ve, j+1})\|_{L^\infty(0,\tau; L^2(\Omega_M^\ve))} \Big[1+ \|\n^{\ve,j}\|_{L^\infty(\Omega_{M,\tau}^\ve)}  +  \|b^{\ve,j}\|_{L^\infty(\Omega_{M,\tau}^\ve)} + \|\n^{\ve, j+1}\|_{L^\infty(\Omega_{M, \tau}^\ve)} 
\\
&\hspace {6. cm } + \|b^{\ve, j+1}\|_{L^\infty(\Omega_{M, \tau}^\ve)}\Big]  
\Big[  \|\widetilde \n^{\ve,j}  \|^2_{L^2(0,\tau; L^4(\Omega_{M}^\ve))} +  \|\widetilde b^{\ve,j}  \|^2_{L^2(0,\tau; L^4(\Omega_{M}^\ve))}\Big]. 
 \end{aligned}
\end{equation*}
Using the trace inequality and the assumptions on $\J_n$ and $\G$, the boundary terms  for $\widetilde \n^{\ve,j}$ are estimated as  
\begin{equation*}
\begin{aligned}
& \langle \J_{n} (\n^{\ve, j}) - \J_n (\n^{\ve, j+1}), \widetilde \n^{\ve,j}\rangle_{\Gamma_{\cE, \tau}} \leq   C_\sigma \| \widetilde \n^{\ve,j}\|^2_{L^2(\Omega^\ve_{M, \tau}) } + 
\sigma\| \nabla \widetilde \n^{\ve,j}\|^2_{L^2(\Omega^\ve_{M, \tau}) }, \\
& \langle \mathcal N_\delta(\be(\bu^{\ve,j})) \big(\G(\n^{\ve, j}) -  \G( \n^{\ve, j+1})\big), \widetilde \n^{\ve, j}\rangle_{\Gamma_{\cI, \tau}} \leq 0 
\end{aligned}
\end{equation*}
and  
\begin{equation*}
\begin{aligned}
\big| \big \langle \big(\mathcal N_\delta(\be(\bu^{\ve,j}))- \mathcal N_\delta(\be(\bu^{\ve, j+1})) \big) \G (\n^{\ve, j+1}), \widetilde \n^{\ve, j} \big \rangle_{\Gamma_{\cI, \tau}}\big|
\leq C_{\sigma_1} \|\mathcal N_\delta(\be(\bu^{\ve,j}))- \mathcal N_\delta(\be(\bu^{\ve, j+1})) \|^2_{L^2(\Gamma_{\cI,\tau}) }
\\
+\sigma_1\|\n_2^{\ve, j+1} \|^2_{L^\infty(\Gamma_{\cI, \tau})}\| \widetilde \n_2^{\ve, j}\|^2_{L^2(\Gamma_{\cI, \tau}) }
\leq  \sigma \big(\| \widetilde \n_2^{\ve, j}\|^2_{L^2(\Omega^\ve_{M, \tau}) } + 
\| \nabla \widetilde \n_2^{\ve,j}\|^2_{L^2(\Omega^\ve_{M, \tau}) } \big)
\\  +   C_\delta \big(\| \be(\bu^{\ve,j}) \|^2_{L^\infty(0, \tau; L^2(\Omega))} \|\widetilde b^{\ve,j} \|^2_{L^2(\Omega_{M, \tau}^\ve)}  + 
\|b^{\ve, j+1} \|^2_{L^\infty(0,\tau; L^2(\Omega_{M}^\ve))} \|\be(\widetilde \bu^{\ve,j}) \|^2_{L^2(\Omega_{\tau})} \big),  
 \end{aligned}
\end{equation*}
with arbitrary $\sigma_1, \sigma>0$. 
Using   the  Gagliardo--Nirenberg inequlity   we  estimate    $\|\widetilde \n^{\ve, j} \|^2_{L^2(0,\tau; L^4(\Omega_{M}^\ve))}$ and
$\|\widetilde b^{\ve, j} \|^2_{L^2(0,\tau; L^4(\Omega_{M}^\ve))}$ 
 in terms of $ \|\widetilde \n^{\ve, j} \|^2_{L^2(\Omega_{M,\tau}^\ve)}$,   $ \|\nabla \widetilde \n^{\ve,j} \|^2_{L^2(\Omega_{M,\tau}^\ve)}$
 and  $ \|\widetilde b^{\ve, j} \|^2_{L^2(\Omega_{M,\tau}^\ve)}$,  $ \|\nabla \widetilde b^{\ve,j} \|^2_{L^2(\Omega_{M,\tau}^\ve)}$, respectively.   Then   the {\it a priori} estimates, similar to those in \eqref{estim_u_21} and \eqref{estim_L2_nd}--\eqref{estim_Linfty_2},  and   the Gronwall inequality yields 
 \begin{equation*}
\begin{aligned}
  \|\widetilde \n^{\ve,j}  \|^2_{L^\infty(0,\tau; L^2(\Omega_M^\ve))} + \|\nabla \widetilde \n^{\ve,j}  \|^2_{L^2(\Omega_{M, \tau} ^\ve)} 
  +  \|\widetilde b^{\ve,j}  \|^2_{L^\infty(0,\tau; L^2(\Omega_M^\ve))} + \|\nabla \widetilde b^{\ve,j}  \|^2_{L^2(\Omega_{M, \tau} ^\ve)} 
\leq  C  \|\be(\widetilde \bu^{\ve,j}) \|^2_{L^2(\Omega_{M, \tau}^\ve)}.
 \end{aligned}
\end{equation*}

Considering  $(\widetilde b^{\ve,j})^{p-1}$  as a test function in the equation   for   the difference of two iterations  
 $b^{\ve, j}$ and $b^{\ve, j+1}$ implies  
\begin{equation} \label{esrim_diff_1}
\begin{aligned}
 \frac 1 p  \|\widetilde b^{\ve, j}(\tau)\|^p_{L^p(\Omega_M^\ve)} +\frac { 4(p-1)}{p^2}  \|\nabla |\widetilde  b^{\ve, j}|^{\frac p 2}\|^2_{L^2(\Omega_{M,\tau}^\ve)}  \leq C_1\Big[ 1+\|\n^{\ve, j}\|_{L^\infty(\Omega_{M, \tau}^\ve)}+ \|b^{\ve, j}\|_{L^\infty(\Omega_{M, \tau}^\ve)}\\ 
 +  \|\n^{\ve, j+1}\|_{L^\infty(\Omega_{M, \tau}^\ve)} + \|b^{\ve, j+1}\|_{L^\infty(\Omega_{M, \tau}^\ve)}\Big] \int_0^\tau 
\Big[  \big(1+ \|\be(\bu^{\ve, j+1})\|_{L^2(\Omega_M^\ve)}\big) \||\widetilde  b^{\ve,j}|^{p}\|_{L^2(\Omega_M^\ve)}  \\
 +  \|\widetilde \n^{\ve,j}  \|_{L^2(\Omega_{M}^\ve)}  \||\widetilde b^{\ve, j}|^{p-1}\|_{L^2(\Omega_{M}^\ve)} \Big] dt \\
 + C_3  \big[1+\|\n^{\ve,j}\|_{L^\infty(\Omega_{M, \tau}^\ve)} + \|b^{\ve,j}\|_{L^\infty(\Omega_{M, \tau}^\ve)}\big]  \int_0^\tau \|\be(\widetilde \bu^{\ve,j}) \|_{L^2(\Omega_M^\ve)} \||\widetilde b^{\ve, j}|^{p-1}\|_{L^2(\Omega_M^\ve)}  dt.
\end{aligned}
\end{equation}
The last term in \eqref{esrim_diff_1} we  rewrite as
\begin{eqnarray}\label{estim_diff_11}
 \int_0^\tau \|\be(\widetilde \bu^{\ve,j}) \|_{L^2(\Omega_M^\ve)} \||\widetilde b^{\ve,j}|^{p-1}\|_{L^2(\Omega_M^\ve)} dt  \leq \Big(\int_0^\tau \|\be(\widetilde \bu^{\ve,j}) \|^{\frac{(1+\varsigma)p}{(p\varsigma +1)}}_{L^2(\Omega_M^\ve)} \Big)^{\frac{p\varsigma + 1}{p(1+ \varsigma)}}
\Big( \int_0^\tau \| |\widetilde b^{\ve, j}|^{\frac p 2}\|^{2(1+\varsigma)}_{L^4(\Omega_M^\ve)} dt \Big)^{\frac {p-1}{(1+\varsigma) p}}
\nonumber \\ 
 \leq \frac 1  p  C \|\be(\widetilde \bu^{\ve,j}) \|_{L^{1+\frac 1\varsigma}(0,\tau; L^2(\Omega_M^\ve))}^p   + \frac{p-1} p 
\Big( \int_0^\tau \| |\widetilde b^{\ve, j}|^{\frac p 2}\|^{2(1+\varsigma)}_{L^4(\Omega_M^\ve)} dt \Big)^{\frac 1{1+\varsigma}}\quad 
\end{eqnarray}
 with some  $0<\varsigma <1$.  Applying the Gagliardo--Nirenberg inequality yields 
 \begin{equation*}
 \begin{aligned}
 \| |\widetilde b^{\ve, j}|^{\frac p 2}\|^{2}_{L^4(\Omega_M^\ve)}       \leq C_1 \big[ \| \nabla |\widetilde b^{\ve, j}|^{\frac p 2}\|^{2 a }_{L^2(\Omega_M^\ve)} 
     \| |\widetilde b^{\ve, j}|^{\frac p 2}\|^{2(1-a)}_{L^1(\Omega_M^\ve)}  + \| |\widetilde b^{\ve, j}|^{\frac p 2}\|^{2}_{L^1(\Omega_M^\ve)} \big], 
     \end{aligned} 
 \end{equation*}
 where  $a =   9/{10}$ (for a three-dimensional domain).   Considering $\varsigma$ such that $a(1+ \varsigma) < 1$ we obtain 
 \begin{equation}\label{complex_in}
\begin{aligned}
 \int_0^\tau \||\widetilde  b^{\ve, j}|^{\frac p 2}\|^{2(1+\varsigma)}_{L^4(\Omega_M^\ve)} dt 
&\leq C_1 \| \nabla |\widetilde  b^{\ve, j}|^{\frac p 2}\|^{2a(1+\varsigma) }_{L^2(\Omega_{M,\tau}^\ve)} 
\Big( \int_0^\tau  \| |\widetilde  b^{\ve, j}|^{\frac p 2}\|^{\frac{2(1+\varsigma)(1-a)}{1- a (1+\varsigma)}}_{L^1(\Omega_M^\ve)} dt \Big)^{1- a (1+\varsigma)} \\
 &+ C_2 \int_0^\tau \| |\widetilde b^{\ve, j}|^{\frac p 2}\|^{2(1+\varsigma)}_{L^1(\Omega_M^\ve)} dt\Big]
\leq  \left (\frac \sigma  p\right)^{1+\varsigma} \| \nabla |\widetilde  b^{\ve, j}|^{\frac p 2}\|^{2(1+\varsigma) }_{L^2(\Omega_{M,\tau}^\ve)}  
\\& + C_\sigma \, p^{\frac{ a(1+\varsigma)}{(1-a)}}\Big( \int_0^\tau  \| |\widetilde b^{\ve,j}|^{\frac p 2}\|^{\frac{2(1+\varsigma)(1-a)}{1-a(1+\varsigma)}}_{L^1(\Omega_M^\ve)} dt \Big)^{\frac {1- a(1+\varsigma)}{(1-a)}}+  C \int_0^\tau \| |\widetilde b^{\ve, j}|^{\frac p 2}\|^{2(1+\varsigma)}_{L^1(\Omega_M^\ve)} dt
 \end{aligned}
\end{equation}
for  an arbitrary $\sigma>0$.  Using  the estimate \eqref{complex_in}  in \eqref{estim_diff_11} implies
\begin{equation*}
\begin{aligned}
 \int_0^\tau \|\be(\widetilde \bu^{\ve, j})\|_{L^2(\Omega_M^\ve)} \||\widetilde b^{\ve, j}|^{p-1}\|_{L^2(\Omega_M^\ve)} dt 
   \leq C_1 \frac 1 p  \|\be(\widetilde \bu^{\ve,j}) \|_{L^{1+\frac 1 \varsigma}(0,\tau; L^2(\Omega_M^\ve))}^p    + 
 \sigma \, \frac{ p-1}{p^2} \| \nabla |\widetilde  b^{\ve, j}|^{\frac p 2}\|^2_{L^2(\Omega_{M,\tau}^\ve)} \\   +
 C_\sigma (\tau^{\frac{1- a(1+\varsigma)}{(1-a)(1+\varsigma)}} + \tau^{\frac 1{1+\varsigma}}) \, p^9 \Big[\sup_{(0,\tau)} \| |\widetilde  b^{\ve, j}|^{\frac p 2}\|_{L^1(\Omega_M^\ve)}\Big]^2. 
 \end{aligned}
\end{equation*}
The same estimates hold for  the term 
$ \|\widetilde \n^{\ve, j}  \|_{L^2(\Omega_M^\ve)}\,  \||\widetilde b^{\ve, j}|^{p-1}\|_{L^2(\Omega_M^\ve)} $ in  \eqref{esrim_diff_1}.
Using the Gagliardo--Nirenberg inequality, the first  integral on the right-hand side of \eqref{esrim_diff_1} is estimated as 
\begin{equation*}
\begin{aligned}
&\int_0^\tau \|\be(\bu^{\ve, j+1})\|_{L^2(\Omega_M^\ve)} \||\widetilde b^{\ve,j}|^{\frac p 2}\|^2_{L^4(\Omega_M^\ve)} d\tau
 \leq   (p-1)/p^2 \,  \|\nabla |\widetilde  b^{\ve,j}|^{\frac p 2}\|^2_{L^2(\Omega_{M, \tau}^\ve)} \\
& \qquad \qquad +C\big( \|\be(\bu^{\ve, j+1})\|^{\frac 1{1-a}}_{L^\infty(0,\tau; L^2(\Omega_M^\ve))} +1\big) \Big(\frac {p^2}{p-1}\Big)^{9} \||\widetilde b^{\ve,j}|^{\frac p 2}\|^2_{L^2(0,\tau; L^1(\Omega_M^\ve))}.
 \end{aligned}
\end{equation*}
   Applying the recursive  iterations as in  \cite[Lemma 3.2]{Alikakos},  in the same way as in the proof of Theorem~\ref{existence_2}, we obtain
\begin{equation}\label{estim_cont_m_2}
\| \widetilde b^{\ve, j}\|_{L^\infty(0, \tau; L^\infty(\Omega_{M}^\ve))} \leq C \|\be(\widetilde \bu^{\ve,j}) \|_{L^{1+\frac 1 \varsigma}(0,\tau; L^2(\Omega))}  \qquad \text{ for any  } \;  \varsigma \in (0, 1/9) \quad \text{ and } \; \tau \in (0, T].  
\end{equation} 
Then,  similar  to the proof of Theorem~\ref{existence_2}, 
combining \eqref{estim_cont_m_2}  and   \eqref{estim_conts_m_1_1}, choosing $\tau$ sufficiently small, applying the same argument as in the proof of  the Banach fixed-point theorem,  and  iterating  over time-intervals,  yield  the existence of a unique weak solution of  Model~II. 
\end{proof} 

\section{Convergence results and the derivation of  the macroscopic equations for  Model~I.}\label{secthom}
In this section we first  prove convergence results for a  sequence of solutions of the microscopic problem \eqref{sumbal}--\eqref{sumbal2} and then, using homogenization techniques,  derive a macroscopic  model for plant cell wall biomechanics. 

\subsection{Convergence results for solutions of the microscopic Model I} \label{converg_model_I} 

\begin{lemma}\label{converg_1}
There exist  $\p, \n \in \big[L^2(0,T; \mathcal V(\Omega)) \cap  L^\infty(0,T; L^\infty(\Omega))\big]^2$,  $\hat \p, 
\hat \n  \in L^2(\Omega_T; H^1_{\rm per} (\hat Y)/\mathbb R)^2$,   and  $b\in W^{1, \infty}(0,T; L^\infty(\Omega\times \hat Y_M))$, $\bu, \partial_t \bu \in L^\infty(0,T; \cW(\Omega))$,  $\hat \bu \in L^2(\Omega_T; H^1_{{\rm per}}(\hat Y)/\mathbb R)^3$, such that for a subsequence $(\p^\ve, \n^\ve, b^\ve, \bu^\ve)$ of the sequence of  solutions  of the microscopic problem   \eqref{sumbal}--\eqref{sumbal2} (denoted again by $(\p^\ve, \n^\ve, b^\ve, \bu^\ve)$) we have 
\begin{eqnarray*}
\begin{aligned}
&\p^\ve  \rightharpoonup \p, \quad&& \n^\ve  \rightharpoonup \n  &&   \text{ weakly in } L^2(0,T; H^1(\Omega)),  \\
& \p^\ve  \rightharpoonup \p,   && \n^\ve  \rightharpoonup \n   &&   \text{ two-scale},\\
&\nabla \p^\ve  \rightharpoonup \nabla \p + \hat\nabla_y  \hat \p,  &&  \nabla \n^\ve \rightharpoonup \nabla \n + \hat\nabla_y  \hat \n &&   \text{ two-scale},  \\
&\p^\ve  \to \p ,   && \n^\ve  \to \n &&  \text{ strongly in } L^2(\Omega_T) \text{ and }   L^2((0,T)\times \partial \Omega),  \\
&  b^\ve \rightharpoonup  b ,   && \partial_t  b^\ve  \rightharpoonup  \partial_t  b   && \text{ two-scale}
\end{aligned}
\end{eqnarray*}
as $\ve \to 0$, where $\hat\nabla_y v = (\partial_{y_1} v, \partial_{y_2} v, 0)^T$, and
\begin{eqnarray*}
\begin{aligned}
&\bu^\ve  \rightharpoonup \bu  &&  \text{ weakly}^\ast  \text{ in } L^\infty(0,T; \cW(\Omega)), \\
&\partial_t \bu^\ve  \rightharpoonup \partial_t \bu  &&  \text{ weakly}  \text{ in } L^2(0,T; \cW(\Omega)), \\
&\bu^\ve \rightharpoonup \bu, \;  \quad \nabla \bu^\ve \rightharpoonup  \nabla \bu + \hat\nabla_y \hat \bu \; \;  && \text{ two-scale},  \\
&\int_\Omega \be(\bu^\ve) \, dx  \to \int_\Omega \be(\bu) \, dx  &&  \text{ strongly  in } L^2(0,T),  \; \; \; \qquad   \text{as }\; \ve \to 0.
\end{aligned}\;\;\;\;
\end{eqnarray*}
\end{lemma}

\begin{proof}
The \textit{a priori} estimates in Theorem~\ref{existence_2}  together with the  extension Lemma \ref{lem:extension}  and the compactness theorems for the two-scale convergence, see e.g.~\cite{allaire,Nguetseng}, ensure the weak and two-scale convergence of $\p^\ve$, $\n^\ve$, $b^\ve$, and $\bu^\ve$, stated in the Lemma.  

The strong convergence of $\p^\ve$ and $\n^\ve$  in $L^2(\Omega_T)$ follows from  the estimates for  $\|\nabla \p^\ve\|_{L^2(\Omega^\ve_{M, T})}$,  $\|\nabla \n^\ve\|_{L^2(\Omega^\ve_{M,T})}$, $\|\vartheta_h\p^\ve- \p^\ve\|_{L^2(\Omega^\ve_{M,T-h})}$, and $\|\vartheta_h\n^\ve - \n^\ve \|_{L^2(\Omega^\ve_{M,T-h})}$, see \eqref{estim_un_1},  together with  the linearity of the extension from $\Omega^\ve_M$ to $\Omega$ and the Kolmogorov compactness theorem \cite{Brezis, necas}.
The  embedding  $\{ \gamma(v) \; | \;  v \in  H^{\varsigma}(\Omega) \} \subset L^2(\partial\Omega)$, with $\varsigma \in (1/2, 1)$ and $\gamma(v)$ denote the trace of $v$ on $\partial \Omega$,    the compact embedding $H^1(\Omega) \subset H^{\varsigma}(\Omega)$, the estimates for $\|\vartheta_h\p^\ve- \p^\ve\|_{L^2(\Omega^\ve_{M,T-h})}$, and $\|\vartheta_h\n^\ve - \n^\ve \|_{L^2(\Omega^\ve_{M,T-h})}$,  and the compactness result in \cite{Simon} ensure  the strong convergence in  $L^2((0,T)\times \partial \Omega)$.   
The boundedness of $\p^\ve$, $\n^\ve$, $b^\ve$, and $\partial_t b^\ve$, along with the convergence results,  implies  the boundedness of the limit functions $\p$, $\n$, $b$, and $\partial_t b$.

The \textit{a priori} estimate  for $\partial_t \be(\bu^\ve)$ yields the last strong convergence stated in the Lemma. 
\end{proof}

In order to pass to the limit in the nonlinear functions $\R_n$, $R_b$ and $\bbE_M$ we have to show  the strong convergence of a subsequence of  $\{b^\ve\}$. To show the strong convergence of a sequence   defined on the perforated  $\ve$-dependent domain $\Omega_M^\ve$,  we  use the unfolding operator to map it to a sequence defined on the fixed domain $\Omega \times \hat Y_M$, see e.g.  \cite{CDG,CDDGZ}.

\begin{definition}
For a measurable function $\phi$ on $\Omega^\ve_M$, 
  the  unfolding operator 
$\mathcal T_\ve$ is defined as 
$$
\mathcal T_\ve(\phi)(x,y) = \phi(\ve [\hat x/\ve]_{\hat Y_M}+ \ve y, x_3)\qquad \quad   \text{ for } \; x \in \Omega, \; y \in \hat Y_M,
$$
where $\hat x = (x_1, x_2)$ and $[\hat x/\ve]_{\hat Y_M}$ is the unique integer combination of the periods, s.t.\ $\hat x/\ve - [\hat x/\ve]_{\hat Y_M}  \in \hat Y_M$.
\end{definition}
For the unfolded sequence $\{\mathcal T_\ve(b^\ve)\}$ we have the following strong convergence result. 
 \begin{lemma} \label{strong_nb_2} 
 Under Assumption~\ref{assumptions} we have, up to a subsequence,  
\begin{eqnarray*}
\cT_\ve(b^\ve)  \to b \quad   \text{ strongly}  \text{ in } L^2(\Omega_T\times \hat Y_M),  \quad  \text{ as } \ve \to 0.
\end{eqnarray*}
\end{lemma} 
 \begin{proof} 
 Using the extension of $\n^\ve$ from $\Omega_M^\ve$ to $\Omega$, see  Lemma~\ref{lem:extension}, we define 
 the extension of $b^\ve$ from $\Omega_M^\ve$ to $\Omega$ as a solution of  the ordinary differential equation 
 \begin{equation}\label{extend_n_b}
 \begin{aligned}
& \partial_t b^\ve = R_{b}(\n^\ve, b^\ve, \mathcal N_\delta(\be(\bu^\ve)))
 \qquad  && \text{ in } (0,T)\times\Omega, \\
& b^\ve(0,x) = b_0(x)  && \text{ in } \Omega.
 \end{aligned} 
 \end{equation}
 The construction of the extension for $\n^\ve$ and the uniform boundedness of $\n^\ve$  in $\Omega_{M,T}^\ve$, shown in Theorem~\ref{th:exist_1}, ensure 
 $$
 \| \n^\ve \|_{L^\infty(0, T; L^\infty(\Omega))} \leq C_1  \|\n^\ve \|_{L^\infty(0, T; L^\infty(\Omega_{M}^\ve))}\leq C_2, 
 $$
 with the constants $C_1$ and $C_2$ independent of $\ve$.  Then, from equation \eqref{extend_n_b} and using the assumptions on $R_b$,  we also  obtain the uniform  boundedness of  $b^\ve$ and $\partial_t b^\ve$ in $L^\infty(0, T; L^\infty(\Omega))$. 
 
 It follows from the properties of the unfolding operator   \cite{CDG,CDDGZ} that the lemma holds if it is shown that $b^\ve$ converges strongly to $b$. We show the strong convergence of $b^\ve$  by applying the Kolmogorov  theorem \cite{Brezis,necas}.  Considering equation  \eqref{extend_n_b} at $(t,x+\textbf{h}_j)$ and $(t,x)$, where   ${\bf h}_j = h \mathbf{b}_j$, with $(\mathbf{b}_1, \mathbf{b}_2, \mathbf{b}_3)$ being the canonical basis in $\mathbb R^3$ and $h>0$,  taking $b^\ve(t,x+{\bf h}_j) -  b^\ve(t,x)$ as a test function and using the boundedness of $\n^\ve_1$, $\n^\ve_2$, and  $b^\ve$, along with  the local  Lipschitz continuity of  $R_b$,  yield 
 \begin{equation*} 
 \begin{aligned}
 \| b^\ve(\tau, \cdot+\textbf{h}_j) - b^\ve(\tau,\cdot) \|^2_{L^2(\Omega_{2h})} \hspace{-0.1 cm} \leq \hspace{-0.05 cm}
 C \hspace{-0.1 cm}\int_0^\tau\hspace{-0.2 cm} \big[  \|\n^\ve(t, \cdot+{\bf h}_j) -  \n^\ve(t, \cdot)\|^2_{L^2(\Omega_{2h})}\hspace{-0.15 cm} + \|b^\ve(t, \cdot+{\bf h}_j) -  b^\ve(t, \cdot)\|^2_{L^2(\Omega_{2h})}\big] dt 
 \\ +\| b_0(\cdot+{\mathbf h}_j) - b_0(\cdot) \|^2_{L^2(\Omega_{2h})} +   \delta^{-6}\int_0^\tau \Big \| \int_{B_{\delta,h}(x)\cap \Omega}{\rm tr }\, \bbE^\ve(b^\ve) \be(\bu^{\ve}(t,\tilde x)) d\tilde x \Big\|^2_{L^2(\Omega_{2h})}  dt 
  \end{aligned}
 \end{equation*}
for  $\tau \in (0, T]$,  where the constants $C_1, C_2$ are independent of $\ve$ and $h$, 
 $\Omega_{2h} = \{ x \in \Omega  \; | \; \text{dist} (x, \partial \Omega) > 2h \}$,  and 
 $B_{\delta, h}(x)= \big[B_\delta(x+\textbf{h}_j) \setminus B_\delta(x) \big]\cup \big[B_\delta(x) \setminus B_\delta(x+\textbf{h}_j) \big]$.
 Using the regularity of the initial condition $b_0 \in H^1(\Omega)$, the \textit{a priori} estimates for $\be(\bu^\ve)$ and $\nabla \n^\ve$ along with  the fact that
 $|B_{\delta, h}(x)\cap \Omega| \leq C \delta^2 h $ for all $x\in \overline \Omega$, and applying the Gronwall inequality we  obtain 
  \begin{eqnarray} \label{estim_strong1}
 \sup\limits_{t\in (0,T)}\| b^\ve(t, \cdot +\textbf{h}_j) - b^\ve(t,\cdot) \|^2_{L^2(\Omega_{2h})} \leq C_\delta h.
 \end{eqnarray} 
 Extending  $b^\ve$ by zero from $\Omega_T$ into $\mathbb R_{+}\times \mathbb R^3$ and using the uniform boundedness of $b^\ve$ in $L^\infty(0,T; L^\infty(\Omega))$ imply 
 \begin{eqnarray}\label{estim_strong2}
  \|b^\ve\|^2_{L^\infty(0,T; L^2(\widetilde \Omega_{3h}))} +\|b^\ve\|^2_{L^2((T-h,T+h)\times  \Omega)} \leq C h, 
 \end{eqnarray}
where $\widetilde \Omega_{3h} = \{ x\in \mathbb R^3 \; | \; \text{dist}(x, \partial \Omega)< 3h\}$ and the constant $C$ is independent of $\ve$ and $h$. The  estimates for  $\partial_t b^\ve$ ensure
 \begin{eqnarray}\label{estim_strong3}
\|b^\ve(\cdot+h, \cdot) - b^\ve(\cdot,\cdot) \|^2_{L^2((0,T-h)\times\Omega)} \leq C_1 h^2 \|\partial_t b^\ve \|^2_{L^2(\Omega_T)}\leq C_2 h^2, 
 \end{eqnarray}
where $C_1$ and $C_2$ are independent of $\ve$ and $h$. Combining \eqref{estim_strong1}--\eqref{estim_strong3} and applying  the Kolmogorov theorem imply  the strong convergence of $b^\ve$ to $\tilde b$ in $L^2(\Omega_T)$.  The definition of the two-scale-convergence  yields  $\tilde b = b$,  and hence the two-scale limit of $b^\ve$ is independent of $y$.
 \end{proof}

\noindent\textbf{Remark.} Notice that the two-scale convergence of $\bu^\ve$  and the estimates for $\partial_t \be(\bu^\ve)$ and $\partial_t b^\ve$, together with the Kolmogorov theorem \cite{Brezis}, imply  that 
$$
\int_\Omega \mathbb E^\ve(b, x)\be(\bu^\ve) dx  \to \int_\Omega \dashint_{\hat Y} \mathbb E (b, y) ( \be(\bu) + \hat \be_y(\hat \bu) ) dy dx  \qquad  \text{ in } L^2(0,T), 
$$
where $ \hat \be_{y}({\bf v})_{33}= 0$,   $\hat \be_{y}({\bf v})_{j3}=\hat \be_{y}({\bf v})_{3j} = \frac 12 \partial_{y_j} {\bf v}_3$ for $j=1,2$, and 
$ \hat \be_{y}({\bf v})_{ij}= \frac 12 (\partial_{y_i} {\bf v}_j +\partial_{y_j} {\bf v}_i)$ for $i,j=1,2$,
 and
$$
\dashint_{B_\delta(x)\cap\Omega} \mathbb E^\ve(b, \tilde x)\be(\bu^\ve) d\tilde x  \to \dashint_{B_\delta(x)\cap\Omega} \dashint_{\hat Y} \mathbb E (b, y) ( \be(\bu) + \hat \be_y(\hat \bu) ) dy d\tilde x  \; \;   \text{ in } L^2(0,T) 
$$
as $\ve \to 0$, for all $x \in \overline \Omega$. Then,  Lebesgue's dominated convergence  theorem   ensures 
\begin{equation}\label{strong_con_stress_1}
\dashint_{B_\delta(x)\cap\Omega} \mathbb E^\ve(b, \tilde x)\be(\bu^\ve) d\tilde x  \to \dashint_{B_\delta(x)\cap\Omega} \dashint_{\hat Y} \mathbb E (b, y) ( \be(\bu) + \hat \be_y(\hat \bu) ) dy d\tilde x  \quad  \text{ in } L^2(\Omega_T) \; \text{ and } \;   L^2(\Gamma_{\cI,T}), 
\end{equation}
as $\ve \to 0$. 
In the same way we also obtain
\begin{equation}\label{conv_int_eu_2}
\dashint_{B_\delta(x)\cap\Omega} \be(\bu^\ve) \, d\tilde x  \to \dashint_{B_\delta(x)\cap\Omega} \be(\bu) \,  d\tilde x  \qquad  \text{ in }\  L^2(\Omega_T) \;  \text{ and} \;  
L^2(\Gamma_{\cI,T}), \qquad \text{ as } \; \; \ve \to 0. 
\end{equation}

\subsection{Derivation of the macroscopic equations for  Model I. Proof of Theorem \ref{th:macro_1}.}
Using the convergence results, shown in Lemma~\ref{converg_1},  and  applying the two-scale convergence and the periodic unfolding methods, see e.g.~\cite{allaire,CDG,CDDGZ,Nguetseng}, we derive the macroscopic equations for  Model I.
 \begin{proof}[\bf Proof of  Theorem \ref{th:macro_1}]
We consider  $\varphi_b \in C^\infty(\overline \Omega_T)$ and  $\boldsymbol{\phi}_\alpha(t,x) = \boldsymbol{\varphi}_\alpha(t,x)  +\ve \boldsymbol{ \psi}_\alpha(t,x, \hat x/\ve)$, where   $\boldsymbol{\varphi}_\alpha \in C^\infty(\overline \Omega_T)^2$  is $a_3$-periodic in  $x_3$,  $\boldsymbol{\varphi}_\alpha(T,x) ={\bf 0}$ for $x\in \overline\Omega$,   and  $\boldsymbol{\psi}_\alpha \in C^\infty_0(\Omega_T; C^\infty_{\text{per}}(\hat Y))$,  for $\alpha = p, n$,   as test functions  in \eqref{weak_sol_n1}--\eqref{weak_sol_n2}. Applying the two-scale convergence and using  the strong convergence of  $\cT_\ve(b^\ve)$, $\p^\ve$ and  $\n^\ve$, yield
\begin{equation}\label{limit_11}
\begin{aligned}
&- \langle  \p, \partial_t \boldsymbol{\varphi}_p \rangle_{\Omega_T\times \hat Y_M} + \langle D_p (\nabla \p + \hat \nabla_y \hat \p), \nabla \boldsymbol{\varphi}_p + \hat \nabla_y \boldsymbol{\psi}_p \rangle_{\Omega_T\times \hat Y_M}
=  -\langle  \F_{p}(\p), \boldsymbol{\varphi}_p \rangle_{\Omega_T\times\hat Y_M} 
 \\
&\hspace{5 cm } 
 - |\hat Y| \langle \gamma_p \, \p, \boldsymbol{\varphi}_p \rangle_{\Gamma_{\cE,T}} 
 + |\hat Y| \left\langle \J_p\big(\p),  \boldsymbol{\varphi}_p \right\rangle_{\Gamma_{\cI,T}}+   \langle  \p_0,  \boldsymbol{\varphi}_p(0,x) \rangle_{\Omega\times \hat Y_M},
\\
&- \langle  \n, \partial_t   \boldsymbol{\varphi}_n \rangle_{\Omega_T\times \hat Y_M} + \langle D_n (\nabla \n + \hat \nabla_y\hat  \n), \nabla  \boldsymbol{\varphi}_n + \hat \nabla_y  \boldsymbol{\psi}_n \rangle_{\Omega_T\times \hat Y_M} 
 =  |\hat Y|\langle \G(\n)\, \mathcal N^{\rm eff}_\delta(\be(\bu)),  \boldsymbol{\varphi}_n \rangle_{\Gamma_{\cI,T}}  \\
 &\hspace{0.5 cm }   + |\hat Y| \langle \J_n(\n) ,  \boldsymbol{\varphi}_n \rangle_{\Gamma_{\cE,T}} 
 +  \langle \F_n(\p, \n)+  \R_n(\n, b, \mathcal N^{\rm eff}_\delta(\be(\bu) )),   \boldsymbol{\varphi}_n \rangle_{\Omega_T\times \hat Y_M}  
  +\langle  \n_0,  \boldsymbol{\varphi}_n(0,x) \rangle_{\Omega\times \hat Y_M},\\ 
& \langle \partial_t b, \varphi_b \rangle_{\Omega_T} = \langle  R_{b}(\n, b,  \, \mathcal N^{\rm eff}_\delta(\be(\bu))), \varphi_b \rangle_{\Omega_T}.
 \end{aligned}
\end{equation} 
Here we used the fact that the strong convergence of $\cT_\ve(b^\ve)$,   the two-scale convergence of $\bu^\ve$,  and the estimates for  $\partial_t \bu^\ve$ and $\partial_t b^\ve$ ensure
$$
\mathcal N_\delta(\be(\bu^\ve)) \to \mathcal N^{\rm eff}_\delta(\be(\bu)) \qquad  \text{ in } \;   L^2(\Omega_T) \quad  \text{ and } \quad  \text{ in } \;   L^2(\Gamma_{\cI, T})  \qquad \text{ as } \; \; \ve \to 0,  
$$
see the convergence in  \eqref{strong_con_stress_1}  and  the definition of $\mathcal N^{\rm eff}_\delta$ in  \eqref{macro_over_n_eu}.  Choosing $ \boldsymbol{\varphi}_\alpha \equiv {\bf 0}$ we obtain 
\begin{equation}\label{unit_11}
\begin{aligned}
\langle D_p (\nabla \p +\hat \nabla_y \hat \p), \hat \nabla_y  \boldsymbol{\psi}_p\rangle_{\Omega_{T}\times \hat Y_M} =0 \quad \text{ and } \quad 
 \langle D_n (\nabla \n +\hat \nabla_y\hat  \n), \hat \nabla_y  \boldsymbol{\psi}_n\rangle_{\Omega_{T}\times \hat Y_M} =0.
\end{aligned}
\end{equation}
The linearity of  \eqref{unit_11}  implies that  $\hat \p$ and $\hat \n$ have the form 
$$
\hat \p(t,x,y) = \sum_{j=1,2,3} \partial_{x_j} \p(t,x) {\bf v}^j_p(y), \quad  \hat \n(t,x,y) = \sum_{j=1,2,3} \partial_{x_j} \n(t,x) {\bf v}^j_n(y),
$$ 
where ${\bf v}^j_\alpha= ({\bf v}^j_{\alpha,1}, {\bf v}^j_{\alpha,2})^T$ are solutions of the unit cell problems \eqref{unit_n}, for $\alpha=n,p$, $j=1,2,3$. 
The definition of $\mathcal D_\alpha^l$ and  the fact that ${\bf v}^j_{\alpha, l}$ are solutions of \eqref{unit_n} ensure that $\mathcal D_\alpha^l$ are strongly elliptic, $\alpha = n, p$ and $l=1,2$.  

Taking  $ \boldsymbol{\psi}_\alpha \equiv {\bf 0}$
 and considering  first $\boldsymbol{\varphi}_\alpha \in C^1_0(\Omega_T)^2$ and 
then $\boldsymbol{\varphi}_\alpha \in C^1_0(0,T; C^1(\overline \Omega))^2$, with   $\boldsymbol{\varphi}_\alpha$  being $a_3$-periodic in $x_3$, where $\alpha=p,n$,
yields the macroscopic equations and boundary conditions  for $\p$ and $\n$ in \eqref{macro_1} and \eqref{macro_11}.
 
 The equation  \eqref{macro_11} and the regularity of $\p$, $\n$ and $\bu$ imply  $\partial_t \n, \, \partial_t \p  \in L^2(0,T; \cV(\Omega)^\prime)^2$.  Thus,   $\p, \n \in C([0,T]; L^2(\Omega))^2$ and together with the equations for $\p$ and $\n$ in \eqref{limit_11} we obtain  $\p(t, \cdot) \to \p_0$ and $\n(t, \cdot) \to \n_0$  in $L^2(\Omega)^2$ as $t \to 0$. The regularity of $b$ and the two-scale convergence of $\partial_t b^\ve$ ensure  $b(t, \cdot) \to b_0$  in $L^2(\Omega)$. 

Considering $\boldsymbol{\psi}(t,x)=\boldsymbol{\psi}_1(t,x) + \ve \boldsymbol{\psi}_2(t,x,\hat x/\ve)$, 
with $\boldsymbol{\psi}_1 \in C^1(\overline \Omega_T)^3 \cap L^2(0,T; \cW(\Omega))$ and  $\boldsymbol{\psi}_2 \in C^\infty_0(\Omega_T; C^\infty_\text{per}(\hat Y))^3$, as a test function  in  \eqref{weak_sol_u} and applying  the strong convergence of $\cT_\ve(b^\ve)$ and the two-scale convergence of $\bu^\ve$, we obtain
\begin{equation*}\label{macro_22}
|\hat Y|^{-1}  \langle \bbE(b,y) (\be(\bu) + \hat \be_y(\hat \bu)), \be(\boldsymbol{\psi}_1)  + \hat \be_y(\boldsymbol{\psi}_2) \rangle_{\Omega_T\times \hat Y}  
= \big[\langle \mathbf{f}, \boldsymbol{\psi}_1 \rangle_{ \Gamma_{\cE\cU,T}} -
\langle p_{\cI}\bnu, \boldsymbol{\psi}_1 \rangle_{ \Gamma_{\cI,T}} \big].
\end{equation*}
 Taking $\boldsymbol{\psi}_1 \equiv \textbf{0}$ yields 
 \begin{equation}\label{macro_23}
\langle \bbE(b,y) (\be(\bu) + \hat \be_y(\hat \bu)),  \hat \be_y(\boldsymbol{\psi}_2) \rangle_{\Omega_T\times \hat Y}  
=0.
\end{equation}
  The structure of  equations \eqref{macro_23} implies 
$$
\hat \bu(t,x,y) = \frac 12 \sum_{i,j=1}^3  \Big(\frac{\partial u_i(t,x)}{\partial x_j} + \frac{\partial u_j(t,x)}{\partial x_i} \Big) \textit{\bf{w}}^{ij}(t,x,y),
$$
where   $\textit{\bf w}^{ij}$ are solutions of the unit cell problems \eqref{unit_u}.  
Considering $\boldsymbol{\psi}_2\equiv \textbf{0}$ implies  the macroscopic equations for $\bu$. Similar to  \cite[Theorem II.1.1]{OShY} we obtain the symmetries and   the strong ellipticity of   $\bbE_{\rm hom}$. 

In the same manner as for the microscopic model, by deriving a contraction inequality we prove 
the uniqueness of a weak solution of the macroscopic equations and obtain that  the whole sequence of   microscopic solutions converges to a solution of the macroscopic problem.  Here we use  that due to the boundedness of $b$ for  solutions of the unit cell problem \eqref{unit_u} we have ${\bf w}^{ij} \in L^\infty(\Omega_T; H^1_{\rm per}(\hat Y))^3$, for $i,j=1,2,3$.  The contraction inequality and a fixed-point argument also  ensure the existence of a solution of the macroscopic equations \eqref{macro_1}--\eqref{macro_bc_1}. 
 \end{proof}

 \section{Convergence results and the derivation of the macroscopic equations for  Model~II. }\label{macro_model_II}
 Considering the  extension of   solutions of the microscopic problem \eqref{sumbal},~\eqref{BC1}--\eqref{BC4},  given by Lemma~\ref{lem:extension},   we have the following convergence results. 
\begin{lemma}\label{converg_2}
 There exist $\p, \, \n \in \big( L^2(0,T; \cV(\Omega))\cap L^\infty(0,T; L^\infty(\Omega))\big)^2$,  $\hat \p, \,  \hat  \n \in L^2(\Omega_T; H^1_{\rm per}(\hat Y_M)/\mathbb R)^2$,
  $b \in  L^2(0,T; \cV(\Omega))\cap L^\infty(0,T; L^\infty(\Omega))$, $\hat b \in L^2(\Omega_T; H^1_{\rm per}(\hat Y_M)/\mathbb R)$, 
   and   $\bu \in L^\infty(0,T; \cW(\Omega))$, $\hat \bu \in L^2(\Omega_T; H^1_{\rm per}(\hat Y)/\mathbb R)^3$, 
   such that for a  subsequence $(\p^\ve, \n^\ve, b^\ve,  \bu^\ve)$   of the sequence of solutions  of the microscopic problem \eqref{sumbal},~\eqref{BC1}--\eqref{BC4} (denoted again by $(\p^\ve, \n^\ve, b^\ve,  \bu^\ve)$) we have 
 \begin{equation}
 \begin{aligned}
 &\p^\ve  \rightharpoonup  \p, && \n^\ve  \rightharpoonup  \n,&& b^\ve  \rightharpoonup  b  &&\text{ weakly in }  L^2(0,T; H^1(\Omega)),   \\
  &\p^\ve  \to  \p, && \n^\ve \to  \n, &&  b^\ve \to  b  && \text{ strongly in } L^2(\Omega_T) \\ 
 &\p^\ve  \to  \p, && \n^\ve \to  \n, &&  b^\ve \to  b  && \text{ strongly in }    L^2((0,T)\times \partial \Omega),  \\
 &\nabla \p^\ve \rightharpoonup  \nabla \p + \hat \nabla_y \hat \p, && \nabla \n^\ve \rightharpoonup \nabla \n + \hat \nabla_y \hat  \n, && \nabla b^\ve \rightharpoonup \nabla b + \hat \nabla_y \hat b  && \text{ two-scale}, \\ 
 & \bu^\ve  \rightharpoonup  \bu && &&  && \text{ weakly}^\ast \text{ in }  L^\infty(0,T;  \cW(\Omega)),  \\
& \nabla \bu^\ve \rightharpoonup \nabla \bu + \hat \nabla_y \hat \bu && &&  && \text{ two-scale}.
 \end{aligned}
 \end{equation}
 \end{lemma}
\begin{proof}
The convergence results for $\p^\ve$ and $\bu^\ve$  are obtained in the same way as in Lemma~\ref{converg_1}. 
Due to  the \textit{a priori} estimates  derived  in Theorem~\ref{existence_3} and the compactness theorem for the two-scale convergence,  
we obtain the weak and two-scale convergence for $\n^\ve$ and $b^\ve$.  
The estimates for  $\|\nabla \n^\ve\|_{L^2(\Omega_T)}$ and $\|\nabla b^\ve\|_{L^2(\Omega_T)}$,  as well as for $\|\n^\ve(\cdot+h, \cdot) - \n^\ve(\cdot, \cdot)\|_{L^2(\Omega_{T-h})}$ and $\|b^\ve(\cdot+h, \cdot) - b^\ve(\cdot, \cdot)\|_{L^2(\Omega_{T-h})}$, obtained  from the last estimate in \eqref{estim_un_3} and the linearity of the extension,   together with  the Kolmogorov compactness theorem \cite{Brezis, necas}, ensure the strong convergence of $\n^\ve$ and $b^\ve$ in $L^2(\Omega_T)$.  As in the proof of Lemma~\ref{converg_1},    the compactness of the embedding $H^1(\Omega) \subset H^{\varsigma}(\Omega)$, for $\varsigma < 1$, and the compactness result in \cite{Simon} ensure  the strong convergence in  $L^2((0,T)\times\partial \Omega)$. The boundedness of $\n^\ve$ and $b^\ve$, together with convergence in $L^2(\Omega_T)$,  implies the boundedness of $\n$  and $b$.
\end{proof}

Next we derive the macroscopic equations \eqref{macro_1}, \eqref{macro_bc_1}, \eqref{macro_2}, \eqref{macro_bc_2} for Model II, i.e.\ for equations  \eqref{sumbal},   \eqref{BC1}--\eqref{BC4}. 

\begin{proof}[\bf Proof of Theorem~\ref{th:macro_2}]  Since  the equations for $\p^\ve$ are the same in both microscopic problems Model~I and Model~II, the derivation of the  equations for $\p$ follows along the same line as in the proof of Theorem~\ref{th:macro_1}. Using the strong convergence of $b^\ve$, in the same way as in the proof of Theorem~\ref{th:macro_1},   we obtain the   equations for~$\bu$.

To derive  the macroscopic equations for $\n$ and $b$ we have to show the strong two-scale convergence of $\be(\bu^\ve)$ in order to pass to the limit in the nonlinear functions $\Q_n(\n^\ve, b^\ve, \be(\bu^\ve))$ and $Q_b(\n^\ve, b^\ve, \be(\bu^\ve))$.  Consider $\bu^\ve$ as a test function    in \eqref{weak_sol_u}. 
Then,   using the lower-semicontinuity of a norm, the positive definiteness of $\bbE^\ve(b^\ve,x)$, the uniform in $\ve$ boundedness of $b^\ve$ in $L^\infty(0, T; L^\infty(\Omega_{M}^\ve))$, together with the weak convergence of $\bu^\ve$ in $L^2(0,T; \cW(\Omega))$,  the two-scale  convergence of $\be(\bu^\ve)$,  and the strong convergence of $b^\ve$,  we obtain  
\begin{equation}
\begin{aligned}
&|\hat Y|^{-1}  \big\langle \bbE(b, y) (\be(\bu) + \hat  \be_{y}(\hat \bu)), \be(\bu) +\hat  \be_{y}(\hat \bu) \big\rangle_{\Omega_T, \hat Y} 
 \leq \liminf_{\ve \to 0 } \langle \bbE^\ve(b^\ve,x) \be(\bu^\ve), \be(\bu^\ve) \rangle_{\Omega_T} \\
&\leq \limsup_{\ve \to 0 } \langle \bbE^\ve(b^\ve,x) \be(\bu^\ve), \be(\bu^\ve) \rangle_{\Omega_T} \hspace{-0.1 cm }  =  \hspace{-0.05 cm} 
\lim\limits_{\ve \to 0} \big[\langle  \bff, \bu^\ve \rangle_{\Gamma_{\cE\cU, T}} \hspace{-0.15 cm }  - \langle  p_\cI \bnu, \bu^\ve \rangle_{\Gamma_{\cI, T}} \big]  \hspace{-0.1 cm} =  \hspace{-0.05 cm} 
\langle  \bff, \bu \rangle_{\Gamma_{\cE\cU, T} } \hspace{-0.15 cm }  - \langle  p_\cI \bnu, \bu \rangle_{\Gamma_{\cI, T}}. 
\end{aligned}
\end{equation}
Taking $\bu$ as a test function  in   \eqref{macro_1}
and $\hat \bu$ as a test function  in   \eqref{macro_23}, and using the definition of $\bbE_{\rm{hom}}$,  yields 
 \begin{equation*}
\begin{aligned}
 \langle  \bff, \bu \rangle_{\Gamma_{\cE\cU, T} }   - \langle  p_\cI \bnu, \bu \rangle_{\Gamma_{\cI, T}} =
|\hat Y|^{-1}\big \langle \bbE(b, y) (\be(\bu) + \hat \be_y(\hat \bu)), \be(\bu) + \hat \be_y(\hat \bu)\big \rangle_{\Omega_T, \hat Y}. 
 \end{aligned}
\end{equation*}
Hence we obtain that 
\begin{equation}\label{conver_two_strong_1}
\begin{aligned}
 \lim_{\ve \to 0 } \langle \bbE^\ve(b^\ve,x) \be(\bu^\ve), \be(\bu^\ve) \rangle_{\Omega_T} = |\hat Y|^{-1}  \big\langle \bbE(b, y) (\be(\bu) + \hat  \be_{y}(\hat \bu)), \be(\bu) +\hat  \be_{y}(\hat \bu) \big\rangle_{\Omega_T, \hat Y} 
\end{aligned}
\end{equation}
and, using the positive definiteness   of $\bbE$,  the strong convergence of $b^\ve$, and the lower semicontinuity of a norm, we conclude the strong two-scale convergence of $\be(\bu^\ve)$. Taking into account  the microscopic structure of   $\Omega_M^\ve$ and the definition of the unfolding operator,  from \eqref{conver_two_strong_1} we also have   
\begin{equation}\label{conver_strong_ener1}
\begin{aligned}
 |\hat Y|^{-1}\hspace{-0.1 cm } \lim\limits_{\ve \to 0 } \big \langle \bbE(\cT_\ve(b^\ve), y) \cT_\ve(\be(\bu^\ve)), \cT_\ve(\be(\bu^\ve) )\big \rangle_{\Omega_T,  \hat  Y}  \hspace{-0.1cm} = \hspace{-0.1 cm}  |\hat Y|^{-1}\big \langle \bbE(b, y) (\be(\bu) +\hat \be_y(\hat \bu)), \be(\bu) + \hat \be_y(\hat \bu) \big\rangle_{\Omega_T, \hat Y}. \hspace{-0.3 cm } 
 \end{aligned}
\end{equation}
To show the strong convergence of $\cT_\ve(\be(\bu^\ve))$ in $L^2(\Omega_T\times \hat Y)$ we consider 
\begin{equation}
\begin{aligned}
& \mathcal I_\ve = \left\langle \bbE( \cT_\ve(b^\ve), y)\big[ \cT_\ve(\be(\bu^\ve)) - \be(\bu) - \hat \be_y(\hat \bu)\big], \cT_\ve(\be(\bu^\ve)) - \be(\bu) - \hat \be_y(\hat \bu) \right\rangle_{\Omega_T, \hat Y}
\\  
& = \big\langle \bbE(\cT_\ve(b^\ve),y) \cT_\ve(\be(\bu^\ve)),  \cT_\ve(\be(\bu^\ve))\big\rangle_{\Omega_T, \hat Y} -\left \langle \bbE(\cT_\ve(b^\ve),y) \cT_\ve(\be(\bu^\ve)),  \be(\bu) +\hat \be_y(\hat \bu) \right\rangle_{\Omega_T, \hat Y}\\
&  -\left \langle \bbE(\cT_\ve(b^\ve),y)\big[  \be(\bu) +\hat  \be_y(\hat \bu)\big], \cT_\ve(\be(\bu^\ve)) \right \rangle_{\Omega_T, \hat Y}
 + \left \langle \bbE(\cT_\ve(b^\ve), y)\big[ \be(\bu) + \hat \be_y(\hat \bu)\big], \be(\bu) + \hat\be_y(\hat \bu) \right \rangle_{\Omega_T,  \hat Y}.
\hspace{-0.4 cm} \end{aligned}
\end{equation}
Then, the convergence in \eqref{conver_strong_ener1} and the strong ellipticity   of $\mathbb E$, together with the strong convergence of $b^\ve$ and weak convergence of $\mathcal T_\ve ( \be(\bu^\ve))$, ensured  by the two-scale convergence of $\be(\bu^\ve)$,   yield
\begin{equation*}
 \lim\limits_{\ve \to 0} \|\cT_\ve(\be(\bu^\ve)) - \be(\bu) - \hat \be_y(\hat \bu)\|_{L^2(\Omega_T\times \hat Y)}  
 \leq  \omega_E^{-1} \lim\limits_{\ve \to 0}  \mathcal I_\ve  = 0. 
\end{equation*}
The {\it a priori} estimates for   $\bu^\ve$, $\n^\ve$,  and $b^\ve$, see \eqref{estim_un_2} and \eqref{estim_un_3},   and the local Lipschitz continuity of $\Q_n$ and $Q_b$, see Assumption~\ref{assumptions}, ensure 
\begin{equation*}
\begin{aligned}
 \|\cT_\ve(\Q_n(\n^\ve, b^\ve, \be(\bu^\ve)) ) \|_{L^2(\Omega_T\times \hat Y_M)}  +   \|\cT_\ve(Q_b(\n^\ve, b^\ve, \be(\bu^\ve)) ) \|_{L^2(\Omega_T\times \hat Y_M)}& \leq C_1
    \end{aligned}
\end{equation*}
and 
\begin{equation*}
\begin{aligned}
& \|\cT_\ve(\Q_n(\n^\ve, b^\ve, \be(\bu^\ve)) )  -\Q_n(\n, b, \be(\bu) + \hat \be_y(\hat \bu)) \|_{L^1(\Omega_T\times \hat Y_M)} 
\\
& + \|\cT_\ve(Q_b(\n^\ve, b^\ve, \be(\bu^\ve)) )  - Q_b(\n, b, \be(\bu) + \hat \be_y(\hat \bu)) \|_{L^1(\Omega_T\times \hat Y_M)} 
\\& \leq  C_2(1+\|\be(\bu^\ve)\|_{L^2(\Omega_T)})\big[ \|\cT_\ve(\n^\ve) - \n\|_{L^2(\Omega_T\times \hat Y_M)}+\|\cT_\ve(b^\ve) - b\|_{L^2(\Omega_T\times \hat Y_M)}\big] \\
& +  C_3 \big[1+\|\n\|_{L^2(\Omega_{T})} +\|b\|_{L^2(\Omega_{T})} \big]\|\cT_\ve(\be(\bu^\ve)) - \be(\bu) - \hat \be_y(\hat \bu)\|_{L^2(\Omega_T\times \hat Y_M)},
 \end{aligned}
\end{equation*}
where the constants $C_1$, $C_2$ and $C_3$ are independent of $\ve$.  The   strong  convergence of $\cT_\ve(\be(\bu^\ve))$, $\n^\ve$,  and $\cT_\ve(b^\ve)$  implies 
\begin{equation*} 
\begin{aligned}
&\lim\limits_{\ve \to 0} \big \langle \cT_\ve(\Q_n(\n^\ve, b^\ve,  \be(\bu^\ve)) ),   \boldsymbol{\psi} \big \rangle_{\Omega_T, \hat Y_M} = 
\big \langle  \Q_n(\n, b, \be(\bu) + \hat \be_y(\hat \bu)),   \boldsymbol{\psi} \big \rangle_{\Omega_T, \hat Y_M}  \\
 &+ \lim\limits_{\ve \to 0}\big\langle  \cT_\ve(\Q_n(\n^\ve, b^\ve, \be(\bu^\ve)) )  -\Q_n(\n, b, \be(\bu) + \hat \be_y(\hat \bu)) ,  \boldsymbol{\psi} \big \rangle_{\Omega_T, \hat Y_M} 
  = \big \langle  \Q_n(\n, b, \be(\bu) +\hat  \be_y(\hat \bu)),   \boldsymbol{\psi} \big \rangle_{\Omega_T, \hat Y_M}  
 \end{aligned}
\end{equation*}
for all $\boldsymbol{\psi} \in C_0(\Omega_T \times \hat Y_M)^2$. A similar convergence result holds also for $\cT_\ve(Q_b(\n^\ve, b^\ve,  \be(\bu^\ve)) )$.  Hence, we conclude 
\begin{equation*} 
\begin{aligned}
\cT_\ve(\Q_n(\n^\ve, b^\ve, \be(\bu^\ve))) \rightharpoonup \Q_n(\n, b, \be(\bu) + \hat \be_y(\hat \bu)) \quad \text{ weakly  in } L^2(\Omega_T\times  \hat Y_M), \\
\cT_\ve(Q_b(\n^\ve, b^\ve, \be(\bu^\ve))) \rightharpoonup Q_b(\n, b, \be(\bu) + \hat \be_y(\hat \bu)) \quad \text{ weakly  in } L^2(\Omega_T\times  \hat Y_M)
 \end{aligned}
\end{equation*}
and, due to the relation between the two-scale convergence of a sequence and  the weak convergence of the unfolded sequence, see e.g.\  \cite{CDDGZ}, we  also obtain  
\begin{equation}\label{conv_Q}
\begin{aligned}
\Q_n(\n^\ve, b^\ve, \be(\bu^\ve)) \rightharpoonup \Q_n(\n, b, \be(\bu) + \hat \be_y(\hat \bu)) \qquad \text{ two-scale}, \\
Q_b(\n^\ve, b^\ve, \be(\bu^\ve)) \rightharpoonup Q_b(\n, b, \be(\bu) + \hat \be_y(\hat \bu)) \qquad \text{ two-scale}. 
\end{aligned}
\end{equation}
Considering $\boldsymbol{\phi}_n(t,x)= \boldsymbol{\varphi}_n(t,x) + \ve \boldsymbol{\psi}_n(t,x, \hat x/\ve)$ and 
$\phi_b(t,x) = \varphi_b(t,x) + \ve \psi_b( t,x, \hat x/\ve)$, 
with $\boldsymbol{\varphi}_n \in C^1(\overline \Omega_T)^2$,  ${\varphi}_b \in C^1(\overline \Omega_T)$, $\boldsymbol{\varphi}_n(T,x) ={\bf 0}$, $\varphi_b(T,x)=0$ for $x\in \overline \Omega$,   and  $\boldsymbol{\varphi}_n$, $\varphi_b$ are $a_3$-periodic in $x_3$,  $\boldsymbol{\psi}_n\in C^1_0(\Omega_T; C^1_{\rm per}(\hat Y))^2$, $\psi_b \in C^1_0(\Omega_T; C^1_{\rm per}(\hat Y))$,    as  test functions in  the equations for $\n^\ve$ and $b^\ve$ in  \eqref{weak_sol_n3}  and using the convergence results in Lemma~\ref{converg_2},  together with  the two-scale convergence of 
$\Q_n(\n^\ve, b^\ve, \be(\bu^\ve))$ and $Q_b(\n^\ve, b^\ve, \be(\bu^\ve))$,  see \eqref{conv_Q},      we obtain 
\begin{eqnarray}\label{macro_model_II_1}
\begin{aligned}
-\theta_M \langle  \n,  \partial_t\boldsymbol{\varphi}_n \rangle_{\Omega_T}  +\frac 1{|\hat Y|}  \langle D_n( \nabla \n + \hat \nabla_y \hat \n), \nabla \boldsymbol{\varphi}_n  + \hat \nabla_y \boldsymbol{\psi}_n \rangle_{\Omega_{T}, \hat Y_M}  = \theta_M \langle \n_{0}, \boldsymbol{\varphi}_n(0, \cdot)\rangle_{\Omega} + \langle \J_n(\n), \boldsymbol{\varphi}_n \rangle_{\Gamma_{\cE, T}}
   \\
 + \frac 1{ |\hat Y|}   \langle  \F_n(\p, \n)+  \Q_n(\n, b, \be(\bu)+ \hat \be_y(\hat \bu)), \boldsymbol{\varphi}_n \rangle_{\Omega_{T}, \hat Y_M}
+  \big\langle   \mathcal N^{\rm eff}_\delta(\be(\bu))\, \G(\n), \boldsymbol{\varphi}_n \big\rangle_{\Gamma_{\cI,T}}  , \\ 
-\theta_M \langle  b,  \partial_t\varphi_b \rangle_{\Omega_T}  + \frac 1{|\hat Y| }  \langle D_b( \nabla b + \hat \nabla_y \hat b), \nabla \varphi_b  + \hat \nabla_y \psi_b \rangle_{\Omega_{T}, \hat Y_M}= \theta_M \langle b_0, \varphi_b(0, \cdot) \rangle_{\Omega}  \\ +  \frac 1{|\hat Y|}  \big\langle Q_{b}(\n, b, \be(\bu)+ \hat \be_y(\hat \bu)) , \varphi_b \big\rangle_{\Omega_T,\hat Y_M},
\end{aligned}
\end{eqnarray}
where $\theta_M= |\hat Y_M|/|\hat Y|$ and $ \mathcal N^{\rm eff}_\delta(\be(\bu))$ is defined in \eqref{macro_over_n_eu}.  As in the proof of Theorem~\ref{th:macro_1}, choosing  $\boldsymbol{\varphi}_n \equiv {\bf 0}$ and   $ {\varphi}_b \equiv 0$  we obtain the unit cell problems  \eqref{unit_n},  \eqref{unit_b} and the macroscopic  diffusion coefficients $\mathcal D_n^l$ and $\mathcal D_b$, with $l=1,2$.  Taking $\boldsymbol{\psi}_n \equiv {\bf 0}$,  ${\psi}_b \equiv 0$ and considering
 first $\boldsymbol{\varphi}_n \in C^1_0(\Omega_T)^2$,  $\varphi_b \in C^1_0(\Omega_T)$ and 
then $\boldsymbol{\varphi}_n \in C^1_0(0,T; C^1(\overline \Omega))^2$, $\varphi_b \in C^1_0(0,T; C^1(\overline \Omega))$, with  $\boldsymbol{\varphi}_n$,  $\varphi_b$  being $a_3$-periodic in $x_3$,  we obtain the 
macroscopic equations and boundary conditions in \eqref{macro_bc_1}, \eqref{macro_2}, and \eqref{macro_bc_2}.
The equations \eqref{macro_2} and the regularity of $\n$, $b$ and $\bu$ imply  $\partial_t \n  \in L^2(0,T; \cV(\Omega)^\prime)^2$ and 
$\partial_t b  \in L^2(0,T; \cV(\Omega)^\prime)$.  
Thus  $\n \in C([0,T]; L^2(\Omega))^2$ and $b \in C([0,T]; L^2(\Omega))$, and using \eqref{macro_model_II_1} we obtain    $\n(t, \cdot) \to \n_0$ and $b(t, \cdot) \to b_0$  in $L^2(\Omega)$ as $t \to 0$. 

To show the uniqueness of a solution of the macroscopic problem  \eqref{macro_1}, \eqref{macro_bc_1},  \eqref{macro_2}, 
and \eqref{macro_bc_2} we consider  the  equations for the difference of two solutions.  Using the boundedness of $\p$ and the local Lipschitz continuity of $\F_{p}$ we obtain the uniqueness of a solution of the subsystem for $\p$.   In the same way as in the proof of Lemma~\ref{lemub} and Theorem~\ref{existence_3}  we obtain
 \begin{equation}\label{contr_macro_in}
\begin{aligned}
  \|\widetilde \n   \|_{L^\infty(0,\tau; L^2(\Omega))} + \|\nabla \widetilde \n  \|_{L^2(\Omega_{\tau})}+ \|\widetilde b \|_{L^\infty(0,\tau; L^2(\Omega))} + \|\nabla \widetilde  b  \|_{L^2(\Omega_{\tau})}  &\leq  C  \|\be(\widetilde\bu) \|_{L^2(\Omega_{\tau})},\\
\|\widetilde b \|_{L^\infty(0,\tau; L^\infty(\Omega))} &\leq C \|\be(\widetilde \bu)\|_{L^{1+1/\varsigma}(0,\tau; L^2(\Omega))}, \\
 \|\be(\widetilde\bu)\|_{L^{\infty}(0,\tau; L^2(\Omega))} &\leq C \|\widetilde b  \|_{L^\infty(0,\tau; L^\infty(\Omega))},  
 \end{aligned}
\end{equation}
for $0<\varsigma < 1/9$ and $\tau \in (0,T]$, where $\widetilde \n= \n^1-\n^2$,  $\widetilde b= b^1-b^2$,  $\widetilde \bu= \bu^1-\bu^2$, and   $(\n^1, b^1, \bu^1)$ and $(\n^2, b^2, \bu^2)$ are two solutions of the macroscopic problem  \eqref{macro_1}, \eqref{macro_bc_1},  \eqref{macro_2}, 
and \eqref{macro_bc_2}.  Then considering $\tau$ sufficiently small and iterating  over time-intervals yield  the uniqueness of a weak solution of  the coupled system \eqref{macro_1}, \eqref{macro_bc_1},  \eqref{macro_2}, 
and \eqref{macro_bc_2}  in $\Omega_T$.  The  inequalities \eqref{contr_macro_in} together with a fixed-point argument also ensure  the well-posedness of the macroscopic problem. 
\end{proof}

 \section{Analysis of  the microscopic equations with  the reaction terms depending on strain. } 
It is possible to  assume that the breakage of calcium-pectin cross-links depends on the strain instead of stress: 
 \begin{equation}\label{def_GN_1}
 \begin{aligned}
& \mathcal N_\delta (\be(\bu^\ve))(t,x) =  \Big(\dashint_{B_\delta(x)\cap \Omega}\text{tr } \be(\bu^\ve)(t, \tilde x) \, d\tilde x\Big)^{+}  \qquad \text{for all } \; x\in \overline\Omega \text{ and } t \in (0,T), \\
& P (b^\ve, \be(\bu^\ve)) =  \left(\text{tr } \be(\bu^\ve) \right)^{+}, 
\end{aligned}
\end{equation}
with $\Q_n(\n^\ve, b^\ve)= \Q(\n^\ve, b^\ve) P (b^\ve, \be(\bu^\ve))$, see \eqref{def_G_2}. In this situation the analysis of   the microscopic problems follows  along the same lines as in Theorems~\ref{existence_2}~and~\ref{existence_3}.  Moreover,  in this situation the growth assumption on $\mathbb E_M$ is not needed. 
In the macroscopic equations,  see Theorems~\ref{th:macro_1}~and~\ref{th:macro_2},  we will have 
$$
\mathcal N^{\rm eff}_\delta(\be(\bu))= \Big(\dashint_{B_\delta(x)\cap \Omega} {\rm tr } \, \be(\bu) \, d\tilde x \Big)^{+} \quad  \text{ and } \quad 
P^{\rm eff} (b, \mathbb W\be(\bu)) =  \left(\text{tr } \mathbb W \be(\bu) \right)^{+},
 $$
where $\mathbb W_{ijkl}(t,x,y) = \delta_{ik}\delta_{jl} + \big( \hat\be_y(\textbf{w}^{ij} (t,x,y)) \big)_{kl}$ and  $\textbf{w}^{ij}$ being  solutions of the unit cell problems \eqref{unit_u}. The proof of the  strong convergence of $\cT_\ve(b^\ve) $ in the case where $\mathcal N_\delta(\be(\bu^\ve))$ depends on the strain can be conducted in a different way.
\begin{lemma}\label{strong_nb_1} 
In the case of Model I,   \eqref{sumbal}--\eqref{sumbal2},  and $\mathcal N_\delta(\be(\bu^\ve))(t, x)=\Big[\dashint_{B_\delta(x)\cap \Omega }{\rm tr }\, \be(\bu^{\ve}(t, \tilde x)) d\tilde x\Big]^{+}$  for $x \in \overline \Omega$ and $t \in (0,T)$, we have, up to a subsequence,  
$$
\cT_\ve(b^\ve)  \to b  \quad  \text{ strongly}  \text{ in } L^2(\Omega_T\times \hat Y_M)  \qquad  \text{ as } \ve \to 0.
$$
\end{lemma} 
\begin{proof}
To show the strong convergence of $\cT_\ve(b^\ve)$ in $L^2(\Omega_T\times \hat Y_M)$, we  prove that $\cT_\ve(b^\ve)$ is a Cauchy sequence in  $L^2(\Omega_T\times \hat Y_M)$. 
In  the proof of  the Cauchy property for the sequence $\cT_\ve(b^\ve)$ we shall use the Gronwall inequality. Thus,  we consider equations for $b^\ve$ in $(0,\tau)\times \Omega_M^\ve$, for any $\tau \in (0,T]$.  

Applying  the unfolding operator $\cT_\ve$ to the equation for $b^\ve$ in  \eqref{sumbal_11}  and  taking $\mathcal T_{\ve_m}(b^{\ve_m}) - \mathcal T_{\ve_k}(b^{\ve_k})$ as a test function in the  difference of the equations for $\mathcal T_{\ve_m}(b^{\ve_m}) $ and $\mathcal T_{\ve_k}(b^{\ve_k})$ yield
\begin{equation*}
\begin{aligned}
&  \|\mathcal T_{\ve_m}(b^{\ve_m})(\tau)  - \mathcal T_{\ve_k}(b^{\ve_k})(\tau)\|^2_{L^2(\Omega\times \hat Y_M)}=  \|\mathcal T_{\ve_m}(b_0)  - \mathcal T_{\ve_k}(b_0)\|^2_{L^2(\Omega\times \hat Y_M)}+ 2(I_1+I_2+I_3), 
 \end{aligned}
\end{equation*}
where 
\begin{equation*}
\begin{aligned}
&I_1= \\
&\langle R_{b}(\cT_{\ve_m}(\n^{\ve_m}), \mathcal T_{\ve_m}(b^{\ve_m}), \cT_{\ve_m}\mathcal N_\delta(\be(\bu^{\ve_m})) )  - R_{b}(\cT_{\ve_k}(\n^{\ve_k}),  \mathcal T_{\ve_m}(b^{\ve_m}),\cT_{\ve_m}\mathcal N_\delta(\be(\bu^{\ve_m}))), \delta^{\ve_{m,k}}\mathcal T_{\ve}(b^{\ve}) \rangle_{\Omega_\tau\times \hat Y_M},
\\
&I_2= \hspace{-0.05 cm} \langle R_{b}(\cT_{\ve_k}(\n^{\ve_k}), \mathcal T_{\ve_k}(b^{\ve_k}), \cT_{\ve_m}\mathcal N_\delta(\be(\bu^{\ve_m})) )  - R_{b}(\cT_{\ve_k}(\n^{\ve_k}),  \mathcal T_{\ve_k}(b^{\ve_k}),\cT_{\ve_k}\mathcal N_\delta(\be(\bu^{\ve_k}))), \delta^{\ve_{m,k}}\mathcal T_{\ve}(b^{\ve}) \rangle_{\Omega_\tau\times \hat Y_M}, \\
&I_3=\\
&\langle R_{b}(\cT_{\ve_k}(\n^{\ve_k}), \mathcal T_{\ve_m}(b^{\ve_m}), \cT_{\ve_m}\mathcal N_\delta(\be(\bu^{\ve_m})) )  - R_{b}(\cT_{\ve_k}(\n^{\ve_k}),  \mathcal T_{\ve_k}(b^{\ve_k}),\cT_{\ve_m}\mathcal N_\delta(\be(\bu^{\ve_m}))), \delta^{\ve_{m,k}}\mathcal T_{\ve}(b^{\ve}) \rangle_{\Omega_\tau\times \hat Y_M},
 \end{aligned}
\end{equation*}
with  $\delta^{\ve_{m,k}}\mathcal T_{\ve} (b^{\ve})=  \mathcal T_{\ve_m}(b^{\ve_m}) - \mathcal T_{\ve_k}(b^{\ve_k})$ and $\tau \in (0, T]$. 
Using the strong convergence of $\n^\ve$  in $L^2(\Omega_T)$, the assumptions on $R_b$,  and the  boundedness of $\n^\ve$ and $b^\ve$ for the first term  we have 
\begin{equation*}
\begin{aligned}
|I_1| \leq  \sigma_1(\ve_m, \ve_k)  + \|\mathcal T_{\ve_m}(b^{\ve_m}) - \mathcal T_{\ve_k}(b^{\ve_k})\|^2_{L^2(\Omega_\tau\times\hat  Y_M)},
 \end{aligned}
\end{equation*}
where $\sigma_1(\ve_m, \ve_k) \to 0$ as $\ve_m, \ve_k \to 0$. The boundedness of $\n^\ve$ and $b^\ve$   yield
\begin{multline*}
|I_2|
\leq C\big[\| \cT_{\ve_m} (\mathcal N_\delta(\be(\bu^{\ve_m})))  - \cT_{\ve_k}( \mathcal N_\delta(\be(\bu^{\ve_k}))) \|^2_{L^2(\Omega_\tau\times \hat Y_M)}  
+ \|\mathcal T_{\ve_m}(b^{\ve_m}) - \mathcal T_{\ve_k}(b^{\ve_k}) \|^2_{L^2(\Omega_\tau\times \hat Y_M)}\big] \\ = C (I_{21}+ I_{22}) .
\end{multline*}
Using the properties of the unfolding operator $\cT_\ve$, i.e.\   $\cT_\ve (\phi) \to \phi$ strongly in $L^2(\Omega\times \hat Y)$ for $\phi \in L^2(\Omega)$ and $\|\cT_\ve (\phi) \|^2_{L^2(\Omega \times \hat Y)} \leq |\hat Y| \| \phi \|^2_{L^2(\Omega)}$,  we have
\begin{equation*}
\begin{aligned}
I_{21}
&\leq 
\sum_{j=m,k}\Big[\| \cT_{\ve_j} \mathcal N_\delta(\be(\bu^{\ve_j}))  -\cT_{\ve_j}  \mathcal N_\delta(\be(\bu)) \|^2_{L^2(\Omega_\tau\times \hat Y)}+ \| \cT_{\ve_j} \mathcal N_\delta(\be(\bu))  - \mathcal N_\delta(\be(\bu)) \|^2_{L^2(\Omega_\tau\times \hat Y)}\Big] \\
&\leq |\hat Y|\sum_{j=m,k}\| \mathcal N_\delta(\be(\bu^{\ve_j}))  - \mathcal N_\delta(\be(\bu)) \|^2_{L^2(\Omega_\tau)} + \sigma_2(\ve_m, \ve_k),
 \end{aligned}
\end{equation*}
where $\sigma_2(\ve_m, \ve_k) \to 0$ as $\ve_m, \ve_k \to 0$. Applying the weak convergence of $\be(\bu^\ve)$ and the strong convergence of $\int_{B_\delta(x)\cap \Omega}\be(\bu^{\ve}) d\tilde x$, see \eqref{conv_int_eu_2},  yields 
\begin{equation*}
\begin{aligned}
&\| \mathcal N_\delta(\be(\bu^{\ve_j}))  - \mathcal N_\delta(\be(\bu)) \|^2_{L^2(\Omega_\tau)}   \leq 
\int_0^\tau\int_\Omega \Big|\dashint_{B_\delta(x)\cap \Omega} \big[\be(\bu^{\ve_j}(t,\tilde x)) - \be(\bu(t,\tilde x))\big] d\tilde x\Big|^2 dx dt 
\leq C \delta^{-6}\sigma_3(\ve^j),
 \end{aligned}
\end{equation*}
where $\sigma_3(\ve_j) \to 0$ as $\ve_j \to 0$, with $j=m,k$.
We estimate  $I_3$ as 
\begin{equation*}
\begin{aligned}
|I_3| &\leq  C \big[1+\sup_{\Omega_T} \mathcal N_\delta(\be(\bu^{\ve_m}))\big] \|\mathcal T_{\ve_m}(b^{\ve_m}) - \mathcal T_{\ve_k}(b^{\ve_k}) \|^2_{L^2(\Omega_\tau\times \hat Y_M)} .
 \end{aligned}
\end{equation*}
Combining  the estimates for $I_1$, $I_{21}$, and $I_3$  and applying the Gronwall inequality we obtain 
$$
\sup\limits_{(0,T)}\|\mathcal T_{\ve_m}(b^{\ve_m}) - \mathcal T_{\ve_k}(b^{\ve_k}) \|^2_{L^2(\Omega\times \hat Y_M)} 
\leq \sigma_4(\ve_m, \ve_k), 
$$
where $\sigma_4(\ve_m, \ve_k) \to 0$ as $\ve_m, \ve_k \to 0$. Hence, we have that  $\{\mathcal T_{\ve}(b^{\ve})\}$ converges strongly  in $L^2(\Omega_T\times \hat Y_M)$.

The macroscopic equation for $b$ implies  that $b$ is independent of the microscopic variables $y \in \hat Y_M$. 
\end{proof}

\section{On the computation of the effective  elasticity tensor}\label{numerics} 

The macroscopic equations allow for the development of efficient numerical simulations of the models discussed above.  To solve the macroscopic problem  numerically, the first step is to obtain the macroscopic effective  diffusion coefficients and elasticity tensor.  To compute the effective  diffusion coefficients, one must solve the unit cell problems \eqref{unit_n} and \eqref{unit_b}. Since  in    \eqref{unit_n} and \eqref{unit_b} we have  Poisson equations defined on a unit cell $\hat Y_M$ and are independent of the macroscopic variables, standard numerical methods can be applied.  Solving \eqref{unit_u} to determine the effective elasticity tensor is more involved, since  the unit cell problems in \eqref{unit_u}  depend on the  microscopic variables $y\in \hat Y_M$ as well as on the macroscopic variables $x$ and the time variable $t$ through the density of the calcium-pectin cross-links $b$.  Thus, in general the system \eqref{unit_u} must be solved for each possible value of $b(t,x)$ for $(t,x) \in \Omega_T$.  However, under  physically reasonable assumptions on the elasticity tensor for the cell wall matrix, it can be shown that it is sufficient to solve \eqref{unit_u} for only two  values of $b$.

If the cell wall matrix is isotropic, then the effective elasticity tensor depends linearly on the Young's modulus of the wall matrix.  To see this, consider the definition of the homogenized elasticity tensor 
\beqn\label{homEl}
\bbE_{\text{hom},ijkl}(b)=\dashint_{\hat Y} \big[\bbE_{ijkl}(b,y)+\big(\bbE(b,y)\hat\be_y(\bw^{ij}(b))\big)_{kl}\big]dy,
\eeqn
where 
$$
\bbE(b,y)=\bbE_M(b)\chi_{Y_M}(y) + \bbE_F\chi_{Y_F}(y)
$$
and 
$\bw^{ij}$ are the unique solutions of the unit cell problems  \eqref{unit_u}.  Experiments suggest that only the Young's modulus $E=E(b)$ depends on the density of the calcium-pectin cross-links $b$, see e.g. \cite{ZMSR}.  Since the cell wall matrix is isotropic, the elasticity tensor of the cell wall matrix depends linearly on the Young's modulus, and thus is of the form
$$
\bbE_M(b)= \widetilde \bbE_M(E(b))=E(b)\bbE_1+\bbE_0. 
$$
 To prove that $\bbE_{\rm hom}$ also depends linearly on $E(b)$, we shall show that there exist $\bbE_{{\rm hom},1}$ and $\bbE_{{\rm hom},0}$ such that
\begin{equation}\label{claim}
\bbE_{{\rm hom}}(b)=E(b)\bbE_{{\rm hom},1}+\bbE_{{\rm hom},0}.
\end{equation}
For any positive numbers $\alpha$ and $\beta$, we define
$$
\bW_{kl}(\alpha, \beta)=\big(\widetilde \bbE(\alpha+\beta,y)\hat\be_y(\bw^{ij}(\alpha+\beta))\big)_{kl}-\big(\widetilde\bbE(\alpha,y)\hat\be_y(\bw^{ij}(\alpha))\big)_{kl}.
$$
It follows from \eqref{homEl} that for \eqref{claim} to hold it is sufficient to show that $\int_{\hat Y} \bW(\alpha,\beta) dy$ is linear in $\beta$ and independent of $\alpha$. It follows from  equations    \eqref{unit_u} and the $\hat Y$-periodicity of  $\bw^{ij}$ and $\bbE(b, \cdot)$  that  $\bW(\alpha,\beta)$ is $\hat Y$-periodic  and satisfies
\beqn\label{funitcell}
\hat{\text{div}}_y(\bW(\alpha,\beta)+\beta \bbE_1\bb^{ij}\chi_{\hat{Y}_M})=\textbf{0}  \quad  \text{in}\; \hat Y.
\eeqn
The Helmholtz--Hodge decomposition theorem for $L^2(\hat Y)^{3 \times 3}$ implies that $\bW(\alpha, \beta)$ has a unique representation
$$
\bW(\alpha,\beta)=(\hat{\text{curl}}_y\, \bU(\alpha,\beta))^{\rm T}+\bZ(\alpha,\beta),
$$
where $(\hat{\text{curl}}_y\, \bU(\alpha,\beta))^{\rm T}$ and $\bZ(\alpha, \beta)$ are $L^2$-orthogonal  periodic matrix functions, see e.g.\  \cite{MB_2001}.  Here $(\hat{\text{curl}}_y\bU)_{ij} =\epsilon_{i1k} \partial_{y_1}\bU_{jk}+\epsilon_{i2k} \partial_{y_2}\bU_{jk}$ with summation  over $k=1,2,3$ and $\epsilon$ beeing  the three dimensional  Levi-Civita symbol. Then  
 $\bZ(\alpha,\beta)$ is the unique $\hat Y$-periodic solution of
$$
\hat{\text{div}}_y(\bZ(\alpha,\beta)+\beta \bbE_1\bb^{ij}\chi_{\hat{Y}_M})=\textbf{0} \quad  \text{in}\; \hat Y, 
$$
that is orthogonal to the kernel of $\hat{\text{div}}_y$. Therefore $\bZ$ is  a linear function of $\beta$ and is independent of $\alpha$.  Moreover, it follows from the periodicity of $\bU$ that $\int_{\hat Y} (\hat{\text{curl}}_y\, \bU(\alpha,\beta))^{\rm T} dy = \textbf{0}$. Thus,  $\int_{\hat Y} \bW(\alpha,\beta) \, dy$ is independent of $\alpha$ and is linear in $\beta$.

Since $\bbE_\text{hom}$ is a linear function of the Young's modulus of the cell wall matrix, knowing $\bbE_\text{hom}$ for two different values of this modulus completely determines $\bbE_\text{hom}$.  Thus, given the Young's modulus of the cell wall matrix as a function of the calcium-pectin cross-link density, the macroscopic elasticity tensor can be computed for different cross-link densities by solving the unit cell problems \eqref{unit_u} for only two values of $b$.

As an example, we compute $\bbE_\text{hom}$ numerically for $E= 10$ MPa, which corresponds to $b=2.48\, \mu$M when $E(b)=0.775\, b+8.08$, see  \cite{ZMSR}.  Since the cell wall matrix is assumed to be isotropic, it suffices to specify the Poisson's ratio to completely determine the elasticity tensor of the wall matrix.  Here, we consider the Poisson's ratio to be equal to $0.3$, a common value for biological materials.  The microfibrils are transversely isotropic \cite{DMKM} and, hence, the elasticity tensor is determined by specifying five parameters: the Young's modulus $E_F$ associated with the directions lying perpendicular to the fibril, the Poisson's ratio $\nu_{F1}$ which characterizes the transverse reduction of the plane perpendicular to the microfibril for stress lying in this plane, the ratio $n_F$ between $E_F$ and the Young's modulus associated with the direction of the fibril, the Poisson's ratio $\nu_{F2}$ governing the reduction in the plane perpendicular to the micofibril for stress in the direction of the microfibril, and the shear modulus $Z_F$ for planes parallel to the fibril.  These parameters are assigned the values
$$
E_F = 15000\, \text{MPa},\ \nu_{F1}= 0.3,\ n_F = 0.068,\ \nu_{F2}=0.11,\ Z_F = 84842\, \text{MPa},
$$
which are chosen based on experimental results \cite{ZMSR}  and to ensure that the elasticity tensor for the microfibrils is strongly elliptic.
The computations involved the unit cell $\hat Y=(0,1)^2$ with  
$$
\hat Y_F = \{ (y_1,y_2)\in \hat Y\ |\ (y_1-0.5)^2+(y_2-0.5)^2 < 0.25^2\}.
$$
The unit cell problem \eqref{unit_u} was solved using FEniCS  \cite{fenics, dolfin, ffc}.  This involved the discretization of the domain $\hat Y$ using a nonuniform mesh with 15,991,809 vertices that had a higher density of vertices near the boundary between the cell wall matrix and the microfibrils.  The linear system, obtained  using the continuous Galerkin finite element method, was solved using the general minimal residual method with an algebraic multigrid preconditioner.

\begin{table}
\begin{center}
$$
{\bf C}=
\begin{pmatrix}
21.2 & 8.9 & 23.3 & 0 & 0 & 0\\
8.9 & 21.2 & 23.3 & 0 & 0 & 0\\
23.3 & 23.3 & 43367.5 & 0 & 0 & 0\\
0 & 0 & 0 & 14 & 0 & 0 \\
0 & 0 & 0 & 0 & 14 & 0 \\
0 & 0 & 0 & 0 & 0 & 5.7
\end{pmatrix}
$$
\end{center}
\caption{Effective macroscopic elasticity tensor for  the Young's modulus ${\rm E} = 10$ MPa.}
\label{table1}
\end{table}

Table~\ref{table1} shows the computed effective elasticity tensor expressed in Voigt notation.  When a larger value of $b$ is considered, the components of the resulting effective elasticity tensor are larger.  As can be seen, the macroscopic elasticity tensor possesses tetragonal symmetry. This is in agreement  with a general result on the symmetry of the effective coefficients, see \cite{PS_2015}.  The ${\bf C}_{33}$ component is several orders of magnitude larger than the other components of $\bC$ since it describes the resistivity of the cell wall to being stretched in the direction parallel to the microfibrils, and the microfibrils are much stiffer than the cell wall matrix.

\section{Derivation of the mathematical model}\label{sectdermodel}
In this section   Model I for  plant cell wall biomechanics is derived. The derivation of  Model II follows along the same lines.  We will   emphasise the differences between Model I and Model II at the end of the section.  

The primary wall of a plant cell consists mainly of oriented cellulose microfibrils, pectin, hemicellulose, structural proteins, and water, see Fig.~\ref{fig11}. The cross-linked pectin network is the main composite of the middle lamella which joins individual cells together. The main force for cell elongation (turgor pressure) acts isotropically, and so it is the microscopic structure of the cell wall which determines the anisotropic growth of plant cells and tissue. The orientation of microfibrils, their length, high tensile strength, and interaction with cell  wall matrix macromolecules strongly influence the wall's stiffness. Hemicelluloses  form hydrogen bonds with the surface of cellulose microfibrils, which may strengthen the cell wall by creating a microfibril-hemicellulose network, but also
 weaken the mechanical strength of cell walls by preventing cellulose aggregation \cite{Somerville}.   Pectin is deposited into cell walls in a methylesterified form, where it   can be modified  by the enzyme pectin methylesterase (PME), which removes methyl groups by breaking ester bonds. The de-esterified pectin  is able to form  calcium-pectin cross-links, and so stiffen the cell wall and reduce its expansion, see e.g.~\cite{WG,ZMSR}. 
 
 Thus, the biomechanics of plant cell walls is determined by the cell wall microstructure, given by   the microfibrils, and  the physical  properties of the cell
wall matrix. There are a number of models of a plant cell wall, each of which focuses on different aspects of its structure.
Mathematical models of the  cellulose-hemicellulose network were proposed in \cite{DBJ,PF}.  The account of the microstructure of a cell wall has been addressed by considering the anisotropic yield stresses or by
distinguishing between the free energies related to the elasticity of (i) macromolecules and hydrogen bonds or (ii)   the matrix and microfibrils \cite{Chaplain,Dumais,VC}.
 The influence of the microfibril orientation and the external torque on the expansion process has been considered in \cite{DJ}. The effect of changes in the chemical configurations of  pectins (methylesterified and demethylesterified) and  the calcium concentration on the viscous behavior of a cell wall in a pollen tube has been analyzed  in \cite{KZG,RHD}.

\begin{figure}
\begin{center}
\includegraphics[width=5.6 cm]{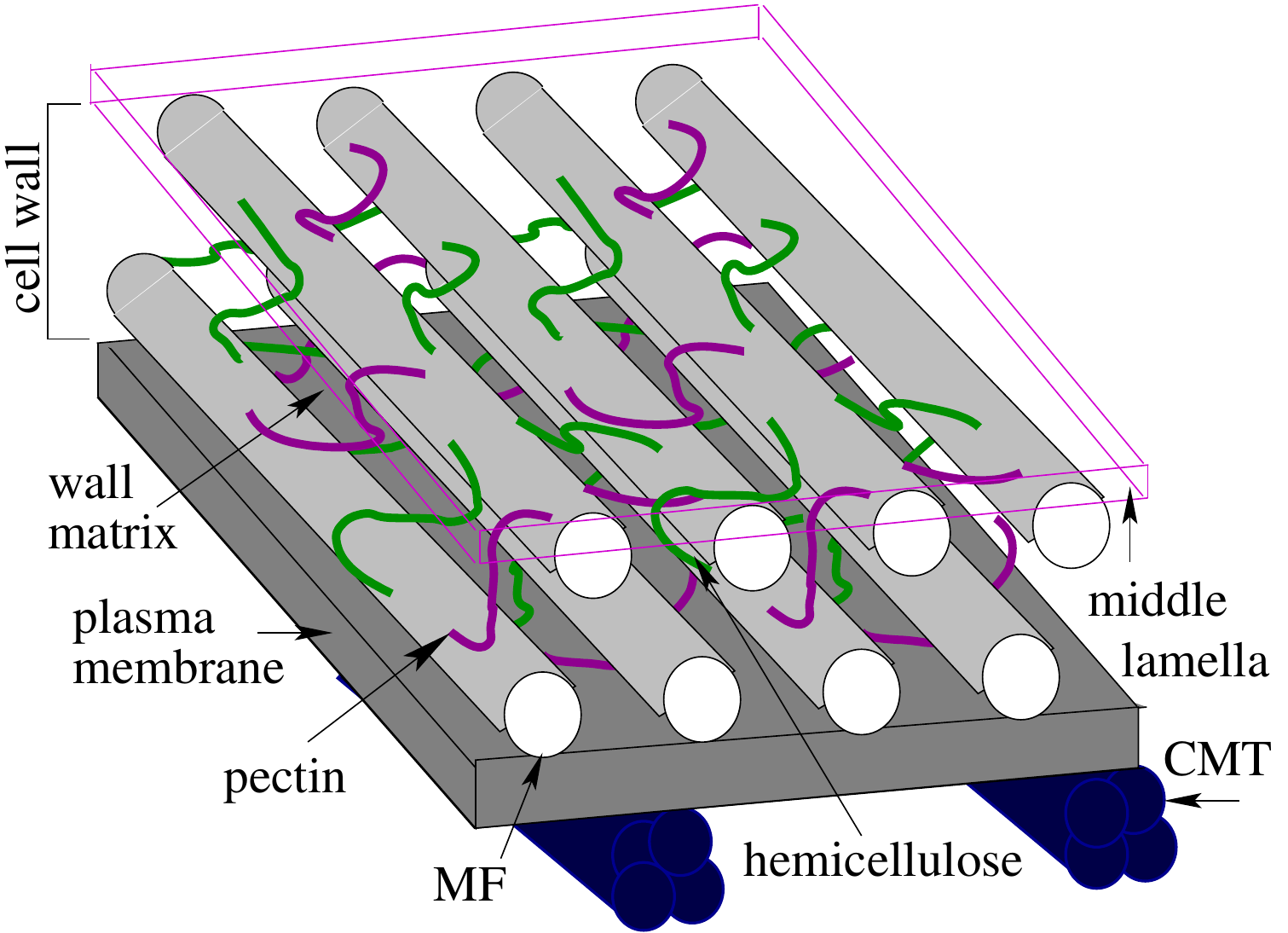}
\caption{Schematic diagram of a plant cell wall. MF denotes cell wall  microfibrils, CMT denotes cortical microtubules in a plant cell.}\label{fig11}
\end{center}
\end{figure}

In our model  we focus on two aspects  which have not been considered together before:  the influence of the microstructure, associated with the cellulose microfibrils,  and  of the calcium-pectin cross-links on the mechanical properties of  plant  cell walls.  In the  microscopic model of cell wall biomechanics derived  here, the cell wall microstructure and   the dynamics of the formation and dissociation  of calcium-pectin cross-links are considered explicitly.

   It is supposed that calcium-pectin cross-linking chemistry is one of the main regulators of cell wall elasticity and extension \cite{WHH}.
It has been shown that   the    modification of pectin  by PME and  the control of the amount of  calcium-pectin cross-links  greatly influence the  mechanical deformations of plant cell walls \cite{P2011,P2010},  and the interference with PME activity causes dramatic changes in growth behavior of plant cells and tissues \cite{Wolf}.  We  consider the most abundant  subclass of pectin, homogalacturonan, which is important for the regulation of plant biomechanics and growth.  Homogalacturonan  consists of a long linear chain of galacturonic acids. Pectin is deposited into the cell wall in a highly methylestrified state and then  is modified by the  enzyme pectin-methylesterase (PME), which removes methyl groups \cite{WG}.   The demethylesterified pectin  interacts with calcium ions to produce load bearing cross-links, which reduce cell wall expansion, see e.g.~\cite{WHH}.
 
In the mathematical model  a flat section of a cell wall composed of a polysaccharide matrix and  cellulose microfibrils is considered. Let $\Omega$ be a reference configuration of  the plant cell wall.  The domains $\Omega_M$ and $\Omega_F$ denote the parts of $\Omega$ occupied by the cell wall matrix and microfibrils, respectively.  

We  consider  five species within the  plant  cell  wall  matrix: methylestrified pectin, the enzyme PME, demethyl-estrified pectin, calcium ions, and calcium-pectin cross-links.   
 To form cross-links  with calcium ions Ca$^{2+}$, pectin molecules  need to have only some of their constituent acids  de-estrified, see e.g.~\cite{C,WG}.  Thus, when describing the density of pectin in the different states, we refer to the density of  the galacturonic acid groups in the different states.  

  Let $n_e$, $n_E$, $n_d$, $n_c$, and $n_b$ denote the number densities   of  methylestrified  pectin acid groups,  PME enzyme,  demethylestrified pectin acid groups, calcium ions, and  calcium-pectin cross-links in the reference configuration  $\Omega_M$, respectively.   Let $S=\{e,E, d,c,b\}$ be an index set and  $n_S$ will denote all five of the  densities. 
  We assume that  the densities $n_\alpha$, with $\alpha\in S$, are changing due to  spacial movement, reactions between the species, and external agencies.  
Thus, the balance equation  for  $n_\alpha$ is given by
\beqn\label{LSB}
\partial_t n_\alpha=r_\alpha-\text{div}\,\bj_\alpha+h_\alpha\qquad \text{in} \; \;  \Omega_M, \qquad \qquad \alpha\in S, 
\eeqn
where $r_\alpha$ models the chemical reactions between the species, $\bj_\alpha$ is the flux,  and  $h_\alpha$ is the species supply due to external agencies.  
The momentum balance for the cell wall reads
\beqn\label{LFB}
\textbf{0}=\text{div}\,\bT_\text{R}+\bb\qquad \text{in}\ \Omega,
\eeqn
where $\bT_R$ is the Piola stress and $\bb$ denotes  the external body forces, including inertial forces.
We consider    elastic behavior  of the  wall material and assume that the chemical processes  in the  wall matrix influence the mechanical properties of  the cell walls, see e.g.~\cite{C,WG}. 
The  constitutive law of linear elasticity, see e.g.~\cite{Ciarlet},  for the stress is assumed:
\beqn\label{conlaw1}
\bT_\text{R}= \big(\bbE_M(n_S)\chi_{\Omega_M} + \bbE_F \chi_{\Omega_F}\big)\be(\bu),
\eeqn
 where $\bbE_M(n_S)$ and $\bbE_F$ are elasticity tensors for cell wall matrix and microfibrils, respectively, and $\be(\bu)=\half(\nabla\bu+\nabla\bu^{\rm T})$ is the symmetric part of the displacement gradient, $\chi_A$ is the characteristic function of a domain~$A$. 

The   interactions between the mechanical properties of the cell wall  and the biochemistry of the wall matrix  are also reflected  in the reactions terms
\beqn\label{conlaw2}
r_\alpha=r_{\alpha 0}(n_S)+\bZ_\alpha(n_S)\cdot\widetilde \cN_\delta(\be(\bu)),
\eeqn
where 
\begin{equation*}
\widetilde{\mathcal N}_\delta (\be(\bu))(x) =\dashint_{B_\delta(x)\cap \Omega} \bbE(n_S) \be(\bu)(\tilde x) d\tilde x 
\end{equation*}
for all $x\in \overline\Omega$, with $\bbE(n_S) = \bbE_M(n_S) \chi_{\Omega_M} + \bbE_F \chi_{\Omega_F}$.
We assume  that  the  stress influences the chemical reactions and the dynamics of calcium-pectin cross-links \cite{PB}. 
 Since pectin are long molecules, we assume the nonlocal impact of cell wall mechanics on chemical processes. Thus, stresses within a neighborhood of a point affect the rate of the chemical reactions.    The length scale $\delta$ is associated with the length of the pectin molecules.

The flux of species $\alpha$ is assumed to be determined by Fick's law:
\beqn\label{Fick}
\bj_\alpha=-\bD_{\alpha}\nabla n_\alpha,
\eeqn
where $\bD_{\alpha}$ is the diffusion coefficient  of the species $\alpha$. 

Next,    we  specify  assumptions on the constitutive laws introduced in \eqref{conlaw1}--\eqref{Fick}  that reflect the physics of the plant cell wall.   The cell wall matrix has the same properties in all directions and, hence, is isotropic, see e.g.\ \cite{ZMSR}.  This is expressed mathematically by requiring 
\beqn\label{isoCL0}
\left.
\begin{aligned}
\bbE_M(n_S)(\bQ\be(\bu)\bQ^{\rm T})&=\bQ(\bbE_M(n_S)\be(\bu))\bQ^{\rm T},\\
\bZ_\alpha(n_S)&=\bQ\bZ_{\alpha}(n_S)\bQ^{\rm T},\\
\bD_{\alpha}&=\bQ\bD_{\alpha}\bQ^{\rm T},
\end{aligned}
\right\}\ \text{for all rotations}\ \bQ.
\eeqn
Using standard representation theorems for  isotropic  functions, see e.g.~\cite{GFA}, equations \eqref{isoCL0} imply that
\beqn\label{isoCL}
\begin{aligned}
\bbE_M(n_S)\be(\bu)&=2\mu(n_S)\be(\bu)+\lambda(n_S)(\text{div}\,\bu)\bone,\\
\bZ_\alpha(n_S)&=z_\alpha(n_S)\bone,\\
\bD_{\alpha}&=D_{\alpha}\bone.
\end{aligned}
\eeqn
While it is known that polymers can diffuse \cite{HSP}, the diffusion coefficient  of calcium-pectin cross-links is  much lower  than the diffusion coefficients  of the other species and, thus, we assume  first that calcium-pectin cross-links do not diffuse, i.e.\ $D_b=0$. 
Unlike the matrix, the microfibrils have different elastic properties in different directions, see e.g.~\cite{DMKM}.
For a plant cell wall, the amount of  calcium-pectin cross-links plays a decisive role in determining the  elastic properties of the wall matrix, \cite{Somerville,WHH}. Thus, we assume that $\bbE_M$ or, equivalently, $\mu$ and  $\lambda$,  depends only on $n_b$.

We  consider the following four interactions between the species in the matrix: 
\begin{enumerate}
\item The enzyme PME interacts with methylestrified pectin to form demethylestrified pectin. 
\item Demethylestrified pectin decays.
\item Demethylestrified pectin and calcium ions bind together to form calcium-pectin cross-links.
\item Under the presence of  stress, calcium-pectin cross-links break to yield demethylestrified pectin and calcium ions.
\end{enumerate}
For a detailed discussion of Interactions 1--4, see e.g.~\cite{PB,WG,ZMSR}.

We begin by discussing the reaction term $r_d$, which is decomposed into the sum of three terms:
\begin{equation*}
r_d=r_{eE}+r_{dd}+r_{fb}, 
\end{equation*}
where $r_{eE}$ is the rate of change of the density of demethylestrified pectin acid groups, $n_d$,  associated with Interaction 1, $r_{dd}$ is the rate of decay of $n_d$ mentioned in Interaction 2, and $r_{fb}$ is the rate of change of $n_d$ associated with the formation and breakage of calcium-pectin cross-links specified in Interactions 3 and 4.

From Interaction 1, we have
\begin{equation*}
r_{eE}=-r_e.
\end{equation*}
 We assume that the binding of PME  to and  dissociation  from  a pectin acid group are very fast, and  that the enzyme PME is not used up during the demethyl-esterification process so that $r_E=0$.  From Interactions 3 and 4 it follows that
\begin{equation*}
r_b=-r_c=-\frac 1 2 r_{fb}.
\end{equation*}
 The factor of a half  in front of $r_{fb}$ reflects the fact  that  two demethylestirified galacturonic acids are needed to form a calcium-pectin cross-link. 
We assume  that
\begin{align*}
r_{eE}&= R_{eE}(n_e,n_E),\\
r_{dd}&=-R_d \, n_d,
\end{align*}
where $R_{eE}$ defines the demethyl-esterification reaction between methylestrified galacturonic acid groups and PME and $R_d>0$ is a decay constant of the demethylestrified pectin.  Interactions between demethylestirified pectin and calcium ions increase the number of cross-links, while  stress can break the cross-links. Thus
\begin{equation}\label{Ninre}
\begin{aligned}
r_{b}&=R_{dc}(n_d,n_c)-R_b (n_b) \, {\cN}_\delta(\be(\bu)),
\end{aligned}
\end{equation}
where $R_{dc}$ models the formation of   cross-links through  the interactions between demethyl\-esterified pectin and calcium ions,    and  ${\cN}_\delta(\be(\bu))$ is  defined
as 
  \begin{equation}\label{def_N_12}
\mathcal N_\delta (\be(\bu))(t,x) =  \Big(\dashint_{B_\delta(x)\cap \Omega}\text{tr }  \bbE(n_b) \,  \be(\bu)(t, \tilde x) \, d\tilde x \Big)^{+} \qquad \text{for all } \; x\in \overline\Omega \text{ and } t \in (0,T).
\end{equation}
Having $r_{b}$ depend on the positive part of the local average of the stress   does not follow from \eqref{conlaw2}, but it is consistent with the isotropy assumption.  The reason for the choice \eqref{Ninre} is based on the idea that stretching, rather than compressing,  of the cross-links will cause them to break.

Possible choices for the functions $R_{eE}$, $R_{dc}$, and $R_b$ are
$$
\begin{aligned}
R_{eE}(n_e,n_E)=k_{eE}n_en_E, && \quad 
R_{dc}(n_d,n_c)= \frac{k_{dc,1}n_c}{k_{dc,2}+n_c} n_d, && \quad 
R_{b}(n_b)=k_b n_b,
\end{aligned}
$$
where $k_{eE}$, $k_{dc,1}$, $k_{dc,2}$, and $k_b$ are positive constants.  Due to the high calcium concentration in plant cell walls,  we assume saturation kinetics for the density of calcium ions in the reaction term $R_{dc}$. \\

\noindent\textbf{Remark.}
 The constitutive laws, considered here,  are consistent with the Second Law of Thermodynamics in that, in the elastic case, Maxwell's relation
\begin{equation*}
\frac{\partial\bT_R}{\partial n_\alpha}=\frac{\partial\mu_{\alpha}}{\partial \be(\bu)}
\end{equation*}
 holds, where $\mu_\alpha$ is the chemical potential for species $\alpha$, which depends on $n_S$, and $\be(\bu)$, see~\cite{GFA}.  The chemical potential is related to the flux $\bj_\alpha$ through the relation
\begin{equation*}
\bj_\alpha=-\bM_\alpha\nabla\mu_\alpha.
\end{equation*}
 To obtain the flux used in this section, set
\begin{align*}
\mu_\alpha&=
\begin{cases}
n_\alpha&  \qquad \alpha\not=b,\\
\dfrac 12\be(\bu)\cdot\dfrac{\partial\bbE}{\partial n_b}(n_b)\be(\bu)& \qquad\alpha=b,
\end{cases}\\
\bM_\alpha&=
\begin{cases}
D_\alpha\textbf{1}& \qquad \alpha\not=b,\\
0&\qquad \alpha=b.
\end{cases}
\end{align*}

The environment can effect the cell wall in two different ways: through external influences and boundary conditions. The effects of the supply of species $h_\alpha$ and external body forces $\bb$, including inertial terms, are neglected so
\begin{equation*}
h_\alpha=0\ \ \text{in}\ \Omega_M\qquad\text{and}\qquad \bb=\textbf{0}\ \ \text{in}\ \Omega.
\end{equation*}

The boundary $\partial\Omega$ of $\Omega$ is decomposed into four disjoint surfaces: $\Gamma_\cI$, $\Gamma_\cE$,  $\Gamma_{\cU}$, and $\partial\Omega\setminus (\Gamma_\cI\cup\Gamma_\cE\cup \Gamma_\cU)$, where  $\Gamma_\cI$ is the part of $\partial\Omega$ in contact with the interior of the cell and $\Gamma_\cE$ is the part of $\partial\Omega$ in contact with the middle lamella.  Let  $\bnu$ denote the exterior unit-normal to whatever surface is under discussion.  On $\Gamma= \partial \Omega_F \setminus \partial \Omega$, $\bnu$ points away from $\Omega_M$.

  PME, produced in the Golgi apparatus of a plant cell, is deposited into the cell wall  and diffuses through the cell wall into the middle lamella.  PME can also diffuse back into the cell to degrade. 
Thus, we assume that the enzyme PME can enter or leave the cell wall through $\Gamma_\cI$ but can only leave the wall through $\Gamma_\cE$.  To account for the mechanisms  controlling the amount of PME in a cell wall  \cite{WHH},  we  assume that the  inflow  of PME  into  the cell wall  depends on the total amount of methylestrified pectin within the  wall, which leads to the boundary fluxes
\begin{equation*}
\begin{aligned}
\bj_E\cdot\bnu&=-J_E\Big(\int_{\Omega_M} n_e\, dx\Big) +\zeta_En_E&&\qquad \text{on}\ \Gamma_\cI,\\
\bj_E\cdot\bnu&=\gamma_En_E &&\qquad \text{on}\ \Gamma_\cE,
\end{aligned}
\end{equation*}
where $\zeta_E$ and $\gamma_E$ are  non-negative constants. 

Methylesterified pectin is produced by the cell and then transported into the cell wall through $\Gamma_\cI$, e.g.~\cite{WG}.  To account for mechanisms  controlling the amount of pectin in the cell wall, we assume that the  inflow of new methylestrified pectin  decreases with an increasing   amount of methylestrified pectin in the  wall.   Methylestrified pectin can leave the wall   through $\Gamma_\cE$ and enter the middle lamella. Thus,
\beqn\label{IBC2}
\begin{aligned}
\bj_{e}\cdot\bnu&=-J_e\Big(\int_{\Omega_M}n_e\, dx\Big) &&\qquad \text{on}\ \Gamma_\cI,\\
\bj_e\cdot\bnu&=\gamma_en_e &&\qquad \text{on}\ \Gamma_\cE,
\end{aligned}
\eeqn
where $\gamma_e$ is a non-negative constant. We assume an outflow of demethylesterified pectin from the cell wall into the middle lamella:
\beqn\label{IBC3}
\begin{aligned}
\bj_{d}\cdot\bnu=0  \qquad \text{on}\ \Gamma_\cI,\qquad \qquad
\bj_d\cdot\bnu=\gamma_dn_d   \qquad \text{on}\ \Gamma_\cE.
\end{aligned}
\eeqn

Calcium ions may enter or leave the cell wall through both $\Gamma_\cI$ and $\Gamma_\cE$, but the flow of calcium through $\Gamma_\cI$ is controlled by stretch activated calcium channels  in the plasma membrane, see e.g.\ \cite{DR, White}.   Thus, the flow of calcium through $\Gamma_\cI$ is assumed to depend on the local average of the  stress and on the density of calcium, so that
\beqn\label{IBC4}
\begin{aligned}
\bj_c\cdot\bnu&= -J_{c, \cE} (n_c)  &&   \text{on } \Gamma_\cE, \\
\bj_c\cdot\bnu&=-{\cN}_\delta(\be(\bu))\,  J_{c,\cI}(n_c)&&  \text{on}\ \Gamma_\cI,
\end{aligned}
\eeqn
where
$$J_{c, \cI}(n_c)=\gamma_{c,1} - \gamma_{c,2} n_c, \qquad \qquad J_{c, \cE}(n_c)=\zeta_{c,1} - \zeta_{c,2} n_c,$$
with non-negative constants $\gamma_{c,i}$ and $\zeta_{c,i}$, where $i=1,2$,  and $ \cN_\delta(\be(\bu))$  is given  by \eqref{def_N_12}.  
Similar to    \eqref{Ninre}, we assume that the right-hand side of \eqref{IBC4}$_2$ depends on the positive part of the local average of the stress. 

The traction boundary conditions
\begin{equation*}
\begin{aligned}
\bT_R\bnu&=-p_\cI\bnu &&\qquad\text{on}\ \Gamma_\cI,\qquad \qquad   && 
\bT_R\bnu=\bff &&\qquad\text{on}\ \Gamma_\cE\cup \Gamma_\cU,
\end{aligned}
\end{equation*}
come from the constant, positive turgor pressure $p_\cI$  within the cell and the traction force $\bff$, caused  by surrounding cells.
We consider zero-flux boundary conditions on the surface of  the microfibrils and on $\Gamma_\cU$:
\begin{equation*}
\bj_\alpha\cdot\bnu =0 \qquad \text{on}\ \Gamma \qquad \quad \text{ and }\quad  \qquad  \bj_\alpha\cdot\bnu =0 \qquad \text{on}\ \Gamma_\cU, \qquad \alpha = e,E,d,c.
\end{equation*}
On $\partial\Omega\setminus (\Gamma_\cI\cup\Gamma_\cE\cup \Gamma_\cU)$ periodic boundary conditions for the densities and displacement are imposed.

Possible choices for the functions that determine the boundary conditions are
\begin{equation*}
\begin{aligned}
& J_E\Big(\int_{\Omega_M} n_e\, dx\Big)= \beta_E \int_{\Omega_M} n_e\, dx ,\qquad J_e\Big(\int_{\Omega_M} n_e\, dx\Big)=\frac{\beta_e }{1+\zeta_e \int_{\Omega_M} n_e\ dx },\\
& \bff=p_{\cE}\, \bnu  \quad \text{ on } \Gamma_\cE,  \qquad \bff=p_{\cU}\, \bnu  \quad \text{ on } \Gamma_\cU, 
\end{aligned}
\end{equation*}
where $\beta_E$,  $\beta_e$, $\zeta_e$, $p_\cE$, and $p_\cU$ are positive constants. 

The difference between Model I and  Model II is that in Model II we assume that the calcium-pectin cross-links can diffuse, i.e. $D_b >0$.  In this situation we  assume  that the reaction term associated with the formation and destruction  of cross-links depends on the point-wise values of the displacement gradient rather than the local average:
\begin{equation}\label{Ninre_II}
\begin{aligned}
r_{b}&=R_{dc}(n_d,n_c)-R_b (n_b) \, Q(n_b, \be(\bu)),
\end{aligned}
\end{equation}
where  a possible choice for $Q$ is $ Q(n_b, \be(\bu))=\big({\rm tr }\,  \bbE(n_b) \,  \be(\bu)\big)^+$. Considering  the diffusion of calcium-pectin cross-links corresponds  to the situation where  the calcium-pectin network is less connected and the mechanical stress in the cell wall  have a point-wise impact on  chemical processes.

 \section{Summary}\label{conclusions}
In this paper we developed a  mathematical model for plant cell wall biomechanics which  explicitly considers the microscopic structure of  the cell wall  and the biochemical processes that take place within the wall matrix.  The microscopic model defined on the scale of the cell wall's structural elements  describes the interconnections between   the calcium-pectin cross-links dynamics and   the changes in the  mechanical properties of the cell wall. 
We consider both  a non-local effect of  strain or  stress  on the calcium-pectin cross-link dynamics as well as a point-wise dependence of chemical reactions on  mechanical forces.
Applying homogenization techniques we rigorously derive macroscopic models for plant cell wall biomechanics. 
We also show that since the cell wall matrix is isotropic, the macroscopic elasticity tensor is a linear function of the Young's modulus of the wall matrix.  Then  assuming that only the Young's modulus 
of the wall matrix depends on the density of  calcium-pectin cross-links  we compute numerically the effective macroscopic elastic properties of the plant  cell wall  as  a function of the density of calcium-pectin cross-links.     In the numerical simulations, the cell wall  microfibrils are assumed to be transversal isotropic.
 The numerical analysis of the  full macroscopic model will be the  subject of   future research. \\
 
 \textit{Acknowledgements.} The authors would like to thank Dr.\ Sebastian Wolf for fruitful discussions on the molecular biology of plant cell walls.

\end{document}